\documentclass[12pt]{amsart}

\usepackage{mathtools}
\usepackage{amssymb}
\usepackage{fullpage}
\usepackage{hyperref}

\newtheorem{theorem}{Theorem}[section]
\newtheorem{lemma}[theorem]{Lemma}
\newtheorem{corollary}[theorem]{Corollary}
\newtheorem{proposition}[theorem]{Proposition}

\theoremstyle{plain}
\newtheorem{definition}[theorem]{Definition}

\theoremstyle{remark}
\newtheorem{remark}[theorem]{Remark}

\numberwithin{equation}{section}

\newcommand{\C}{\mathbb{C}}
\newcommand{\R}{\mathbb{R}}
\newcommand{\E}{\mathbb{E}}
\newcommand{\Z}{\mathbb{Z}}

\newcommand{\cQ}{\mathcal{Q}}
\newcommand{\cT}{\mathcal{T}}
\newcommand{\cP}{\mathcal{P}}
\newcommand{\eps}{\varepsilon}
\renewcommand{\d}{\mathrm d}
\newcommand{\sgn}{\operatorname{sgn}}
\renewcommand{\pmod}[1]{\,\,(\mathrm{mod}\,{#1})}

\title{Random linear configurations in dense sets and primes}

\author{Lasse Grimmelt}
\address{Department of Pure Mathematics and Mathematical Statistics, University of Cambridge, Cambridge CB3 0WB, UK}
\email{lpg31@cam.ac.uk}

\author{Joni Ter\"av\"ainen}
\address{Department of Pure Mathematics and Mathematical Statistics, University of Cambridge, Cambridge CB3 0WB, UK}
\email{joni.p.teravainen@gmail.com}

\begin{document}

\begin{abstract}
We prove that every polylogarithmically dense subset of $[N]$ contains a nontrivial configuration $x+b_1m,\ldots,x+b_km$ for almost all choices of the coefficient vector $(b_1,\ldots, b_k)$ in a wide range of scales.  We prove the same statement for polylogarithmically relatively dense subsets of the primes, in a shorter range of scales.  The main ingredients are a new quantitative generalised von Neumann theorem, degree lowering to the $U^{1+}$ norm, and densification arguments that transfer the result to the primes.
\end{abstract}

\pagestyle{headings}
\markboth{L. Grimmelt and J. Ter\"av\"ainen}{Random linear configurations in dense sets and primes}

\maketitle

\section{Introduction}

Let $A\subset[N]=\{1,2,\ldots,\lfloor N\rfloor\}$ and let
$\mathbf{b}=(b_1,\ldots,b_k)\in\mathbb{Z}^k$ have entries of modulus at most
$B$.  We study when $A$ contains a configuration of the form
\begin{align}\label{eq:nonlin-config}
    x+b_1m,\,x+b_2m,\ \ldots,\ x+b_km
    \qquad\text{with } x\in\mathbb{Z},\ m\in[H],
\end{align}
where $H=\lfloor N/B\rfloor$.  These are translation invariant linear
configurations in two variables.  Indeed, any finite family of integer linear forms in two variables
whose image is invariant under simultaneous integer translations of all
coordinates can be written after a unimodular change of variables as
$(x+b_i m)_{i=1}^k$ for some $k\in\mathbb N$ and some integers $b_i$.  The
case $b_i=i$ gives arithmetic progressions of length $k$. 

For arithmetic progressions in dense sets, the best quantitative bounds have been intensively studied in connection with Szemer\'edi's theorem. If  $\delta_{k}(N)$ is an upper bound for the minimal density of a subset of $[N]$ guaranteeing an arithmetic progression of length $k$, then one may take $\delta_3(N)=\exp(-c(\log N)^{c'})$ with absolute constants $c,c'>0$ by the work of Kelley--Meka and Bloom--Sisask~\cite{kelley-meka,bloom-sisask}, $\delta_4(N)=(\log N)^{-c}$ by the work of Green and Tao~\cite{green-tao-r4}, and $\delta_k(N)=\exp(-(\log\log N)^{c})$ with $c=c(k)>0$ for $k\geq 5$ by the work of Leng, Sah and Sawhney~\cite{leng-sah-sawhney}.  Much less seems to be known for a fixed coefficient tuple $\mathbf b$ whose entries are large, for example larger than powers of $\log N$.

Our first main result shows that one can obtain stronger results for random translation invariant configurations.  If $\mathbf b$ in~\eqref{eq:nonlin-config} is chosen to be of size $B$, then for almost every tuple a polylogarithmic density threshold already forces a nontrivial configuration, for a wide range of $B$.

\begin{theorem}[Random configurations in dense sets of integers]\label{thm:main}
Let $k\geq 3$ and assume that $N$ is sufficiently large in terms of $k$.
There exist constants $c_k,c_k'>0$ such that the following holds. Let
$$
    (\log N)^{1/c_k} \leq B \leq N\exp\bigl(-(\log N)^{c_k'}\bigr),
$$
and let $A\subseteq[N]$ satisfy
\begin{align}\label{eq:main-density-threshold}
    |A| \geq \frac{N}{(\log N)^{c_k}}.
\end{align}
Then, for all but an $O_k((\log N)^{-c_k})$ proportion of tuples
$\mathbf{b}\in((B/2,B]\cap\mathbb Z)^k$, the set
$A$ contains a configuration
\begin{align}\label{eq:main-existence}
    x+b_1m,\,x+b_2m,\ \ldots,\ x+b_{k}m
    \qquad\text{with } x\in\mathbb{Z},\ m\in[N/B].
\end{align}
\end{theorem}

We also prove a relative version in the primes. This may be compared with the Green--Tao theorem~\cite{GT} and its quantifications. Indeed, if  $\delta_{k,\mathbb{P}}(N)$ is an upper bound for the minimal relative density of a subset of the primes in $[N]$ guaranteeing an arithmetic progression of length $k$, then  one can take $\delta_{4,\mathbb{P}}(N)=(\log \log N)^{-c}$ for $k=4$ and $\delta_{k,\mathbb{P}}(N)=\exp(-(\log \log \log N)^c)$ for $k\geq 5$, with $c=c(k)>0$, by~\cite{teravainen-wang-sparse}.  For random translation invariant configurations we again obtain a polylogarithmic relative density threshold, although in a shorter range of $B$.

The shorter coefficient range in the prime setting comes ultimately from the $W$-trick and the uniform
linear forms estimate needed for the sieve majorant.  It may however be possible to enlarge the range of $B$ with additional work, see
Remark~\ref{rmk:prime-majorant-range}.

\begin{theorem}[Random configurations in dense subsets of the primes]
\label{thm:prime-main}
Let $k\geq 3$ and assume that $N$ is sufficiently large in terms of $k$.
There exists a constant $c_k>0$ such that the following holds.  Let
\begin{align}\label{eq:prime-main-B-range}
    (\log N)^{1/c_k} \leq B \leq \exp\bigl((\log N)^{c_k}\bigr),
\end{align}
and let $A\subseteq\{p\leq N:p\text{ prime}\}$ have relative density at least
$(\log N)^{-c_k}$ in the primes, in the sense that
\begin{align}\label{eq:prime-main-density}
    |A| \geq \frac{\pi(N)}{(\log N)^{c_k}}.
\end{align}
Then, for all but an $O_k((\log N)^{-c_k})$ proportion of tuples
$\mathbf b\in((B/2,B]\cap\mathbb Z)^k$, the set $A$ contains a
configuration
\begin{align}\label{eq:prime-main-config}
    x+b_1m,\,x+b_2m,\ \ldots,\ x+b_km
    \qquad\text{with }x\in\mathbb Z,\ m\in[N/B].
\end{align}
\end{theorem}

\subsection{Strategy}

We first introduce the counting operators and norm that play a central role in the proofs of the main theorems.  

Let $\mathbf b\in\mathbb Z^k$ and let $H$ be a positive integer. For finitely
supported functions $f_1,\ldots,f_k\colon\mathbb Z\to\mathbb C$, define
\begin{align}\label{eq:r-def}
    r_H(\mathbf b;f_1,\ldots,f_k)
    \coloneqq
    \sum_{x\in\mathbb Z}\sum_{m\in[H]}
    \prod_{i=1}^k f_i(x+b_im).
\end{align}
For a $1$-bounded coefficient function
$\lambda\colon\mathbb Z^k\to\mathbb C$, set
\begin{align}\label{eq:R-def}
    R_H(\lambda;f_1,\ldots,f_k)
    \coloneqq
    \sum_{\mathbf b\in\mathbb Z^k}
    \lambda(\mathbf b)\,
    r_H(\mathbf b;f_1,\ldots,f_k).
\end{align}
We aim to connect these counting operators to the 
$U^{1+}$ norm, introduced explicitly in~\cite{taoblog}, defined by
\begin{align}\label{eq:U1-plus-def}
    \|f\|_{U^{1+}[N]}
    \coloneqq
    \frac{1}{|[N]|}
    \max_{\substack{P\subseteq[N]\\
                    P\text{ arithmetic progression}}}
    \left|\sum_{n\in P}f(n)\right|.
\end{align}

\textbf{Strategy for Theorem~\ref{thm:main}.}

Let $A\subset [N]$ satisfy $|A|=\delta N$ with $\delta\geq (\log N)^{-c_k}$ for a small constant $c_k>0$. If Theorem~\ref{thm:main} failed for $A$, then a standard splitting to balanced functions would show that there exists a $1$-bounded function $\lambda$ and functions $g_1,\ldots, g_k\in \{\delta 1_{[N]},1_{A}-\delta1_{[N]}\}$, not all equal to $\delta 1_{[N]}$, such that
\begin{align*}
|R_H(\lambda; g_1,\ldots, g_k)|\gg_k \delta^k B^k\frac{N^2}{B}.     
\end{align*}
A key technical result in this paper is that there is a polynomial inverse theorem for the counting operator $R_H$ in terms of the $U^{1+}$ norm.

\begin{theorem}[$U^{1+}$ control of the counting operator]\label{thm:1+}
Let $N\geq1$ and $C_0\geq1$.  There are constants $c=c(k,C_0)>0$ and
$C=C(k)\geq1$ such that the following holds.  Let $0<\delta<c$, let $H$
be a positive integer, set $B=N/H$, and suppose that
\[
    \delta^{-C}\leq B\leq N\delta^C.
\]
Let $\lambda\colon\mathbb Z^k\to\mathbb C$ be $1$-bounded and supported on
$([-C_0B,C_0B]\cap\mathbb Z)^k$, and let
$f_1,\ldots,f_k\colon\mathbb Z\to\mathbb C$ be $1$-bounded and supported on
$[N]$.  If
\begin{align}\label{eq:R-large}
    |R_H(\lambda;f_1,\ldots,f_k)|
    \geq
    \delta B^k\frac{N^2}{B},
\end{align}
then
\begin{align}\label{eq:gowers-control-conclusion}
    \min_{1\leq i\leq k}\|f_i\|_{U^{1+}[N]}
    \geq c\delta^C.
\end{align}
\end{theorem}

Once Theorem~\ref{thm:1+} is proved, we may apply it to the functions $g_i$ above to find a long subproression $P\subset [N]$ on which $A$ has polynomially increased density, that is, $|A\cap P|/|P|\geq \delta+c_k\delta^{O_k(1)}$. We may then pass to this progression by translation invariance of the pattern under consideration and perform a density increment argument to obtain a contradiction with the assumption $\delta\geq (\log N)^{-c_k}$.

Our proof of Theorem~\ref{thm:1+} builds on the quantitative concatenation and
degree lowering developed by Peluse--Prendiville in their work on the nonlinear
Roth theorem~\cite{PP}.  The new feature here is that the coefficient vector
ranges over a growing box: exploiting this additional averaging, we recover
estimates uniform in $B$ and obtain polynomial control by the $U^{1+}$ norm, in
the form needed for random configurations with large coefficients.

For a fixed coefficient tuple $\mathbf b$, a standard generalised von Neumann
argument controls the corresponding configuration \eqref{eq:nonlin-config} by the $U^k$ Gowers norm.  When the coefficients have size $B$, however, repeated
Cauchy--Schwarz produces differences whose increments involve the large
coefficient differences $b_i-b_j$.  A direct comparison with the
$U^k[N]$ norm consequently loses uniformity as $B$ grows.

The averaging over $\mathbf b$ compensates for this loss.  Iterated
Cauchy--Schwarz first produces differences of the form $h b_i'$, where
$b_i'$ ranges over an interval of length comparable with
$B$ and $h$ over an interval of length comparable with
$H$.  Quantitative concatenation then converts this to an
ordinary Gowers norm.  We give the required form as
Lemma~\ref{lem:pp_5.3} and iterate it in
Corollary~\ref{cor:iterated-concatenation}.  Combined with the iterated
Cauchy--Schwarz estimate, this proves Theorem~\ref{thm:1+} with the $U^k[N]$ norm in place of the $U^{1+}[N]$ norm.

The next task is to lower the degree of this control. By pigeonholing and making a linear change of variables,~\eqref{eq:R-large} implies
\begin{align*}
\left|\sum_{x}f_1(x)D(x)\right|\gg_{k,C_0} \delta B^{k-2}N^2,    
\end{align*}
where $D$ is the dual function
\begin{align*}
 D(x)
 =1_{[N]}(x)
 \sum_{\mathbf b\in\mathbb Z^{k-1}}\mu(\mathbf b)
 \sum_{m\in[H]}
 \prod_{i=2}^k f_i(x+b_im),
\end{align*}
with $\mu$ a $1$-bounded function constructed from $\lambda$. Cauchy--Schwarz and the preceding \(U^k\) control then give
\begin{align*}
    \|D\|_{U^k[N]}
    \gg_{k,C_0}\delta^{O_k(1)}B^{k-2}N.
\end{align*}
We then perform a degree lowering argument showing that a
large \(U^s[N]\) norm of \(D\) forces a large \(U^{s-1}[N]\) norm.

To illustrate the argument, suppose that \(k=3\).  The recursion for the Gowers norms and
the \(U^2\) inverse theorem show that, for many \(h\), there is a frequency
\(\alpha(h)\in\mathbb T\) such that
\begin{align}\label{eq:intro-frequency-correlation}
    \left|
        \sum_x\Delta_hD(x)e(\alpha(h)x)
    \right|
    \gg_{C_0}\delta^{O(1)}B^2N^3.
\end{align}
The difficulty is that \(\alpha(h)\) may initially depend arbitrarily on
\(h\). If we can show that $\alpha(h)$ is constant for many $h$, the control complexity is lowered.

Write
\begin{align*}
    \psi_{\mathbf b,m}(x)
    =
    1_{[N]}(x)f_2(x+b_2m)f_3(x+b_3m),
\end{align*}
so that $D=\sum_{\mathbf b\in \mathbb{Z}^2,m\in [H]}\mu(\mathbf b)\psi_{\mathbf b,m}$.
The dual--difference interchange lemma,
Lemma~\ref{lem:DDI}, transfers the
difference from $D$ to these individual summands.  Applied
to~\eqref{eq:intro-frequency-correlation}, it gives a set
$\mathcal H\subseteq[-N,N]\cap\mathbb Z$ of size
$\gg_{C_0}\delta^{O(1)}N$ such that
\begin{align*}
    \sum_{h,h'\in\mathcal H}
    \left|
        \sum_x\sum_{\mathbf b\in \mathbb{Z}^2,m\in [H]}
        \Delta_{h'-h}\psi_{\mathbf b,m}(x)
        e\bigl(\partial\alpha(h,h')x\bigr)
    \right|
    \gg_{C_0}\delta^{O(1)}BN^4,
\end{align*}
where
\begin{align*}
    \partial\alpha(h,h')
    \coloneqq\alpha(h)-\alpha(h').
\end{align*}

Let $u=h'-h$ and $E_u=[N]\cap([N]-u)$.  After the substitution
$x\mapsto x-b_2m$ and the replacement of $b_3-b_2$ by $b_3$, the inner sum
in the preceding display becomes
\begin{align*}
    \sum_{b_2,b_3\in \mathbb{Z}}\sum_{m\in[H]}\sum_x
    1_{E_u}(x-b_2m)\Delta_u f_2(x)\Delta_u f_3(x+b_3m)
    e\bigl(\partial\alpha(h,h')x\bigr)
    e\bigl(-\partial\alpha(h,h')b_2m\bigr),
\end{align*}
where the transformed coefficient variables still lie in a box of side
$O_{C_0}(B)$.  Thus the dependence on $b_2$ and $m$ includes the 
phase
\begin{align*}
    e\bigl(-\partial\alpha(h,h')b_2m\bigr),
\end{align*}
with the $b_2$ and $m$ variables being unweighted after splitting into short intervals.
A classical Type I argument shows that largeness forces
\(\partial\alpha(h,h')\) to lie on a major arc: for some natural number
\(q\ll_{C_0}\delta^{-O(1)}\), we have
\begin{align*}
    \|q\partial\alpha(h,h')\|
    \ll_{C_0}\frac{\delta^{-O(1)}}{N}.
\end{align*}
At the inverse theorem step we can choose the frequencies $\alpha(h)$ from a finite grid
$Q^{-1}\mathbb Z/\mathbb Z$, where $Q\asymp N\delta^{-O(1)}$.
Hence, pigeonholing the major arc approximation 
shows that \(\partial\alpha(h,h')\) is constant for many pairs.  It follows
that \(\alpha(h)\) is constant on a large set, which is precisely the
low rank condition needed to deduce that \(\|D\|_{U^2[N]}\) is large.

\textbf{Strategy for Theorem~\ref{thm:prime-main}.}

A useful way of transferring statements about dense functions to unbounded ones is densification, introduced by Conlon, Fox and Zhao~\cite{Conlon-Fox-Zhao}.
For the prime setting, we adapt the quantitative densification strategy of~\cite{teravainen-wang-sparse} (see also~\cite{tao-teravainen-gowers-vonmangoldt}) for the quasipolynomial inverse theorem for the Gowers norms. The final conclusion is that Theorem~\ref{thm:1+} continues to hold for unbounded functions $f_i$ that satisfy a technical condition, the $(K, \delta^{-K},\delta^K)$ linear forms condition at scales $(N, B, H)$ (with $K$ large in terms of $k$); see Theorem~\ref{thm:relative-u1plus} for the precise statement. The relevant linear forms condition is then verified in Section~\ref{sec:prime-majorant} for the von Mangoldt function by extending the work of Green and Tao~\cite{green-tao-linear-equations} on correlations of GPY sieve weights. 

For the proof of the unbounded $U^{1+}$ inverse theorem, we combine the iterated Cauchy--Schwarz estimate of Lemma~\ref{lem:fibrewise-CS} with Lemma~\ref{lem:terminal-LFC}, which uses the linear forms condition for the majorant, to show that the dual function $D$ is ``essentially bounded'' in a precise sense.  This combination yields a generalised von Neumann estimate that may also be of independent interest: it holds uniformly over the full permitted ranges of all the parameters, with all parameter dependencies made quantitative.  We may then apply Theorem~\ref{thm:1+} to replace $f_1$ by the indicator of an arithmetic progression while preserving~\eqref{eq:R-large} up to polynomial losses.  Iterating this procedure replaces all but one of the functions in $R_H$ by indicators of arithmetic progressions, at which point the remaining function $f_k$ must have large $U^{1+}$ norm.  By symmetry, the same conclusion holds for every $f_i$.

After this, the proof of Theorem~\ref{thm:prime-main} uses a similar density increment strategy as Theorem~\ref{thm:main}.

%The employed techniques, are based on relatively recent developments and besides proving the main configuration theorems, one of our aims has been exposition.   
%This is particularly true of degree lowering.  For this
%reason, we include detailed, and occasionally modified, proofs of some results
%essentially contained in~\cite{PP}. 

\subsection{Organisation}

In Section~\ref{sec:gowers-control}, we prove that \(R_H\) admits \(U^k\)
control.  The main ingredients are a Cauchy--Schwarz process, which we
formulate in a general setting, and concatenation.  In
Section~\ref{sec:degree-lowering}, we develop the degree lowering argument.
We apply it in Section~\ref{sec:proof-1+} to reduce the \(U^k\) control to
\(U^{1+}\) control and thereby prove Theorem~\ref{thm:1+}.

In Section~\ref{sec:relative-u1plus}, we transfer this conclusion to functions
dominated by a pseudorandom majorant, using a densification argument.  In
Section~\ref{sec:prime-majorant}, we construct a suitable majorant for the
function obtained by applying the \(W\)-trick to the von Mangoldt function
and verify the required linear forms condition.  Finally,
Section~\ref{sec:density-increment} combines the bounded
and relative $U^{1+}$ inverse theorems with a density increment argument to prove
Theorems~\ref{thm:main} and~\ref{thm:prime-main}.

\subsection{Acknowledgements}
The authors were supported by European Union's Horizon Europe research and innovation programme under ERC grant agreement no. 101162746.

\section{Notation}

Throughout, $k\geq 2$ is a fixed integer and all implied constants are allowed to depend
on $k$. We write $X\ll Y$, $Y\gg X$ or $X=O(Y)$ to mean $|X|\leq CY$ for some
constant $C$, and $X\asymp Y$ to mean $X\ll Y\ll X$; a subscript such as
$\ll_s$ or $O_s(\cdot)$ records an additional dependence of the constant on a
parameter $s$. 

For a real number $N\geq 1$ we write $[N]=\{1,2,\ldots,\lfloor N\rfloor\}$, and
for a set $S$ we write $1_S$ for its indicator function. For a finite nonempty
set $S$ and $g\colon S\to\mathbb{C}$ we write
$\mathbb{E}_{n\in S}\,g(n)=\frac{1}{|S|}\sum_{n\in S}g(n)$. We use $\sum_{n}f(n)$ to mean $\sum_{n\in \mathbb{Z}}f(n)$.

We set
$e(\theta)=e^{2\pi i\theta}$ and write $\|\theta\|=\min_{n\in\mathbb{Z}}|\theta-n|$
for the distance from $\theta\in\mathbb{R}$ to the nearest integer.  We write
$\mathbb T=\mathbb R/\mathbb Z$ and for $z\in\mathbb C$ set
\[
    \sgn z=
    \begin{cases}
        z/|z|,&z\ne0,\\
        0,&z=0.
    \end{cases}
\]
We also use the notation $x_{+}=\max(x,0)$ for real $x$. We write
$\Lambda$ for the von Mangoldt function, $\mu$ for the M\"obius function,
$\varphi$ for Euler's totient function, $\pi(N)$ for the number of primes at
most $N$, and $\tau$ for the divisor function.

For a finitely supported function $f\colon \mathbb{Z}\to \mathbb{C}$ and a
positive integer $s$,
we recall the (unnormalised) Gowers norm
\begin{align*}
\|f\|_{U^s(\mathbb{Z})}
=\left(\sum_{x,h_1,\ldots, h_s\in \mathbb{Z}}
\prod_{\omega\in \{0,1\}^s}\mathcal{C}^{|\omega|}f(x+\omega\cdot\mathbf{h})\right)^{1/2^s},
\end{align*}
where $|\omega|=\omega_1+\cdots+\omega_s$ and $\mathcal{C}$ denotes complex
conjugation, and where
$\omega\cdot\mathbf h=\sum_{j=1}^s\omega_jh_j$. In particular
\begin{align*}
    \|1_{[N]}\|_{U^s(\mathbb{Z})}^{2^s}\asymp_s N^{s+1}.
\end{align*}
For $N\geq 1$ and $f\colon \mathbb{Z}\to\mathbb{C}$ we define the normalised
interval Gowers norm
\begin{align}\label{eq:Gowersdef}
\|f\|_{U^s[N]}=\frac{\|f\,1_{[N]}\|_{U^s(\mathbb{Z})}}{\|1_{[N]}\|_{U^s(\mathbb{Z})}}.
\end{align}

For $h\in\mathbb{Z}$ we write $\Delta_h f(x)=f(x)\overline{f(x+h)}$ for the
multiplicative difference operator, and for
$\mathbf{h}=(h_1,\ldots,h_s)\in\mathbb{Z}^s$ we set
$$
\Delta_{\mathbf{h}}f=\Delta_{h_1}\cdots\Delta_{h_s}f.
$$
We also sometimes write this as $\Delta_{(h_i)_{i\leq s}}f$. When $s=0$, we interpret this quantity as $f$.
With this notation the
Gowers norms satisfy the recursive identity
\begin{align}\label{eq:gowers-recursion}
 \|f\|_{U^{s+1}(\mathbb{Z})}^{2^{s+1}}
 =\sum_{h\in \mathbb{Z}}\|\Delta_h f\|_{U^{s}(\mathbb{Z})}^{2^s}
 \quad\text{for }s\geq 1,
\end{align}
where $\|f\|_{U^1(\mathbb{Z})}=\bigl|\sum_{x\in\mathbb{Z}}f(x)\bigr|$.
We also use the $U^{1+}$ norm defined in~\eqref{eq:U1-plus-def}.

We record the invariance of the counting operators~\eqref{eq:r-def} under diagonal shifts.
\begin{remark}\label{rmk:symmetric}
For any $t\in\mathbb{Z}$, the substitution $x\mapsto x-tm$ in the inner sum of~\eqref{eq:r-def} gives
\begin{align*}
    r_H(\mathbf{b}+t(1,\ldots,1);f_1,\ldots,f_{k})
    =r_H(\mathbf{b};f_1,\ldots,f_{k}),
\end{align*}
so $r_H$ depends only on the differences $b_i-b_1$. In particular, with $t=-b_1$ we obtain the \emph{anchored representation}
\begin{align*}
    r_H(\mathbf{b};f_1,\ldots,f_{k})
    =\sum_{x\in\mathbb{Z}}\sum_{m\in[H]}f_1(x)\prod_{i=2}^{k}f_i\bigl(x+(b_i-b_1)m\bigr).
\end{align*}
Correspondingly, $R_H(\lambda;f_1,\ldots,f_{k})$ is unchanged if $\lambda$ is replaced by $\lambda'(\mathbf{b})=\lambda(\mathbf{b}+t(1,\ldots,1))$, and $\lambda'$ is supported on a cube of the same side length as $\lambda$. The hypotheses of Theorem~\ref{thm:1+} are therefore invariant under recentring the support of $\lambda$. We use this invariance in the anchoring reductions below to centre the coefficient support.
\end{remark}

\section{Gowers control}\label{sec:gowers-control}

We now prove that \(R_H\) is controlled by the \(U^k\) norm.  Instead of
assuming that the functions involved are pointwise bounded, we impose a
weaker boundedness condition.  This formulation will be useful when the unbounded setting is considered
in Section~\ref{sec:relative-u1plus}.

\subsection{Iterated Cauchy--Schwarz and the terminal form}\label{sec:engine}

We begin with the moment condition which bounds the majorant factors created
by repeated applications of Cauchy--Schwarz.

\begin{definition}[Boundedness condition]\label{def:cube-moment}
Let $N,B,A,C_0\geq1$, let $H$ and $r$ be positive integers, and let
$w\colon\Z\to\R_{\geq0}$ be finitely supported.  We say that $w$ satisfies
the $(r,A,C_0)$ boundedness condition at scale $(N,B,H)$ if, for every
$0\leq j\leq r$ and all nonzero integers $a_1,\ldots,a_j$ with
$|a_i|\leq C_0B$, one has
\begin{align}\label{eq:cube-moment-condition}
 \sum_x\sum_{u_1,\ldots,u_j\in[0,H-1]}
 \Delta_{a_1u_1,\ldots,a_ju_j}w(x)
 \leq ANH^j.
\end{align}
For $j=0$, this means $\sum_xw(x)\leq AN$.
\end{definition}

Expanding the $j$ iterated differences produces a product over the $2^j$
vertices of a discrete cube of dimension $j$.  This is the reason for referring
to \eqref{eq:cube-moment-condition} as a cube moment bound.

For a finite interval $I\subseteq\Z$ and $j\geq0$, define
\begin{align}\label{eq:Qj-def}
 \cQ_j(I)\coloneqq {}&\{(m,h_1,\ldots,h_j)\in\Z^{j+1}:
       m-\boldsymbol\omega\cdot\mathbf h\in I
       \text{ for every }\boldsymbol\omega\in\{0,1\}^j\},
\end{align}
where $\boldsymbol\omega\cdot\mathbf h=\sum_{\ell=1}^j\omega_\ell h_\ell$.
Thus $\cQ_0(I)=I$.  Separating the two possibilities for the last coordinate
of $\boldsymbol\omega$ gives the elementary identity
\begin{align}\label{eq:Q-recursion}
 &(m,h_1,\ldots,h_j)\in\cQ_j(I)\notag\\
 &\qquad\Longleftrightarrow
 \begin{cases}
 (m,h_1,\ldots,h_{j-1})\in\cQ_{j-1}(I),\\
 (m-h_j,h_1,\ldots,h_{j-1})\in\cQ_{j-1}(I).
 \end{cases}
\end{align}

\begin{definition}[Terminal form]\label{def:terminal}
Let $I\subseteq\Z$ be a finite interval, let $c_1,\ldots,c_r\in\Z$ be
pairwise distinct, let $F\colon\Z\to\C$ be finitely supported, and let
$\nu_i\colon\Z\to\R_{\geq0}$ be finitely supported.  Define
\begin{align}\label{eq:terminal-one}
\cT_{\mathbf c,I}(F;\nu_1,\ldots,\nu_r)
\coloneqq {}&\sum_{(m,h_1,\ldots,h_r)\in\cQ_r(I)}\sum_x
 \Delta_{c_1h_1,\ldots,c_rh_r}F(x)\notag\\
&\quad\times\prod_{i=1}^r
 \Delta_{\big((c_\ell-c_i)h_\ell\big)_{\ell\neq i}}
       \nu_i(x+c_im).
\end{align}
Here
$\Delta_{\big((c_\ell-c_i)h_\ell\big)_{\ell\neq i}}$ means the
iterated difference over all $\ell\in\{1,\ldots,r\}\setminus\{i\}$.
When $I=[H]$, we omit $I$ from the notation $\mathcal{T}_{\mathbf{c},I}$. When all the functions $\nu_i$ are
equal to $\nu$, we write simply $\cT_{\mathbf c,I}(F;\nu)$.
\end{definition}

The next lemma is the algebraic core of the argument.

\begin{lemma}[Iterated Cauchy--Schwarz]\label{lem:fibrewise-CS}
Let $I\subseteq\Z$ be a finite interval, and let
$c_1,\ldots,c_r\in\Z$ be pairwise distinct.  Let
$f_0,\ldots,f_r\colon\Z\to\C$ and
$\nu_1,\ldots,\nu_r\colon\Z\to\R_{\geq0}$ be finitely supported, with
\begin{align}\label{eq:fibre-domination}
    |f_j|\leq\nu_j\qquad \textnormal{for all}\,\,\, 1\leq j\leq r.
\end{align}
Set
\begin{align}\label{eq:Lambda-one-fibre}
    \mathcal{L}
    \coloneqq \sum_x\sum_{m\in I}f_0(x)\prod_{j=1}^rf_j(x+c_jm),
\end{align}
and, for $1\leq j\leq r$, set
\begin{align}\label{eq:Pj}
 \cP_j\coloneqq {}&
 \sum_{\substack{h_1,\ldots,h_{j-1}\in\Z\\
          \exists m:\ (m,h_1,\ldots,h_{j-1})\in\cQ_{j-1}(I)}}
 \sum_y
 \Delta_{\big((c_\ell-c_j)h_\ell\big)_{\ell<j}}\nu_j(y).
\end{align}
Then every $\cP_j$ and
$\cT_{\mathbf c,I}(f_0;\nu_1,\ldots,\nu_r)$ is nonnegative, and
\begin{align}\label{eq:fibrewise-CS}
 |\mathcal{L}|
 \leq
 \cT_{\mathbf c,I}(f_0;\nu_1,\ldots,\nu_r)^{2^{-r}}
 \prod_{j=1}^r\cP_j^{2^{-j}}.
\end{align}
Moreover, let $I\subseteq[H]$, let $A,C_0\geq1$, suppose that every
$\nu_j$ satisfies the $(r,A,2C_0)$ boundedness condition at scale
$(N,B,H)$, and suppose that $|c_j|\leq C_0B$ for every $j$.  Then
\begin{align}\label{eq:iterated-CS-conclusion}
 |\mathcal{L}|
 \leq
 2A^{1-2^{-r}}NH
 \left(
 \frac{\cT_{\mathbf c,I}(f_0;\nu_1,\ldots,\nu_r)}
      {NH^{r+1}}
 \right)^{2^{-r}}.
\end{align}
\end{lemma}

\begin{proof}
For $0\leq j\leq r$, let
\begin{align}\label{eq:Lambdaj}
 \mathcal{L}_j\coloneqq {}&
 \sum_{(m,h_1,\ldots,h_j)\in\cQ_j(I)}\sum_x
 \Delta_{(c_\ell h_\ell)_{\ell\leq j}}f_0(x)\notag\\
 &\times\prod_{i=1}^j
 \Delta_{\big((c_\ell-c_i)h_\ell\big)_{\substack{\ell\leq j\\ \ell\neq i}}}
       \nu_i(x+c_im)\notag\\
 &\times\prod_{i=j+1}^r
 \Delta_{\big((c_\ell-c_i)h_\ell\big)_{\ell\leq j}}
       f_i(x+c_im).
\end{align}
Thus $\mathcal{L}_0=\mathcal{L}$ and
$\mathcal{L}_r=\cT_{\mathbf c,I}(f_0;\nu_1,\ldots,\nu_r)$.

We claim that, for $1\leq j\leq r$,
\begin{align}\label{eq:scalar-recursion}
    |\mathcal{L}_{j-1}|^2\leq\cP_j\mathcal{L}_j.
\end{align}
Fix $h_1,\ldots,h_{j-1}$ and substitute $y=x+c_jm$ in the sum in
$\mathcal{L}_{j-1}$.  The factor in the $j$th slot becomes
\begin{align*}
 \Phi_j(y)
 =\Delta_{\big((c_\ell-c_j)h_\ell\big)_{\ell<j}}f_j(y),
\end{align*}
and, by~\eqref{eq:fibre-domination}, it is bounded pointwise by
\begin{align*}
 \Psi_j(y)
 =\Delta_{\big((c_\ell-c_j)h_\ell\big)_{\ell<j}}\nu_j(y).
\end{align*}
Write the remaining sum over $m$ as
\begin{align*}
G_j(y,\mathbf h_{<j})
={}&\sum_{\substack{m\in\Z\\
       (m,\mathbf h_{<j})\in\cQ_{j-1}(I)}}
 \Delta_{(c_\ell h_\ell)_{\ell<j}}f_0(y-c_jm)\notag\\
&\times\prod_{i<j}
 \Delta_{\big((c_\ell-c_i)h_\ell\big)_{\substack{\ell<j\\ \ell\neq i}}}
 \nu_i\bigl(y+(c_i-c_j)m\bigr)\notag\\
&\times\prod_{i>j}
 \Delta_{\big((c_\ell-c_i)h_\ell\big)_{\ell<j}}
 f_i\bigl(y+(c_i-c_j)m\bigr),
\end{align*}
where $\mathbf h_{<j}=(h_1,\ldots,h_{j-1})$.  We then have
\begin{align*}
    \mathcal{L}_{j-1}
    =\sum_{h_1,\ldots,h_{j-1}}\sum_y
       \Phi_j(y)G_j(y,\mathbf h_{<j}).
\end{align*}

Expanding $|G_j(y,\mathbf h_{<j})|^2$ and writing its two variables as
$m,m'$, we make the change of variables $h_j=m-m'$.  The recursion
\eqref{eq:Q-recursion} identifies the resulting domain with $\cQ_j(I)$, and
comparison of the two copies of each factor gives
\begin{align}\label{eq:square-identity}
    \mathcal{L}_j
    =\sum_{h_1,\ldots,h_{j-1}}\sum_y
       \Psi_j(y)|G_j(y,\mathbf h_{<j})|^2.
\end{align}
This identity shows that $\mathcal{L}_j\geq0$.  Applying weighted
Cauchy--Schwarz with weight $\Psi_j$ to the preceding expression for
$\mathcal{L}_{j-1}$ produces $\cP_j$ as its first factor and the right side
of~\eqref{eq:square-identity} as its second.  This proves
\eqref{eq:scalar-recursion}. From the product defining $\cP_j$ it immediately follows that
$\cP_j\geq0$.

Iterating~\eqref{eq:scalar-recursion} proves~\eqref{eq:fibrewise-CS}.
For the final assertion, the projection of $\cQ_{j-1}(I)$ to each
$h_\ell$ coordinate lies in $[-H+1,H-1]$.  Split each such interval
into its nonnegative and nonpositive parts.  The summands in~\eqref{eq:Pj}
are nonnegative and
$c_\ell-c_j\neq0$, so~\eqref{eq:cube-moment-condition} gives
\begin{align}\label{eq:Pj-bound}
    \cP_j\leq2^{j-1}ANH^{j-1}.
\end{align}
The identities
\begin{align*}
 \sum_{j=1}^r2^{-j}=1-2^{-r},
 \qquad
 \sum_{j=1}^r(j-1)2^{-j}=1-(r+1)2^{-r}
\end{align*}
give~\eqref{eq:iterated-CS-conclusion}. The precise leading constant is
$2^{1-(r+1)2^{-r}}\leq2$.
\end{proof}
We now specialise the estimate to bounded majorants.

For $u\in[0,H-1]$, set $\varpi(0)=1$ and $\varpi(u)=2$ for $u>0$; for
$\mathbf h=(h_1,\ldots,h_r)$ put
$\varpi(\mathbf h)=\prod_{i=1}^r\varpi(h_i)$.  This weight records the
multiplicity arising from folding the signs of the $h_i$.

\begin{lemma}[Bounded majorant estimate]\label{lem:flat-CS}
Let $r$ and $H$ be positive integers, let $C_1\geq1$ and $N\geq1$, and set
$B=N/H$.  Let $c_1,\ldots,c_r$ be pairwise distinct nonzero integers with
$|c_i|\leq C_1B$.  Let
$f_0,f_1,\ldots,f_r\colon\mathbb Z\to\mathbb C$ be $1$-bounded and supported
on $[N]$, and set
\begin{align*}
    \mathcal{L}_{\mathbf c}
    =\sum_x\sum_{m\in[H]}f_0(x)\prod_{i=1}^r f_i(x+c_im).
\end{align*}
Define
\begin{align}\label{eq:Tc-def}
T_{\mathbf c}(f_0)
={}&
\sum_{\substack{h_1,\ldots,h_r\in[0,H-1]\\
                 h_1+\cdots+h_r\leq H-1}}
\varpi(\mathbf h)(H-h_1-\cdots-h_r)\notag\\
&\qquad\times
\sum_x\Delta_{c_1h_1,\ldots,c_rh_r}f_0(x).
\end{align}
Then
\begin{align}\label{eq:flat-CS-conclusion}
    |\mathcal{L}_{\mathbf c}|
    \ll_{r,C_1}
    NH
    \left(
    \frac{|T_{\mathbf c}(f_0)|}{NH^{r+1}}
    \right)^{1/2^r}.
\end{align}
\end{lemma}

\begin{proof}
Choose an interval $J$ of length $N_0\asymp_{r,C_1}N$ such that $[N]\subset J$
and 
\begin{align}\label{eq:terminal-argument}
    x+c_jm+\sum_{i\neq j}\omega_i(c_i-c_j)h_i\in J
\end{align}
for all $x\in[N]$, $1\leq j\leq r$, $m\in\{0,\ldots,H\}$,
$h_i\in\{-H,\ldots,H\}$ and $\omega_i\in\{0,1\}$.  Enlarge $J$ by a
constant factor, if necessary, so that $|c_i|\leq B_0$, where
$B_0=N_0/H$.  Set $\nu_1=\cdots=\nu_r=1_J$.

For every $0\leq j\leq r$ and all integers $a_1,\ldots,a_j$,
\begin{align*}
 \sum_x\sum_{u_1,\ldots,u_j\in[0,H-1]}
 \Delta_{a_1u_1,\ldots,a_ju_j}1_J(x)
 \leq N_0H^j,
\end{align*}
because the cube is a product of values of $1_J$ and includes the factor
$1_J(x)$.  Thus the final assertion of Lemma~\ref{lem:fibrewise-CS} applies
with
$N=N_0$, $B=B_0$, $A=C_0=1$, and
\begin{align}\label{eq:flat-applied}
 |\mathcal{L}_{\mathbf c}|
 \leq2N_0H
 \left(
 \frac{\cT_{\mathbf c}(f_0;1_J)}{N_0H^{r+1}}
 \right)^{2^{-r}}.
\end{align}

It remains to simplify the terminal form in~\eqref{eq:flat-applied}.  Set
\begin{align*}
    A(\mathbf c,\mathbf h)
    =
    \sum_x\Delta_{c_1h_1,\ldots,c_rh_r}f_0(x).
\end{align*}
On expanding $\cT_{\mathbf c}(f_0;1_J)$ according
to~\eqref{eq:terminal-one}, the summands involve the factor
$\Delta_{c_1h_1,\ldots,c_rh_r}f_0(x)$.  Whenever this factor is nonzero,
we must have
$f_0(x)\neq0$, and hence $x\in[N]$.  Moreover,
$(m,\mathbf h)\in\cQ_r([H])$ implies that $m\in[H]$ and
$|h_i|\leq H-1$ for every $i$.  It follows from the choice of $J$
and~\eqref{eq:terminal-argument} that every factor arising from $1_J$ is
equal to $1$.  Therefore
\begin{align}\label{eq:terminal-reduced}
    \cT_{\mathbf c}(f_0;1_J)
    =
    \sum_{(m,\mathbf h)\in\cQ_r([H])}
    A(\mathbf c,\mathbf h).
\end{align}

For fixed $\mathbf h$, the condition
$(m,\mathbf h)\in\cQ_r([H])$ is equivalent to
\begin{align*}
    1+\boldsymbol\omega\cdot\mathbf h
    \leq m
    \leq H+\boldsymbol\omega\cdot\mathbf h
    \qquad
    \text{for every }\boldsymbol\omega\in\{0,1\}^r.
\end{align*}
Thus the admissible integers $m$ are precisely those satisfying
\begin{align*}
    1+\sum_{i=1}^r\max(h_i,0)
    \leq m
    \leq H+\sum_{i=1}^r\min(h_i,0).
\end{align*}
Their number is
\begin{align}\label{eq:m-count}
    \left(H-\sum_{i=1}^r|h_i|\right)_+.
\end{align}
Consequently,~\eqref{eq:terminal-reduced} becomes
\begin{align*}
    \cT_{\mathbf c}(f_0;1_J)
    =
    \sum_{\mathbf h\in\mathbb Z^r}
    \left(H-\sum_{i=1}^r|h_i|\right)_+
    A(\mathbf c,\mathbf h).
\end{align*}

Changing the sign of one $h_i$ and translating the variable $x$ gives
\begin{align*}
    A(\mathbf c,h_1,\ldots,-h_i,\ldots,h_r)
    =
    \overline{A(\mathbf c,h_1,\ldots,h_i,\ldots,h_r)}.
\end{align*}
Folding the sum over the signs of the $h_i$ therefore yields
\begin{align*}
    \cT_{\mathbf c}(f_0;1_J)
    &=
    \sum_{\mathbf u\in\mathbb Z_{\geq0}^r}
    \varpi(\mathbf u)
    \left(H-\sum_{i=1}^r u_i\right)_+
    \operatorname{Re}A(\mathbf c,\mathbf u)\\
    &=
    \operatorname{Re}T_{\mathbf c}(f_0).
\end{align*}
Here, when $\mathbf u=\mathbf0$, the cube is
$|f_0|^{2^r}$ and is already real.  Hence
\begin{align*}
    \cT_{\mathbf c}(f_0;1_J)
    \leq |T_{\mathbf c}(f_0)|.
\end{align*}
Inserting this estimate into~\eqref{eq:flat-applied} and using
$N_0\asymp_{r,C_1}N$ proves~\eqref{eq:flat-CS-conclusion}.
\end{proof}

\subsection{Concatenation}\label{sec:concatenation}

We need the following variant of the quantitative
concatenation estimate of Peluse and Prendiville~\cite[Lemma~5.3]{PP}. We give a somewhat different argument that avoids Fourier analysis and keeps track of the quantitative dependencies in the exponents.

\begin{lemma}[Concatenation]\label{lem:pp_5.3}
Let $s,B_1,H_1$ be positive integers with
$B_1H_1\asymp_sN$, and let $f\colon\mathbb Z\to\mathbb C$ be $1$-bounded and
supported on $[N]$.  Set
\begin{align*}
    \alpha=
    \mathbb E_{b\in[B_1]}\mathbb E_{h\in[H_1]}
    \|\Delta_{bh}f\|_{U^s[N]}^{2^s}.
\end{align*}
Then
\begin{align}\label{eq:sharp-concat-conclusion}
    \|f\|_{U^{s+1}[N]}^{2^{s+1}}
    \gg_s \frac{\alpha^2}{\log(2/\alpha)}.
\end{align}
\end{lemma}

\begin{proof}
We may assume that $\alpha>0$.  Since the roles of $B_1$ and $H_1$ are
symmetric, we may also assume that $B_1\leq H_1$.  By the definition of the
Gowers norm,
\begin{align*}
    \alpha\|1_{[N]}\|_{U^s(\mathbb Z)}^{2^s}
    ={}&
    \mathbb E_{b\in[B_1]}\mathbb E_{h\in[H_1]}
    \sum_{x\in\mathbb Z}\sum_{\mathbf y\in\mathbb Z^s}
    \Delta_{\mathbf y}f(x)
    \overline{\Delta_{\mathbf y}f(x+bh)}.
\end{align*}
Moreover, since $f$ is $1$-bounded and supported on $[N]$,
\begin{align*}
    \sum_{x\in\mathbb Z}\sum_{\mathbf y\in\mathbb Z^s}
    |\Delta_{\mathbf y}f(x)|^2
    \leq \|1_{[N]}\|_{U^s(\mathbb Z)}^{2^s}.
\end{align*}
The Cauchy--Schwarz inequality gives
\begin{align*}
&\alpha^2\|1_{[N]}\|_{U^s(\mathbb Z)}^{2^{s+1}}\notag\\
=&\left|
\sum_{x\in\mathbb Z}\sum_{\mathbf y\in\mathbb Z^s}
\Delta_{\mathbf y}f(x)
\mathbb E_{b\in[B_1]}\mathbb E_{h\in[H_1]}
\overline{\Delta_{\mathbf y}f(x+bh)}
\right|^2\notag\\
\leq&
\|1_{[N]}\|_{U^s(\mathbb Z)}^{2^s}
\sum_{x\in\mathbb Z}\sum_{\mathbf y\in\mathbb Z^s}
\left|
\mathbb E_{b\in[B_1]}\mathbb E_{h\in[H_1]}
\overline{\Delta_{\mathbf y}f(x+bh)}
\right|^2\notag\\
=&
\|1_{[N]}\|_{U^s(\mathbb Z)}^{2^s}
\mathbb E_{b,b'\in[B_1]}\mathbb E_{h,h'\in[H_1]}
\|\Delta_{bh-b'h'}f\|_{U^s(\mathbb Z)}^{2^s},
\end{align*}
where for the last identity, we expanded the square and translated $x$ by $b'h'$.
Consequently,
\begin{align}\label{eq:concat-CS}
    \alpha^2N^{2s+2}
    \ll_s
    N^{s+1}\mathbb E_{b,b'\in[B_1]}\mathbb E_{h,h'\in[H_1]}
    \|\Delta_{bh-b'h'}f\|_{U^s(\mathbb Z)}^{2^s}.
\end{align}

Set $D=\max(2,C_s\alpha^{-2})$, where $C_s$ is sufficiently large.  The
proportion of pairs $(b,b')\in [B_1]^2$ for which $\gcd(b,b')>D$ is
\begin{align*}
    \ll\sum_{d>D}\frac1{d^2}\ll D^{-1}.
\end{align*}
Since $\|\Delta_tf\|_{U^s(\mathbb Z)}^{2^s}\ll_sN^{s+1}$ for every $t$,
the contribution of these pairs to the right side
of~\eqref{eq:concat-CS} can be absorbed into its left side.  If $D>B_1$,
there are no such pairs.

For $t\in\mathbb Z$, let $r_D(t)$ be the proportion of quadruples
$(b,b',h,h')\in[B_1]^2\times[H_1]^2$ such that
\begin{align*}
    bh-b'h'=t,
    \qquad
    \gcd(b,b')\leq D.
\end{align*}
After discarding the pairs with $\gcd(b,b')>D$, the expectation on the
right side of~\eqref{eq:concat-CS} is
\begin{align}\label{eq:rD-identity}
    \sum_{t\in\mathbb Z}r_D(t)
    \|\Delta_tf\|_{U^s(\mathbb Z)}^{2^s}.
\end{align}
Fix $d=\gcd(b,b')$ and set $b=du$ and $b'=du'$, where
$\gcd(u,u')=1$.  For fixed $u,u'$, the solutions of
\begin{align*}
    uh-u'h'=t/d
\end{align*}
are empty unless $d\mid t$.  If $(h_0,h_0')$ is one solution, then all the
integer solutions are
\begin{align*}
    (h,h')=(h_0+u'n,h_0'+un),
    \qquad n\in\mathbb Z.
\end{align*}
Their number in $[H_1]^2$ is
\begin{align*}
    \ll 1+\frac{H_1}{\max(u,u')}.
\end{align*}
After summing over $u,u'\leq B_1/d$, the contribution of the first term is
$O((B_1/d)^2)=O(B_1H_1/d)$, and, by symmetry, the contribution of the second term
is at most
\begin{align*}
    2\sum_{u'\leq B_1/d}\,\sum_{u\leq u'}\frac{H_1}{u'}
    \ll \frac{B_1H_1}{d}.
\end{align*}
Consequently, for each $d\leq\min(D,B_1)$,
\begin{align*}
&\#\{(b,b',h,h')\in[B_1]^2\times[H_1]^2:
bh-b'h'=t,\ \gcd(b,b')=d\}\notag\\
&\qquad\ll \frac{B_1H_1}{d}.
\end{align*}
It follows that
\begin{align}\label{eq:multiplicity-bound}
    r_D(t)
    \ll
    \frac{1}{B_1^2H_1^2}
    \sum_{d\leq\min(D,B_1)}\frac{B_1H_1}{d}
    \ll
    \frac{\log(2D)}{B_1H_1}.
\end{align}
Hence from~\eqref{eq:rD-identity} we obtain
\begin{align*}
    \alpha^2N^{2s+2}
    &\ll_s
    \frac{N^{s+1}\log(2/\alpha)}{B_1H_1}
    \sum_{t\in\mathbb Z}\|\Delta_tf\|_{U^s(\mathbb Z)}^{2^s}\\
    &=
    \frac{N^{s+1}\log(2/\alpha)}{B_1H_1}
    \|f\|_{U^{s+1}(\mathbb Z)}^{2^{s+1}}.
\end{align*}
Using $B_1H_1\asymp_sN$ and the normalisation of the Gowers norms proves the
result.
\end{proof}

As a consequence we obtain the following iterated concatenation result.

\begin{corollary}[Iterated averaged concatenation]
\label{cor:iterated-concatenation}
Let $s,r,B_1,H_1$ be positive integers with $B_1H_1\asymp_sN$, let
$0<\alpha\leq1$, and let $f$ be $1$-bounded and supported on $[N]$.  If
\begin{align}\label{eq:iterated-concat-hyp}
    \mathbb E_{\mathbf b\in[B_1]^r}
    \mathbb E_{\mathbf h\in[H_1]^r}
    \|\Delta_{b_1h_1,\ldots,b_rh_r}f\|_{U^s[N]}^{2^s}
    \geq\alpha,
\end{align}
then
\begin{align}\label{eq:iterated-concat-conclusion}
    \|f\|_{U^{s+r}[N]}^{2^{s+r}}
    \gg_{s,r}
    \frac{\alpha^{2^r}}
    {(\log(2/\alpha))^{2^r-1}}.
\end{align}
\end{corollary}

\begin{proof}
We first establish an averaged estimate which will be iterated to prove the claim.  Let
$\Omega$ be a finite nonempty probability space, and for $\omega\in \Omega$ let $f_\omega$ be
$1$-bounded and supported on $[N]$, and set
\[
 \beta_\omega
 =\mathbb E_{b\in[B_1]}\mathbb E_{h\in[H_1]}
   \|\Delta_{bh}f_\omega\|_{U^s[N]}^{2^s}.
\]
If $\mathbb E_{\omega\in \Omega}\beta_\omega\geq\beta$, then we claim that
\begin{align}\label{eq:averaged-concat}
 \mathbb E_{\omega\in \Omega}\|f_\omega\|_{U^{s+1}[N]}^{2^{s+1}}
 \gg_s\frac{\beta^2}{\log(2/\beta)}.
\end{align}
Indeed, the normalisation gives $0\leq\beta_\omega\leq1$, and
Lemma~\ref{lem:pp_5.3} gives the right side with $\mathbb{E}_{\omega\in \Omega}\beta_\omega^2/\log(2/\beta_{\omega})$ in place
of $\beta$.  The function
\[
 \phi(u)=\frac{u^2}{\log(2/u)},\qquad \phi(0)=0,
\]
is increasing and convex on $[0,1]$: writing $L_u=\log(2/u)$, one has
\[
 \phi'(u)=\frac{2u}{L_u}+\frac{u}{L_u^2},\qquad
 \phi''(u)=\frac{2}{L_u}+\frac{3}{L_u^2}+\frac{2}{L_u^3}>0.
\]
Thus~\eqref{eq:averaged-concat} follows from Jensen's inequality.

For $0\leq j<r$, let
\begin{align*}
    \alpha_j={}&
    \mathbb E_{\mathbf b\in[B_1]^{r-j}}
    \mathbb E_{\mathbf h\in[H_1]^{r-j}}
    \|\Delta_{b_1h_1,\ldots,b_{r-j}h_{r-j}}f\|_{U^{s+j}[N]}^{2^{s+j}},
\end{align*}
and let $\alpha_r=\|f\|_{U^{s+r}[N]}^{2^{s+r}}$.  Thus
$\alpha_0\geq\alpha$.  Apply~\eqref{eq:averaged-concat} successively.  At
step $j$, take
$\Omega=[B_1]^{r-j-1}\times[H_1]^{r-j-1}$ with the uniform measure and
\[
    f_\omega
    =\Delta_{b_1h_1,\ldots,b_{r-j-1}h_{r-j-1}}f.
\]
The variables averaged in the  estimate are
$b_{r-j}$ and $h_{r-j}$.
At every step,
\begin{align*}
    \alpha_{j+1}\gg_{s,r}
    \frac{\alpha_j^2}{\log(2/\alpha_j)}.
\end{align*}
Induction gives
\begin{align*}
    \alpha_j\gg_{s,r}
    \frac{\alpha^{2^j}}
    {(\log(2/\alpha))^{2^j-1}}.
\end{align*}
The inductive lower bound also gives
$\log(2/\alpha_j)\ll_{s,r}\log(2/\alpha)$, so each intermediate logarithm
is controlled by the original one.  Taking $j=r$ proves the result.
\end{proof}

\subsection{Generalised von Neumann estimates in the bounded setting}\label{sec:bounded-gvn}

We combine the iterated Cauchy--Schwarz estimate of
Subsection~\ref{sec:engine} with iterated concatenation from
Subsection~\ref{sec:concatenation}.

This yields a generalised von Neumann theorem in which all quantitative
dependencies are tracked simultaneously.

\begin{proposition}[Quantitative generalised von Neumann estimate for bounded functions]\label{prop:Uk}
Let $k\geq2$ and $C_0\geq1$.  There are constants
$c_{k,C_0},C_{k,C_0}'>0$ such that the following holds.  Let
$0<\delta<c_{k,C_0}$, let $H$ be a positive integer, set $B=N/H$, and assume
\begin{align}\label{eq:Uk-range}
    B\geq C_{k,C_0}'\delta^{-1},
    \qquad
    H\geq C_{k,C_0}'\delta^{-2^{k-1}}.
\end{align}
Let $\lambda\colon\mathbb Z^k\to\mathbb C$ be $1$-bounded and supported on
$([-C_0B,C_0B]\cap\mathbb Z)^k$, and let
$f_1,\ldots,f_k\colon\mathbb Z\to\mathbb C$ be $1$-bounded and supported on
$[N]$.  If
\begin{align}\label{eq:Uk-hyp}
    |R_H(\lambda;f_1,\ldots,f_k)|
    \geq \delta B^k\frac{N^2}{B},
\end{align}
then, for every $1\leq i\leq k$,
\begin{align}\label{eq:Uk-conclusion}
    \|f_i\|_{U^k[N]}
    \gg_{k,C_0}
    \frac{\delta^{2^{k-1}}}
    {(\log(2/\delta))^{(2^{k-1}-1)/2^k}}.
\end{align}
\end{proposition}

Since $H=N/B$, the more restrictive symmetric range
\begin{align}\label{eq:Uk-symmetric-range}
    C_{k,C_0}'\delta^{-2^{k-1}}\leq B
    \leq c_{k,C_0}N\delta^{2^{k-1}}
\end{align}
is sufficient for~\eqref{eq:Uk-range}.

\begin{proof}
The hypotheses and the counting form are invariant under simultaneous
permutation of the functions and the coefficient coordinates, so it suffices
to prove the conclusion for $f_k$.  The iterated Cauchy--Schwarz
estimate produces an averaged $U^1$ lower bound, which iterated concatenation
raises to a $U^k$ lower bound.

Set $r=k-1$.  There are $O_{k,C_0}(B)$ possible values of $b_k$ on the
support of $\lambda$.
First pigeonholing in $b_k$ and then making the change of variables $x\mapsto x-b_k^*m$ for a suitable $b_k^{*}$ gives
\begin{align}\label{eq:anchored-large}
\left|
\sum_{\mathbf c}\lambda_1(\mathbf c)
\sum_x\sum_{m\in[H]}f_k(x)
\prod_{i=1}^r f_i(x+c_im)
\right|
\gg_{k,C_0} \delta B^rNH,
\end{align}
where $c_i=b_i-b_k^*$ and
\begin{align*}
 \lambda_1(\mathbf c)
 =\lambda(c_1+b_k^*,\ldots,c_r+b_k^*,b_k^*).
\end{align*} 
In particular, $\lambda_1$ is supported on a centred cube of side length
$O_{k,C_0}(B)$.

Call $\mathbf c$ \emph{degenerate} if some $c_i=0$ or some $c_i=c_j$ with
$i\ne j$.  These tuples lie in $O_k(1)$ affine hyperplanes and hence their number is
$O_{k,C_0}(B^{r-1})$.  Each sum over $(x,m)$ is at most $NH$, so their total
contribution to~\eqref{eq:anchored-large} is $O_{k,C_0}(B^{r-1}NH)$.  The first
condition in~\eqref{eq:Uk-range} allows us to discard them.  Writing
\begin{align*}
 \mathcal{L}_{\mathbf c}
 =\sum_x\sum_{m\in[H]}f_k(x)
 \prod_{i=1}^r g_i(x+c_im),
\end{align*}
we obtain
\begin{align}\label{eq:sum-counts-large}
 \sum_{\mathbf c\ \mathrm{nondegenerate}}|\mathcal{L}_{\mathbf c}|
 \gg_{k,C_0}\delta B^rNH.
\end{align}

Apply Lemma~\ref{lem:flat-CS} to each term
in~\eqref{eq:sum-counts-large}.  Since there are
$O_{k,C_0}(B^r)$ coefficient tuples,
H\"older's inequality gives
\begin{align}\label{eq:first-terminal-moment}
 \sum_{\mathbf c\ \mathrm{nondegenerate}}|T_{\mathbf c}(f_k)|
 \gg_{k,C_0}\delta^{2^r}B^rNH^{r+1}.
\end{align}
For $\mathbf h=(h_1,\ldots,h_r)$, set
\begin{align*}
 A(\mathbf c,\mathbf h)
 &=\sum_x\Delta_{c_1h_1,\ldots,c_rh_r}f_k(x),\\
 w(\mathbf h)
 &=\varpi(\mathbf h)(H-h_1-\cdots-h_r)_+.
\end{align*}
The triangle inequality and~\eqref{eq:first-terminal-moment} imply
\begin{align}\label{eq:terminal-absolute-sum}
 \sum_{\mathbf c\ \mathrm{nondegenerate}}
 \sum_{\mathbf h\in[0,H-1]^r}
 w(\mathbf h)|A(\mathbf c,\mathbf h)|
 \gg_{k,C_0}\delta^{2^r}B^rNH^{r+1}.
\end{align}
The contribution from tuples with some $h_i=0$ is
$O_{k,C_0}(B^rNH^r)$.
The second condition in~\eqref{eq:Uk-range} allows us to
discard these tuples.  Since
\begin{align*}
 \sum_{\mathbf c,\mathbf h}w(\mathbf h)
 \ll_{k,C_0}B^rH^{r+1},
\end{align*}
the Cauchy--Schwarz inequality gives
\begin{align*}
&\sum_{\substack{\mathbf c\ \mathrm{nondegenerate}\\
1\leq h_1,\ldots,h_r\leq H-1\\
h_1+\cdots+h_r\leq H-1}}
w(\mathbf h)|A(\mathbf c,\mathbf h)|^2
\gg_{k,C_0}\delta^{2^{r+1}}B^rN^2H^{r+1}.
\end{align*}
Using $w(\mathbf h)\ll_kH$, we conclude that
\begin{align}\label{eq:U1-square-sum}
\sum_{\substack{\mathbf c\ \mathrm{nondegenerate}\\
1\leq h_1,\ldots,h_r\leq H-1\\
h_1+\cdots+h_r\leq H-1}}
|A(\mathbf c,\mathbf h)|^2
\gg_{k,C_0}\delta^{2^{r+1}}B^rH^rN^2.
\end{align}

Choose a positive integer $B_1\asymp_{k,C_0}B$ so that every coefficient tuple
$\mathbf c$ having nonzero contribution to the sum~\eqref{eq:U1-square-sum} lies in
$(([-B_1,B_1]\cap\mathbb Z)\setminus\{0\})^r$.  All summands in
\eqref{eq:U1-square-sum} are nonnegative, so we may enlarge the coefficient
set $\mathbf{c}$ and the simplex in $\mathbf h$ to this product set and to $[H]^r$,
respectively.  Reversing the sign of any $c_i$ translates the variable $x$ in
$A(\mathbf c,\mathbf h)$ and possibly conjugates the sum; in particular, it
leaves $|A(\mathbf c,\mathbf h)|$ unchanged.  Folding the coefficient signs
and using
\begin{align*}
 \|\Delta_{a_1h_1,\ldots,a_rh_r}f_k\|_{U^1[N]}^2
 =\frac{|A(\mathbf a,\mathbf h)|^2}{|[N]|^2}
\end{align*}
therefore gives
\begin{align}\label{eq:initial-U1-average}
 \mathbb E_{a_1,\ldots,a_r\in[B_1]}
 \mathbb E_{h_1,\ldots,h_r\in[H]}
 \|\Delta_{a_1h_1,\ldots,a_rh_r}f_k\|_{U^1[N]}^2
 \gg_{k,C_0}\delta^{2^{r+1}}.
\end{align}
Here $B_1H\asymp_{k,C_0}N$.  Let
$\eta=c(k,C_0)\delta^{2^k}$, where $c(k,C_0)>0$ is
sufficiently small.  Since $r+1=k$, equation~\eqref{eq:initial-U1-average}
implies the hypothesis of Corollary~\ref{cor:iterated-concatenation} with
$s=1$, $r=k-1$, and $\alpha=\eta$.  This yields
\begin{align}\label{eq:Uk-power-bound}
 \|f_k\|_{U^k[N]}^{2^k}
 \gg_{k,C_0}
 \frac{\delta^{2^{2k-1}}}
 {(\log(2/\delta))^{2^{k-1}-1}},
\end{align}
since
\begin{align*}
 \eta^{2^{k-1}}\asymp_{k,C_0}\delta^{2^{2k-1}},
 \qquad
 \log(2/\eta)\asymp_{k,C_0}\log(2/\delta).
\end{align*}
Taking roots of order $2^k$ proves~\eqref{eq:Uk-conclusion} for $f_k$, and hence for
every slot by the symmetry observed.
\end{proof}

We also record an anchored version of the preceding estimate.

\begin{corollary}[Anchored generalised von Neumann estimate]
\label{cor:Uk-anchored}
Let $k\geq2$ and $C_0\geq1$, and let $0<\delta<c_{k,C_0}$ and $H$ be a
positive integer.  Set $B=N/H$, and assume
\begin{align*}
    B\geq C_{k,C_0}'\delta^{-1},
    \qquad
    H\geq C_{k,C_0}'\delta^{-2^{k-1}}.
\end{align*}
Let $\lambda\colon\mathbb Z^{k-1}\to\mathbb C$ be $1$-bounded and supported
on $([-C_0B,C_0B]\cap\mathbb Z)^{k-1}$, and let
$f_1,\ldots,f_k$ be $1$-bounded and supported on $[N]$.  If
\begin{align*}
\left|
\sum_{\mathbf b}\lambda(\mathbf b)
\sum_x\sum_{m\in[H]}f_1(x)\prod_{i=2}^kf_i(x+b_im)
\right|
\geq\delta B^{k-2}N^2,
\end{align*}
then, for every $1\leq i\leq k$,
\begin{align*}
 \|f_i\|_{U^k[N]}
 \gg_{k,C_0}
 \frac{\delta^{2^{k-1}}}
 {(\log(2/\delta))^{(2^{k-1}-1)/2^k}}.
\end{align*}
\end{corollary}

\begin{proof}
For $f_1$, the proof of Proposition~\ref{prop:Uk} applies from
\eqref{eq:anchored-large} onwards.  For $f_i$ with $i\geq2$, the change of
variables $x\mapsto x-b_im$ makes the coefficient of $f_i$ equal to zero.
The remaining coefficients are $-b_i$ in the first slot and $b_j-b_i$ for
$j\ne i$.  This invertible integral change of coefficient variables preserves
the $1$-boundedness of the weight and maps the original support cube into a
centred cube of side length $O_{k,C_0}(B)$.  The resulting expression has the
form of~\eqref{eq:anchored-large}, so the proof of
Proposition~\ref{prop:Uk} from that display onward gives the claimed estimate
for every $f_i$.
\end{proof}

\section{Degree lowering}\label{sec:degree-lowering}

We prove that a large $U^s$ norm of the dual function forces a large
$U^{s-1}$ norm for $s\geq3$, and that the resulting $U^2$ control forces
$U^{1+}$ structure.  The strategy follows Peluse--Prendiville~\cite{PP}:
we expand the Gowers norm, replace the shifted dual functions by exponential
phases through the $U^2$ inverse theorem, and lower the complexity of those
phases using dual--difference interchange and low rank arguments.  We include proofs of
several preparatory results from that strategy in order to record the
quantitative dependencies and to formulate them in the form used here.

\subsection{Preparation}
We use the following inverse theorem for the $U^2(\mathbb{Z})$ norm.

\begin{lemma}\label{lem:U2inv} Let $N\geq 1$ and $\delta\in (0,1)$. Let $f\colon \mathbb{Z}\to \mathbb{C}$ be $1$-bounded and supported on $[N]$. Suppose that
\begin{align*}
\|f\|_{U^2(\mathbb{Z})}^4\geq \delta^4 N^3.
\end{align*}
Then, for every integer $Q\geq20N\delta^{-2}$, there exists
$\alpha\in Q^{-1}\mathbb Z/\mathbb Z$ such that
\begin{align*}
\left|\sum_{x\in [N]}f(x)e(\alpha x)\right|\geq \frac{\delta^{2}}{2}N.
\end{align*}

\end{lemma}

\begin{proof}
Writing
\begin{align*}
\widehat{f}(\beta)=\sum_{x\in \mathbb{Z}}f(x)e(\beta x),    
\end{align*}
by expanding out the fourth power we have
\begin{align*}
    \int_{0}^1|\widehat f(\beta)|^4\,\d\beta
    =\sum_{\substack{x,h_1,h_2\in\mathbb Z\\
    x,\ x+h_1,\ x+h_2,\ x+h_1+h_2\in[N]}}
    f(x)\overline{f(x+h_1)}
    \overline{f(x+h_2)}
    f(x+h_1+h_2)=\|f\|_{U^2(\mathbb{Z})}^4.
\end{align*}
Hence 
\begin{align*}
\|f\|_{U^2(\mathbb{Z})}^4\leq \int_{0}^1 |\widehat f(\beta)|^2  \d\beta \cdot \max_{\beta \in \mathbb{T}} |\widehat{f}(\beta)|^2.
\end{align*}
Using Parseval's identity, we conclude that for some $\beta_0\in\mathbb T$ we have
\begin{align*}
\left|\sum_{x\in \mathbb{Z}}f(x)e(\beta_0 x)\right|\geq \delta^2 N.    \end{align*}
Let $\beta'$ be a nearest point to $\beta_0$ in the grid
$Q^{-1}\mathbb Z/\mathbb Z$.  Then
$\|\beta'-\beta_0\|\leq(2Q)^{-1}\leq\delta^2/(40N)$, so, since $f$ is
$1$-bounded and $|e(x)-e(y)|\leq2\pi\|x-y\|$, we have
\begin{align*}
\left|\sum_{x\in \mathbb{Z}}f(x)e(\beta' x)\right|\geq \frac{\delta^2}{2} N
\end{align*}
as required.
\end{proof}

We also need the following elementary lemma, which we extract from the
proof of~\cite[Lemma 6.5]{PP}. We follow their notation of using $\partial$ to
denote certain alternating sums that lift functions from $s$ to $2s$ variables.

\begin{lemma}[Constant $\partial$ implies low complexity]\label{lem:lowcomplexity}
Let $N\geq 1$, let $s\in\mathbb{N}$ and let $\delta \in (0,1)$. Let $\partial \phi\colon \mathbb{Z}^{2s}\to \mathbb{T}$ and $\phi\colon \mathbb{Z}^s\to \mathbb{T}$ be functions such that for all $\mathbf{h}^0,\mathbf{h}^1\in \mathbb{Z}^s$ we have
\begin{align*}
\partial\phi(\mathbf{h}^0,\mathbf{h}^1)=\sum_{\omega \in \{0,1\}^s}(-1)^{|\omega|}\phi(\mathbf{h}^{\omega}),
\qquad \mathbf{h}^{\omega}\coloneqq(h_1^{\omega_1},\ldots,h_s^{\omega_s}).
\end{align*}
Suppose also that there is $c_0\in \mathbb{T}$ such that $\partial\phi(\mathbf{h}^0,\mathbf{h}^1)=c_0$ for all $(\mathbf{h}^0,\mathbf{h}^1)\in \mathcal{H}$, where $\mathcal{H}\subset [N]^{2s}$ is some set with $|\mathcal{H}|\geq \delta N^{2s}$. Then $\phi$ is of low rank in the following sense: there exists a set $\mathcal{H}'\subset [N]^s$, contained in the projection of $\mathcal{H}$ to the first $s$ coordinates, with $|\mathcal{H}'|\geq \delta N^{s}$ and functions $\phi_i\colon \mathbb{Z}^s\to \mathbb{T}$ for $i\in \{1,\ldots, s\}$, with $\phi_i$ independent of the $i$th coordinate, such that for $\mathbf{h}\in \mathcal{H}'$ we have
\begin{align*}
\phi(\mathbf{h})=\sum_{i=1}^{s}\phi_i(\mathbf{h}). \end{align*}
\end{lemma}

\begin{proof}
By the pigeonhole principle, there is some $\widetilde{\mathbf{h}}\in [N]^{s}$ such that $\mathcal{H}'\coloneqq \{\mathbf{h}\in[N]^s:(\mathbf{h},\widetilde{\mathbf{h}})\in \mathcal{H}\}$ has size $\geq \delta N^{s}$. For $\mathbf{h}\in \mathcal{H}'$, taking $\mathbf{h}^0=\mathbf{h}$ and $\mathbf{h}^1=\widetilde{\mathbf{h}}$, we have
\begin{align*}
\phi(\mathbf{h})=c_0-\sum_{\omega\in \{0,1\}^s\setminus \{0\}}(-1)^{|\omega|}\phi(\mathbf h^{\omega}).
\end{align*}
The claim follows since each of the functions $\mathbf h\mapsto\phi(\mathbf{h}^{\omega})$ with $\omega\in \{0,1\}^s\setminus \{0\}$ is independent of every coordinate $h_i$ with $\omega_i=1$, so grouping the terms by such an $i$ (and absorbing the constant $c_0$ into one of them) gives the required decomposition.
\end{proof}

The difference operator $\Delta_h$ does not commute with summation, so a
difference of a dual function is not generally the dual of the corresponding
differences.  The following interchange inequality for dual functions and
differences
is~\cite[Lemma~6.3]{PP}; we include its proof in the present notation.  It
transfers the iterated differences from the dual function to its summands in
the form required for degree lowering.

\begin{lemma}[The interchange lemma]\label{lem:DDI}
Let $r$ be a positive integer. Let $T$ be a finite set and let
$$
D=\sum_{t\in T} c_t\psi_t,
$$
where $|c_t|\le 1$, and each $\psi_t\colon \mathbb Z\to\mathbb C$ is $1$-bounded and supported in an interval of length at most $N$. Let $I_1,\ldots,I_r\subseteq\mathbb Z$ be intervals with $|I_i|=H$ for each $i$, and let
$$
\mathcal H\subseteq I_1\times\cdots\times I_r
$$
be any subset. Then, for every
$\phi:\mathbb Z^r\to\mathbb T$, one has
\begin{align}\nonumber
&\left(
\sum_{\mathbf h\in\mathcal H}
\left|
\sum_x \Delta_{\mathbf h}D(x)e(\phi(\mathbf h)x)
\right|
\right)^{2^r} \\
\label{eq:DDI-unnormalised}&\qquad\le
N^{2^r-1}|T|^{4^r-1}H^{r(2^r-2)}
\sum_{\mathbf h^0,\mathbf h^1\in\mathcal H}
\left|
\sum_x
\left(
\sum_{t\in T}
\Delta_{\mathbf h^1-\mathbf h^0}\psi_t(x)
\right)
e(\partial\phi(\mathbf h^0,\mathbf h^1)x)
\right|,
\end{align}
where
\begin{align}
\label{eq:partialdef}\partial\phi(\mathbf h^0,\mathbf h^1)
\coloneqq
\sum_{\omega\in\{0,1\}^r}
(-1)^{|\omega|}
\phi(\mathbf h^\omega),
\qquad
\mathbf h^\omega
\coloneqq
(h_1^{\omega_1},\dots,h_r^{\omega_r}).
\end{align}
\end{lemma}

\begin{proof}

We prove \eqref{eq:DDI-unnormalised} by induction on $r$.

\textbf{The case $r=1$.}
Let $\mathcal H\subseteq I$ with $|I|=H$, and let
$$
S
\coloneqq
\sum_{h\in \mathcal H}
\left|
\sum_x \Delta_hD(x)e(\phi(h)x)
\right|.
$$
For each $h\in \mathcal H$, choose $\beta_h\in\mathbb C$ with $|\beta_h|= 1$ such that
$$
S
=
\sum_{h\in \mathcal H}
\beta_h
\sum_x \Delta_hD(x)e(\phi(h)x).
$$
Expanding $D=\sum_{t\in T}c_t\psi_t$, we obtain
\begin{align*}
S
&=
\sum_x
\sum_{t\in T} c_t\psi_t(x)
\sum_{t'\in T}\overline{c_{t'}}
\sum_{h\in \mathcal H}
\beta_h\overline{\psi_{t'}(x+h)}e(\phi(h)x).
\end{align*}
By Cauchy--Schwarz in $(x,t,t')$, using that the functions $\psi_t$ are
$1$-bounded and supported on intervals of length at most $N$, we get
\begin{align*}
S^2
&\le
\left(
\sum_x\sum_{t\in T}\sum_{t'\in T}|\psi_t(x)|^2
\right)
\sum_x \sum_{t\in T} \sum_{t'\in T} 
\left|
\sum_{h\in \mathcal H}
\beta_h\overline{\psi_{t'}(x+h)}e(\phi(h)x)
\right|^2 \\
&\le
N|T|^3
\sum_x\sum_{t'\in T}
\left|
\sum_{h\in \mathcal H}
\beta_h\overline{\psi_{t'}(x+h)}e(\phi(h)x)
\right|^2.
\end{align*}
Expanding the square gives
\begin{align*}
S^2
&\le
N|T|^3
\sum_{h^0,h^1\in \mathcal H}
\beta_{h^0}\overline{\beta_{h^1}}
\sum_x\sum_{t'\in T}
\psi_{t'}(x+h^0)\overline{\psi_{t'}(x+h^1)}
e((\phi(h^0)-\phi(h^1))x) \\
&\le
N|T|^3
\sum_{h^0,h^1\in \mathcal H}
\left|
\sum_x\sum_{t\in T}
\psi_t(x+h^0)\overline{\psi_t(x+h^1)}
e((\phi(h^0)-\phi(h^1))x)
\right|.
\end{align*}
After the change of variables $x\mapsto x-h^0$, this becomes
$$
S^2
\le
N|T|^3
\sum_{h^0,h^1\in \mathcal H}
\left|
\sum_x
\left(
\sum_{t\in T}\Delta_{h^1-h^0}\psi_t(x)
\right)
e((\phi(h^0)-\phi(h^1))x)
\right|.
$$
Since for $r=1$ one has
$$
\partial\phi(h^0,h^1)=\phi(h^0)-\phi(h^1),
$$
this matches~\eqref{eq:DDI-unnormalised} in the case $r=1$.

\textbf{The induction step.}
Assume \eqref{eq:DDI-unnormalised} holds in dimension $r-1$. Write $\mathcal B'=I_1\times\cdots\times I_{r-1}$ and $I=I_r$, decompose $\mathbf h=(\mathbf h',h)$ with
$\mathbf h'\in\mathcal B'$ and $h\in I$, and for $h\in I$ let
$$
\mathcal H_h\coloneqq\{\mathbf h'\in\mathcal B':(\mathbf h',h)\in\mathcal H\}
$$
denote the corresponding slice of $\mathcal H$. For each $h\in I$, observe
$$
\Delta_h D(x)=
D(x)\overline{D(x+h)}
=
\sum_{t,t'\in T}
c_t\overline{c_{t'}}
\psi_t(x)\overline{\psi_{t'}(x+h)}.
$$
Then
$$
\Delta_{\mathbf h}D(x)
=
\Delta_{\mathbf h'}\Delta_h D(x).
$$
Moreover, $\Delta_h D$ is a linear combination indexed by $T^2$, and each function
$$
x\mapsto \psi_t(x)\overline{\psi_{t'}(x+h)}
$$
is $1$-bounded and supported in an interval of length at most $N$.

For each $h\in I$, set
$$
S_h
\coloneqq
\sum_{\mathbf h'\in\mathcal H_h}
\left|
\sum_x
\Delta_{\mathbf h'}\Delta_h D(x)e(\phi(\mathbf h',h)x)
\right|.
$$
By H\"older's inequality,
\begin{align}\label{eq:Holder}
\left(\sum_{h\in I}S_h\right)^{2^r}
\le
H^{2^r-2}
\left(
\sum_{h\in I} S_h^{2^{r-1}}
\right)^2.
\end{align}
Applying the induction hypothesis in dimension $r-1$ to $\Delta_h D$, with index set
$T^2$ and the subset $\mathcal H_h\subseteq\mathcal B'$, gives
\begin{align*}
S_h^{2^{r-1}}
&\le
N^{2^{r-1}-1}
|T|^{2(4^{r-1}-1)}
H^{(r-1)(2^{r-1}-2)} \sum_{\mathbf h'{}^0,\mathbf h'{}^1\in\mathcal H_h}|\Sigma_{\mathbf h'{}^0,\mathbf h'{}^1}|, 
\end{align*}
where
\begin{align*}
\Sigma_{\mathbf h'{}^0,\mathbf h'{}^1}=
\sum_x
\sum_{t,t'\in T}
\Delta_{\mathbf h'{}^1-\mathbf h'{}^0}
\bigl(\psi_t(x)\overline{\psi_{t'}(x+h)}\bigr)
e(\partial \phi_h(\mathbf h'{}^0,\mathbf h'{}^1)x),
\end{align*}
where
$$
\phi_h(\mathbf h')\coloneqq \phi(\mathbf h',h)
$$
and we use $\partial$ here to denote the alternating derivative in dimension
$r-1$ (cf.~\eqref{eq:partialdef}).

Using multiplicativity of $\Delta_{\mathbf h'{}^1-\mathbf h'{}^0}$, we see that
\begin{align*}
\Sigma_{\mathbf h'{}^0,\mathbf h'{}^1}=\sum_x
\left(
\sum_{t\in T}
\Delta_{\mathbf h'{}^1-\mathbf h'{}^0}\psi_t(x)
\right)
\overline{
\left(
\sum_{t'\in T}
\Delta_{\mathbf h'{}^1-\mathbf h'{}^0}\psi_{t'}(x+h)
\right)}
e(\partial\phi_h(\mathbf h'{}^0,\mathbf h'{}^1)x).
\end{align*}

Then we have
$$
\Sigma_{\mathbf h'{}^0,\mathbf h'{}^1}=\sum_x
\Delta_h\left(\sum_{t\in T}\Delta_{\mathbf h'{}^1-\mathbf h'{}^0}\psi_t\right)(x)
e(\partial\phi_h(\mathbf h'{}^0,\mathbf h'{}^1)x).
$$

Abbreviating the constant from the induction hypothesis by
$$
K\coloneqq N^{2^{r-1}-1}|T|^{2(4^{r-1}-1)}H^{(r-1)(2^{r-1}-2)},
$$
summing that bound over $h\in I$ and regrouping (note $\mathbf h'{}^0,\mathbf h'{}^1\in\mathcal H_h$ if and only if $h\in J_{\mathbf h'{}^0,\mathbf h'{}^1}\coloneqq\{h\in I:(\mathbf h'{}^0,h),(\mathbf h'{}^1,h)\in\mathcal H\}$) gives
$$
\sum_{h\in I}S_h^{2^{r-1}}
\le
K\sum_{\mathbf h'{}^0,\mathbf h'{}^1\in\mathcal B'}\ 
\sum_{h\in J_{\mathbf h'{}^0,\mathbf h'{}^1}}
\left|
\sum_x
\Delta_h\left(\sum_{t\in T}\Delta_{\mathbf h'{}^1-\mathbf h'{}^0}\psi_t\right)(x)
e(\partial\phi_h(\mathbf h'{}^0,\mathbf h'{}^1)x)
\right|.
$$
Squaring, applying Cauchy--Schwarz over the $H^{2(r-1)}$ pairs
$(\mathbf h'{}^0,\mathbf h'{}^1)$, and the $r=1$ base case to the inner sum for each fixed pair (with $\sum_{t\in T} \Delta_{\mathbf h'{}^1-\mathbf h'{}^0}\psi_t(x)$ in place of $D$, $h\mapsto\partial\phi_h(\mathbf h'{}^0,\mathbf h'{}^1)$ in place of $\phi$, and the subset $J_{\mathbf h'{}^0,\mathbf h'{}^1}\subseteq I$), we obtain
\begin{align*}
\left(\sum_{h\in I} S_h^{2^{r-1}}
\right)^2
&\le
K^2 H^{2(r-1)}N|T|^3
\sum_{\mathbf h'{}^0,\mathbf h'{}^1\in\mathcal B'}
\sum_{h^0,h^1\in J_{\mathbf h'{}^0,\mathbf h'{}^1}} \\
&\qquad\times
\left|
\sum_x
\left(
\sum_{t\in T}
\Delta_{h^1-h^0}
\Delta_{\mathbf h'{}^1-\mathbf h'{}^0}\psi_t(x)
\right)
e\bigl((\partial\phi_{h^0}(\mathbf h'{}^0,\mathbf h'{}^1)
-\partial\phi_{h^1}(\mathbf h'{}^0,\mathbf h'{}^1))x\bigr)
\right|.
\end{align*}
If $h^0,h^1\in J_{\mathbf h'{}^0,\mathbf h'{}^1}$, then $(\mathbf h'{}^0,h^0)$ and $(\mathbf h'{}^1,h^1)$ both lie in $\mathcal H$, so each term on the right side corresponds to a distinct pair $(\mathbf h^0,\mathbf h^1)\in\mathcal H^2$ with $\mathbf h^i=(\mathbf h'{}^i,h^i)$. Furthermore, observe that
$$
\partial\phi_{h^0}(\mathbf h'{}^0,\mathbf h'{}^1)
-
\partial\phi_{h^1}(\mathbf h'{}^0,\mathbf h'{}^1)
=
\partial\phi((\mathbf h'{}^0,h^0),(\mathbf h'{}^1,h^1))=\partial\phi(\mathbf h^0,\mathbf h^1).
$$
Combining this with~\eqref{eq:Holder} and writing $\mathbf h^i=(\mathbf h'{}^i,h^i)$ for $i=0,1$, we get
\begin{align*}
&\left(
\sum_{\mathbf h\in\mathcal H}
\left|
\sum_x \Delta_{\mathbf h}D(x)e(\phi(\mathbf h)x)
\right|
\right)^{2^r} \\
&\qquad\le
N^{2^r-1}|T|^{4^r-1}H^{r(2^r-2)}
\sum_{\mathbf h^0,\mathbf h^1\in\mathcal H}
\left|
\sum_x
\left(
\sum_{t\in T}
\Delta_{\mathbf h^1-\mathbf h^0}\psi_t(x)
\right)
e(\partial\phi(\mathbf h^0,\mathbf h^1)x)
\right|,
\end{align*}
since the powers combine as
\begin{align*}
    N^{2(2^{r-1}-1)}N&=N^{2^r-1},\\
    |T|^{4(4^{r-1}-1)}|T|^3&=|T|^{4^r-1},\\
    H^{2^r-2}
H^{2(r-1)(2^{r-1}-2)}
H^{2(r-1)}
&=
H^{r(2^r-2)}.
\end{align*}
This proves \eqref{eq:DDI-unnormalised}.
\end{proof}

The final ingredient is the following lemma, which connects the low rank structure of the phase to a lower degree Gowers norm.

\begin{lemma}[Low rank correlation implies lower degree]\label{lem:lowrank}
Let $N\geq1$, let $s\geq1$ and $0\leq m\leq s$ be integers, and let
$f\colon \mathbb{Z} \to \mathbb{C}$ be a $1$-bounded function with support in
$[N]$. Then for
$$
\phi_1,\ldots,\phi_m : \mathbb{Z}^{s-1} \to \mathbb{T},
$$
we have
\begin{equation}\label{eq:low-rank-bound}
\frac{1}{N^{s+1}}
\sum_{h_1,\ldots,h_s}
\left|
\sum_x
\Delta_{\mathbf h} f(x)\,
e\left(
\sum_{i=1}^{m}
\phi_i(h_1,\ldots,h_{i-1},h_{i+1},\ldots,h_s)\,x
\right)
\right|
\;\ll_s\;
\left(
\frac{\|f\|_{U^{s+1}(\mathbb{Z})}^{\,2^{s+1}}}{N^{s+2}}
\right)^{2^{-m-1}}.
\end{equation}
\end{lemma}

\begin{proof}
This is~\cite[Lemma 6.4]{PP}.
\end{proof}

\subsection{Reduction to \texorpdfstring{$U^2$}{U2}}

We now prove that a large $U^s$ norm of the dual function already forces a
large $U^2$ norm.
\begin{lemma}[Degree lowering]\label{lem:degreelower}
Let $s\geq3$ be an integer and let $C_0\geq1$.  There are constants
$c=c(s,k,C_0)>0$ and $C=C(s,k)\geq1$ such that the following holds.  Let
$0<\delta<c$, let $H$ be a positive integer, set $B=N/H$, and
suppose that
$\delta^{-C}\leq B\leq N\delta^C$. Let
$f_2,\ldots,f_{k}\colon\mathbb Z\to\mathbb C$ be $1$-bounded
functions supported on $[N]$, let $\lambda\colon\mathbb Z^{k-1}\to\mathbb C$
be $1$-bounded and supported on
$([-C_0B,C_0B]\cap\mathbb Z)^{k-1}$, and define
$$
    D(x)
    =1_{x\in [N]}
    \sum_{\mathbf b\in\mathbb Z^{k-1}}
    \lambda(\mathbf b)
    \sum_{m\in[H]}
    \prod_{i=2}^{k} f_i(x+b_i m).
$$
If 
\begin{equation}\label{eq:degree-lowering-hypothesis}
    \|D\|_{U^s[N]}
    \geq
    \delta B^{k-2}N,
\end{equation}
then
\begin{equation}\label{eq:degree-lowering-conclusion}
    \|D\|_{U^2[N]}
    \geq
    c\delta^CB^{k-2}N.
\end{equation}
\end{lemma}

\begin{proof}
Write
$$
    M=B^{k-2}N,
$$
so that $\|D\|_{U^j[N]}=\|D\|_{U^j(\mathbb Z)}/\|1_{[N]}\|_{U^j(\mathbb Z)}$ for every $j$, and note that
$|D(x)|\leq K_0M$ for all $x$, where $K_0=(2C_0+1)^{k-1}$.
Set
\[
    \mathcal B\coloneqq ([-C_0B,C_0B]\cap\mathbb Z)^{k-1}.
\]
Extending the coefficient list by zeros, we may write
$$
    D=\sum_{(\mathbf b,m)\in\mathcal B\times[H]}\lambda(\mathbf b)\,\psi_{\mathbf b,m},
    \qquad
    \psi_{\mathbf b,m}\coloneqq 1_{[N]}\cdot\prod_{i=2}^{k}f_i(\cdot+b_i m),
$$
where each $\psi_{\mathbf b,m}$ is $1$-bounded and supported in an interval of length at most $N$.

There are constants $a=a(s,k,C_0)>0$ and $A=A(s,k)\geq1$ for which the
argument below proves the implication
\begin{equation}\label{eq:one-step-lowering}
    \|D\|_{U^s[N]}
    \geq
    \delta M
    \quad\Longrightarrow\quad
    \|D\|_{U^{s-1}[N]}
    \geq
    a\delta^AM,
\end{equation}
and then we can iterate it until the $U^2[N]$ norm is reached.  We choose $A$
independently of $C_0$ and absorb all multiplicative constants depending
on $C_0$ into $a$.

By the recursion~\eqref{eq:gowers-recursion} applied $s-2$ times,
$$
    \|D\|_{U^s(\mathbb Z)}^{2^s}
    = \sum_{\mathbf h\in\mathbb Z^{s-2}}\|\Delta_{\mathbf h}D\|_{U^2(\mathbb Z)}^4 .
$$
By \eqref{eq:degree-lowering-hypothesis} and \eqref{eq:Gowersdef}, the left side is $\gg_s(\delta M)^{2^s}N^{s+1}$. Each summand vanishes unless $\mathbf h\in([-N,N]\cap\mathbb Z)^{s-2}$, and is $O_{s,k,C_0}(M^{2^s}N^3)$ since $|\Delta_{\mathbf h}D|\leq (K_0M)^{2^{s-2}}$ pointwise. A popularity argument therefore gives $\gg_{s,k,C_0} \delta^{2^s}N^{s-2}$ tuples $\mathbf h$ with
$$
    \|\Delta_{\mathbf h}D\|_{U^2(\mathbb Z)}^{4}
    \gg_{s,k,C_0}\delta^{2^s}M^{2^s}N^3.
$$
For each such $\mathbf h$ we apply Lemma~\ref{lem:U2inv} to the $1$-bounded
function $\Delta_{\mathbf h}D/(K_0M)^{2^{s-2}}$, with
$\delta^{2^{s-2}}$ (up to a constant depending on $s,k,C_0$) in place of
$\delta$ and with
\[
    Q=\left\lceil20N\delta^{-A}\right\rceil.
\]
For the remaining $\mathbf h$ we set
$\phi(\mathbf h)=0$. This yields a function
$$
    \phi\colon\mathbb Z^{s-2}\to\mathbb T,
    \qquad
    \phi(\mathbf h)\in Q^{-1}\mathbb Z/\mathbb Z,
$$
and a subset $\mathcal{H}\subset ([-N,N]\cap\mathbb Z)^{s-2}$ of size
$\gg_{s,k,C_0}\delta^{O_{s,k}(1)}N^{s-2}$ such that 
\begin{align}\label{eq:large-alpha-correlation}
\left|
    \sum_x\Delta_{\mathbf h}D(x)
        e\bigl(\phi(\mathbf h)x\bigr)
    \right|
    \gg_{s,k,C_0}
    \delta^{O_{s,k}(1)} M^{2^{s-2}}N
    \qquad \textnormal{for all}\qquad \mathbf{h}\in \mathcal{H}.
\end{align}

Set $n=\lfloor N\rfloor$.  We now apply Lemma~\ref{lem:DDI}, with
$r=s-2$, phase $\phi(\mathbf h)$, and the subset $\mathcal H$ of
$([-n,n]\cap\mathbb Z)^{s-2}$, so with $2n+1$ in the role of $H$ there, to
the representation of $D$ displayed above. The index set
$T=\mathcal B\times[H]$ has size
$|T|\ll_{k,C_0} B^{k-1}H\ll_{k,C_0} M$, and the coefficients are the
$1$-bounded values $\lambda(\mathbf b)$, including zeros.
Summing~\eqref{eq:large-alpha-correlation} over
$\mathbf{h}\in \mathcal{H}$, raising to the power $2^{s-2}$ and applying
Lemma~\ref{lem:DDI} gives
\begin{align}
    \sum_{\mathbf h^0,\mathbf h^1\in\mathcal{H}}
    \left|
        \sum_x
        \Bigl(\sum_{\mathbf b,m}
        \Delta_{\mathbf h^1-\mathbf h^0}\psi_{\mathbf b,m}(x)\Bigr)
        e\bigl(\partial\phi(\mathbf h^0,\mathbf h^1)x\bigr)
    \right|
    \gg_{s,k,C_0}
    \delta^{O_{s,k}(1)}MN^{2s-3}.
    \label{eq:after-ddi-lambda-removed}
\end{align}
By multiplicativity of the difference operators,
$$
    \Delta_{\mathbf u}\psi_{\mathbf b,m}(x)
    =1_{E_{\mathbf u}}(x)\prod_{i=2}^{k}\Delta_{\mathbf u}f_i(x+b_i m),
    \qquad
    E_{\mathbf u}\coloneqq\bigcap_{\omega\in\{0,1\}^{s-2}}\bigl([N]-\omega\cdot\mathbf u\bigr),
$$
where $E_{\mathbf u}$ is an interval depending only on $\mathbf u=\mathbf h^1-\mathbf h^0$.  For a fixed pair
$(\mathbf h^0,\mathbf h^1)$, denote the absolute value in
\eqref{eq:after-ddi-lambda-removed} by
\begin{align*}
    S_{\mathbf h^0,\mathbf h^1}
    \coloneqq
    \left|
        \sum_{\mathbf b\in\mathcal B}
        \sum_{m\in[H]}\sum_{x\in E_{\mathbf u}}
        \prod_{i=2}^{k}
        \Delta_{\mathbf u}f_i(x+b_i m)
        e\bigl(\partial\phi(\mathbf h^0,\mathbf h^1)x\bigr)
    \right|.
\end{align*}
Reparametrise this sum by first substituting $x\mapsto x-b_2m$ and then
replacing $b_j-b_2$ by $b_j$ for $3\leq j\leq k$.  The factor
$\Delta_{\mathbf u}f_2(x)$ is then independent of $m$ and restricts $x$ to
$[N]$, while the phase becomes
\[
 e\bigl(\partial\phi(\mathbf h^0,\mathbf h^1)x\bigr)
 e\bigl(-\partial\phi(\mathbf h^0,\mathbf h^1)b_2m\bigr).
\]
For fixed $x,m,b_3,\ldots,b_k$, the admissible values of $b_2$ are the
intersection of the interval imposed by $x-b_2m\in E_{\mathbf u}$ with the
finitely many interval constraints defining the original cube $\mathcal B$.
They therefore form an interval $J$ of length $O_{k,C_0}(B)$.  The sum over $b_2$ has
remained inside the modulus, so
\[
 \left|\sum_{b_2\in J}e(\theta b_2)\right|
 \leq\min\bigl(|J|,\tfrac12\|\theta\|^{-1}\bigr).
\]
Bounding the remaining factors by $1$ and applying the triangle inequality
to the other variables gives
\begin{align}
    S_{\mathbf h^0,\mathbf h^1}
    \ll_{k,C_0}
    \sum_{x\in[N]}
    \sum_{b_3,\ldots,b_{k}}
    \sum_{m\in[H]}
    \min\Bigl(B,\ \bigl\|\partial\phi(\mathbf h^0,\mathbf h^1)m\bigr\|^{-1}\Bigr).
    \label{eq:only-b1-phase}
\end{align}

Combining \eqref{eq:after-ddi-lambda-removed} and
\eqref{eq:only-b1-phase}, and summing trivially over the $O(N)$ values of
$x$ and the $O_{k,C_0}(B^{k-2})$ values of $b_3,\ldots,b_{k}$, we obtain
\begin{equation}\label{eq:large-geometric-average}
    \sum_{\mathbf h^0,\mathbf h^1\in\mathcal H}
    \sum_{m\in[H]}
    \min\Bigl(B,\ \bigl\|\partial\phi(\mathbf h^0,\mathbf h^1)m\bigr\|^{-1}\Bigr)
    \gg_{s,k,C_0}
    \delta^{O_{s,k}(1)}|\mathcal H|^2HB.
\end{equation}
Since the summand is at most $B$, it follows that for a proportion
$\gg_{s,k,C_0}\delta^{O_{s,k}(1)}$ of pairs
$(\mathbf h^0,\mathbf h^1)\in\mathcal H^2$, one has
\begin{equation}\label{eq:large-geometric-for-pairs}
    \sum_{m\in[H]}
    \min\Bigl(B,\ \bigl\|\partial\phi(\mathbf h^0,\mathbf h^1)m\bigr\|^{-1}\Bigr)
    \gg_{s,k,C_0}
    \delta^{O_{s,k}(1)}HB.
\end{equation}
Fix constants $a'=a'(s,k,C_0)>0$ and $A'=A'(s,k)\geq1$ so that the
right side of~\eqref{eq:large-geometric-for-pairs} is at least
$\varepsilon HB$, where
$\varepsilon=a'\delta^{A'}$.  If fewer than
$(\varepsilon/2)H$ values of $m$ satisfy
$\min(B,\|\partial\phi(\mathbf h^0,\mathbf h^1)m\|^{-1})
\geq(\varepsilon/2)B$, then the sum in that display is less than
$\varepsilon HB$.  Hence at least $(\varepsilon/2)H$ values satisfy
\begin{equation}\label{eq:omega-m-small}
    \|\partial\phi(\mathbf h^0,\mathbf h^1)m\|
    \leq \frac{2}{\varepsilon B}.
\end{equation}

By Vinogradov's lemma (\cite[Lemma~4.5]{GT-nil}) using
$\min(B,H)\geq\delta^{-O_{s,k}(1)}$ from the range hypothesis, for every pair
$(\mathbf h^0,\mathbf h^1)$ in this proportion there is an integer
$1\leq q\ll_{s,k,C_0}\delta^{-O_{s,k}(1)}$ such that
\begin{equation}\label{eq:qomega-small}
    \|q\partial\phi(\mathbf h^0,\mathbf h^1)\|
    \ll_{s,k,C_0}
    \frac{\delta^{-O_{s,k}(1)}}{N}.
\end{equation}
Pigeonholing in $q$, there is a single integer
$1\leq q\ll_{s,k,C_0}\delta^{-O_{s,k}(1)}$ for which
\eqref{eq:qomega-small} holds for a proportion
$\gg_{s,k,C_0}\delta^{O_{s,k}(1)}$ of pairs in $\mathcal H^2$.
For this fixed $q$, each value of $\partial\phi(\mathbf h^0,\mathbf h^1)$
occurring in this proportion lies within
$O_{s,k,C_0}(\delta^{-O_{s,k}(1)}/N)$ of one
of the $q$ rationals $a/q$, $a\in\mathbb Z/q\mathbb Z$.  Since
$q\ll_{s,k,C_0}\delta^{-O_{s,k}(1)}$, another pigeonhole step gives a fixed residue
$a\in\mathbb Z/q\mathbb Z$ such that 
\begin{equation}\label{eq:omega-near-fixed-rational}
    \left\|
        \partial\phi(\mathbf h^0,\mathbf h^1)-\frac{a}{q}
    \right\|
    \ll_{s,k,C_0}
    \frac{\delta^{-O_{s,k}(1)}}{N}
\end{equation}
for a proportion $\gg_{s,k,C_0}\delta^{O_{s,k}(1)}$ of pairs
$(\mathbf h^0,\mathbf h^1)\in\mathcal H^2$.  Since
$\partial\phi$ takes values in the finite grid
$Q^{-1}\mathbb Z/\mathbb Z$, the arc in
\eqref{eq:omega-near-fixed-rational} contains
\[
    O_{s,k,C_0}\left(
        1+\frac{Q\delta^{-A_1}}{N}
    \right)
    \leq C_1\delta^{-A_2}
\]
grid points, for fixed constants $A_1,A_2$ depending only on $s,k$ and
$C_1$ depending on $s,k,C_0$.  Pigeonholing once more
among these points, we obtain a constant
$c_0\in\mathbb T$ such that
\begin{equation}\label{eq:partial-alpha-constant}
    \partial\phi(\mathbf h^0,\mathbf h^1)=c_0
\end{equation}
for a proportion $\gg_{s,k,C_0}\delta^{O_{s,k}(1)}$ of pairs
$(\mathbf h^0,\mathbf h^1)\in\mathcal H^2$.

Define $\phi^\ast\colon\mathbb Z^{s-2}\to\mathbb T$ by
\[
 \phi^\ast(\mathbf v)
 \coloneqq \phi\bigl(\mathbf v-(n+1)\mathbf 1\bigr).
\]
Translation preserves the alternating difference $\partial\phi$, so
\eqref{eq:partial-alpha-constant} gives
$\partial\phi^\ast=c_0$ on a subset of
$[2n+1]^{2(s-2)}$ of density
$\gg_{s,k,C_0}\delta^{O_{s,k}(1)}$.
Lemma~\ref{lem:lowcomplexity}, applied to $\phi^\ast$ with $s-2$ in place
of $s$, now shows that $\phi^\ast$ has low rank on a set
$\mathcal H^\ast\subset[2n+1]^{s-2}$ contained in
$\mathcal H+(n+1)\mathbf1$.  Translating back gives a set
$\mathcal H'=\mathcal H^\ast-(n+1)\mathbf1\subseteq\mathcal H$ of size
$|\mathcal H'|\gg_{s,k,C_0}\delta^{O_{s,k}(1)}N^{s-2}$. Thus there exist functions
$\phi_1,\ldots,\phi_{s-2}$, where $\phi_j$ is independent of
coordinate $j$, such that
\begin{equation}\label{eq:alpha-low-rank}
    \phi(\mathbf h)
    =
    \phi_1(\widehat{\mathbf h}_1)+\cdots+
    \phi_{s-2}(\widehat{\mathbf h}_{s-2})
\end{equation}
for all $\mathbf h\in\mathcal H'$, where $\widehat{\mathbf h}_i=(h_1,\ldots, h_{i-1},h_{i+1},\ldots, h_{s-2})$.  Returning to
\eqref{eq:large-alpha-correlation}, restricting to $\mathcal H'\subseteq\mathcal H$ and enlarging the resulting sum to all $\mathbf h\in\mathbb Z^{s-2}$, we find that Lemma~\ref{lem:lowrank}, applied to the $1$-bounded function $D/(K_0M)$ with $s-2$ in place of $s$ and with $m=s-2$ phase functions, gives
$$
    \|D\|_{U^{s-1}[N]}
    \geq a\delta^AM.
$$
This proves \eqref{eq:one-step-lowering}.  Iterating
\eqref{eq:one-step-lowering} from degree $s$ down to degree $2$,  gives
$$
    \|D\|_{U^2[N]}
    \geq c\delta^CB^{k-2}N,
$$
as required.
\end{proof}

\subsection{Reduction to \texorpdfstring{$U^{1+}$}{U1+} norm}

To pass from $U^2[N]$ norm to $U^{1+}[N]$ norm we use the inverse theorem for the $U^2[N]$ norm (which has polynomial dependencies) and then observe that thecorrelation of $D$ with a linear phase can be written as a bilinear exponential sum. This shows that  the
frequency $\alpha$ supplied by the $U^2$ inverse theorem must have many multiples $u$ with $\|u\alpha\|$ small, which implies that $\alpha$ is major arc, giving reduction to the $U^{1+}[N]$ norm. However, applying a standard bilinear estimate would lose a logarithm in the density, which we cannot afford. To overcome this, we smooth the sum in one of the variables and use Poisson summation to get more rapid decay away from the major arcs.

\begin{lemma}[Smoothed bilinear estimate]\label{lem:bilinear}
Let $U,V\geq 1$, let $A\geq 2$, and let $\phi,\beta\colon \mathbb{Z}\to \mathbb{C}$ be $1$-bounded sequences supported on intervals of lengths at most $U$ and $V$, respectively. Then for every $\theta\in\mathbb{R}$,
\begin{align}\label{eq:smoothed-bilinear-count}
  \Bigl|\sum_{m}\sum_{n}\phi(m)\beta(n)e(\theta mn)\Bigr|^2 \ll_A  U^2V\sum_{0\leq |u|\leq V}\bigl(1+U\|\theta u\|\bigr)^{-A}.
\end{align}
\end{lemma}

\begin{proof}
Choose a smooth function $W\colon\mathbb{R}\to[0,1]$ with $W=1$ on $[0,1]$ and $\operatorname{supp}W\subseteq[-1,2]$.
Let $m_0\in\mathbb Z$ and let $J_U=[m_0,m_0+U]$ be an interval containing
the support of $\phi$.  Set $w_m=W((m-m_0)/U)$, so that $w_m=1$ for
$m\in J_U$. By the Cauchy--Schwarz inequality with the weight $w$,
\begin{align*}
  \Bigl|\sum_m\sum_n\phi(m)\beta(n)e(\theta mn)\Bigr|^2
  &\leq
  \Bigl(\sum_{m\in J_U}\frac{|\phi(m)|^2}{w_m}\Bigr)
  \sum_m w_m\Bigl|\sum_n\beta(n)e(\theta mn)\Bigr|^2 \\
  &\ll
  U\sum_{n,n'}|\beta(n)\beta(n')|
  \left|\sum_m w_m e\bigl(\theta m(n-n')\bigr)\right|.
\end{align*}
By Poisson summation, for any $\xi\in\mathbb{R}$,
\begin{align*}
  \sum_{m\in\mathbb{Z}}W\Bigl(\frac{m-m_0}{U}\Bigr)e(\xi m)
  =e(\xi m_0)\,U\sum_{j\in\mathbb{Z}}\widehat W\bigl(U(j-\xi)\bigr),
  \qquad
  \widehat W(y)\coloneqq\int_{\mathbb{R}} W(t)e(-yt)\,\d t,
\end{align*}
and since $W$ is smooth and compactly supported, $|\widehat W(y)|\ll_A(1+|y|)^{-A-2}$ by repeated integration by parts. As $U|j-\xi|\geq U\|\xi\|$ and $U\geq1$, summing over $j$ gives
\begin{align*}
  \Bigl|\sum_{m\in\mathbb{Z}}W\Bigl(\frac{m-m_0}{U}\Bigr)e(\xi m)\Bigr|\ll_A U\bigl(1+U\|\xi\|\bigr)^{-A}.
\end{align*}
Inserting this with $\xi=\theta(n-n')$, writing $u=n-n'$ and noting that each value of $u$ arises from at most $V+1$ pairs $(n,n')$, we obtain~\eqref{eq:smoothed-bilinear-count}.
\end{proof}

\begin{lemma}[Degree lowering to $U^{1+}$ norm]\label{lem:dualU1+}
Let $C_0\geq1$.  There are constants $c=c(k,C_0)>0$ and
$C=C(k)\geq1$ such that the following holds.  Let $0<\delta<c$, let $H$
be a positive integer, set $B=N/H$, suppose that
$\delta^{-C}\leq B\leq N\delta^C$, and let $D$ be as in
Lemma~\ref{lem:degreelower} with support parameter $C_0$. If
    \begin{align*}
           \|D\|_{U^{2}[N]}\geq \delta B^{k-2}N,
    \end{align*}
    then
    \begin{align*}
        \|D\|_{U^{1+}[N]}\geq c\delta^C B^{k-2}N.
    \end{align*}
\end{lemma}

\begin{proof}
Write $M=B^{k-2}N$, put
$K_0=(2C_0+1)^{k-1}$, and let $C_1$ be a sufficiently large constant
depending only on $k$. We take $C$ larger than $4+C_1$.
By decreasing $c=c(k,C_0)$ if necessary, all multiplicative constants
depending on $C_0$ may be absorbed without changing either $C_1$ or $C$.

Choose an integer $Q\geq20N(K_0/\delta)^2$.  Lemma~\ref{lem:U2inv}, applied
to the $1$-bounded function $D/(K_0M)$ with this value of $Q$, gives
$\alpha\in\mathbb T$ such that
\begin{align}\label{eq:U1plus-start}
      \Bigl| \sum_{x} D(x) e(\alpha x)\Bigr|\gg_{k,C_0} \delta^2 MN.
\end{align}
Opening the definition of $D$, substituting $x\mapsto x-b_2m$ and replacing $b_j-b_2$ by $b_j$ for $j=3,\ldots,k$ as in the proof of Lemma~\ref{lem:degreelower}, define
\[
 \lambda'(b_2,b_3,\ldots,b_k)
 \coloneqq \lambda(b_2,b_3+b_2,\ldots,b_k+b_2).
\]
This weight is $1$-bounded and supported on a cube of side
$O_{k,C_0}(B)$.  With
this notation the left side of~\eqref{eq:U1plus-start} equals
\begin{align*}
    \Bigl|\sum_{b_3,\ldots,b_{k}}\sum_{m\in[H]}\sum_{b_2}\lambda'(b_2,\ldots,b_{k})\,
    e(-\alpha b_2 m)
    \sum_{x\in[N]+b_2m}e(\alpha x)f_2(x)\prod_{i=3}^{k}f_i(x+b_im)\Bigr|,
\end{align*}
where the condition $x\in[N]+b_2m$ records the support of the removed factor $1_{[N]}(x-b_2m)$, and $f_2$ restricts $x$ to $[N]$.

Set
\[
    L_B=\max(1,\lfloor\delta^{C_1}B\rfloor),\qquad
    L_H=\max(1,\lfloor\delta^{C_1}H\rfloor).
\]
The range hypotheses ensure that
$L_B\asymp\delta^{C_1}B$ and $L_H\asymp\delta^{C_1}H$.  Split the ranges of
$b_2$ and $m$ into $O(\delta^{-2C_1})$ blocks
$J_B\times J_H$ of cardinalities at most $L_B$ and $L_H$, respectively.
On each block the product $b_2m$ varies by $O_{k,C_0}(\delta^{C_1}N)$, so
replacing $[N]+b_2m$ by $[N]+u_0$, with $u_0$ a fixed reference value of
$b_2m$ on the block, changes each inner sum over $x$ by
$O_{k,C_0}(\delta^{C_1}N)$ and hence the total by
$O_{k,C_0}(\delta^{C_1}MN)$, which is negligible for $C_1$ large. Fixing the
tuple $(b_3,\ldots,b_{k})$ and then the block giving the largest contribution,
of which there are $O_{k,C_0}(B^{k-2})$ and $O(\delta^{-2C_1})$ respectively,
we obtain, suppressing the fixed $b_3,\ldots,b_{k}$ from $\lambda'$,
\begin{align}\label{eq:block-sum-large}
    \Bigl|\sum_{(b_2,m)\in J_B\times J_H}\lambda'(b_2)\,\frac{\rho(m)}{N}\,e(-\alpha b_2m)\Bigr|
    \gg_{k,C_0} \delta^{2+2C_1}
    \frac{MN}{B^{k-2}N}=\delta^{2+2C_1} N,
\end{align}
where, for the fixed $(b_3,\ldots,b_{k})$ and the reference value $u_0$ of the chosen block,
$$
    \rho(m)\coloneqq\sum_{x\in[N]\cap([N]+u_0)}e(\alpha x)f_2(x)\prod_{i=3}^{k}f_i(x+b_im)
$$
is independent of $b_2$ and satisfies $|\rho(m)|\leq N$.

We now apply Lemma~\ref{lem:bilinear}
with $\theta=-\alpha$, $U=L_B$, $V=L_H$, the $1$-bounded sequences
$\lambda'$ and $\rho/N$ supported on the intervals $J_B$ and $J_H$, and
$A=10$. Comparing with~\eqref{eq:block-sum-large} and using $BH=N$, we obtain
\begin{align}\label{eq:count-large}
    \sum_{0\leq|u|\leq \delta^{C_1}H}\bigl(1+\delta^{C_1}B\|\alpha u\|\bigr)^{-10}
    \gg_{k,C_0}
    \frac{(\delta^{2+2C_1}N)^2}{(\delta^{C_1}B)^2\,\delta^{C_1}H}
    = \delta^{4}\cdot\delta^{C_1}H.
\end{align}
The terms in~\eqref{eq:count-large} with $\|\alpha u\|>\delta^{-1-C_1}/B$ contribute $O(\delta^{10}\cdot\delta^{C_1}H)$, which is negligible.  The term $u=0$ contributes
$1\leq\tfrac12\delta^{4+C_1}H$, since the range hypothesis gives
$H\geq\delta^{-C}$, provided $C>4+C_1$. Hence, after reflecting
$u\mapsto -u$, there are
$\gg_{k,C_0}\delta^{4}\cdot\delta^{C_1}H$ values of
$1\leq u\leq \delta^{C_1}H$ with
$$
    \|\alpha u\|\leq \delta^{-1-C_1}/B.
$$
By Vinogradov's lemma, \cite[Lemma~4.5]{GT-nil}, there is an integer
$1\leq q\ll_{k,C_0}\delta^{-O_k(1)}$ with
$$
    \|q\alpha\|\ll_{k,C_0}
    \frac{\delta^{-O_k(1)}}{B\cdot\delta^{C_1}H}
    \ll_{k,C_0}\frac{\delta^{-O_k(1)}}{N}.
$$
Write $\alpha=a/q+\beta$ with $a\in\mathbb{Z}$ and
$|\beta|\ll_{k,C_0}\delta^{-O_k(1)}/N$.

Finally, we extract a progression. Partition $[N]$ into arithmetic progressions of the form
$$
    P_{r,j}=\{x\in I_j: x\equiv r\ (\mathrm{mod}\ q)\},
    \qquad r\in\mathbb{Z}/q\mathbb{Z},
$$
where the $I_j$ are intervals of length
\[
 \ell\coloneqq
 \max\left(1,\left\lfloor
 \delta^{C_1}\min(|\beta|^{-1},N)
 \right\rfloor\right),
\]
with the convention $|0|^{-1}=\infty$, covering $[N]$. On each $P_{r,j}$ the phase $e((a/q)x)=e(ar/q)$ is constant and $e(\beta x)=e(\beta x_j)+O(\delta^{C_1})$ for any fixed $x_j\in I_j$. Hence, by~\eqref{eq:U1plus-start},
$$
    \delta^{2}MN
    \ll_{k,C_0}
    \sum_{r,j}\Bigl|\sum_{x\in P_{r,j}}D(x)\Bigr|
    +\delta^{C_1}\sum_x|D(x)|,
$$
and the second term is negligible for $C_1$ large. The number of pairs
$(r,j)$ is at most
$q\,(N/\ell+1)\ll_{k,C_0}\delta^{-O_k(1)}$, so for some single
progression $P=P_{r,j}\subseteq[N]$ we have
$$
    \Bigl|\sum_{x\in P}D(x)\Bigr|
    \gg_{k,C_0}\delta^{O_k(1)}MN,
$$
which by the definition~\eqref{eq:U1-plus-def} of the $U^{1+}[N]$ norm gives
$\|D\|_{U^{1+}[N]}\gg_{k,C_0}\delta^{O_k(1)}M$.  We now choose the constant $C$
in the statement larger than $4+C_1$ and than the finitely many implicit exponents
above.  This proves the result.
\end{proof}

\section{Proof of Theorem~\ref{thm:1+}}\label{sec:proof-1+}

The goal of this section is to prove Theorem~\ref{thm:1+} by turning a large
average involving the normalised dual function into correlations of the input
functions with arithmetic progressions.  The anchoring identity below allows
us to replace one input slot at a time by the indicator of an arithmetic
progression, and after iterating this procedure the remaining correlation is
detected directly by the $U^{1+}$ norm.

\begin{definition}[Normalised dual function]\label{def:dual}
Let $N\geq1$ and $C_0\geq1$, let $H$ be a positive integer, set $B=N/H$,
and suppose $B\geq1$.  For a $1$-bounded weight
$\mu\colon\Z^{k-1}\to\C$ supported on
$([-C_0B,C_0B]\cap\Z)^{k-1}$, and for
$g_2,\ldots,g_k\colon\Z\to\C$, set
\begin{align}\label{eq:normalised-dual-function}
    \mathcal D_\mu(g_2,\ldots,g_k)(x)
    \coloneqq
    \frac{1}{M}
    \sum_{\mathbf b\in\Z^{k-1}}\mu(\mathbf b)
    \sum_{m\in[H]}\prod_{i=2}^k g_i(x+b_i m),
    \qquad
    M\coloneqq B^{k-2}N=B^{k-1}H .
\end{align}
Here $\mathbf b=(b_2,\ldots,b_k)$.
\end{definition}

Thus $M\,1_{[N]}\mathcal D_\mu$ is the unnormalised dual function denoted by
$D$ in Section~\ref{sec:degree-lowering}.  We use the calligraphic notation
from this point onwards so that both the normalisation and the coefficient
weight remain visible.

The normalisation makes the dual function pointwise bounded.  More precisely,
if $\mu$ is supported on
$([-C_0B,C_0B]\cap\Z)^{k-1}$ and the functions $u_2,\ldots,u_k$ are
$1$-bounded and supported on $[N]$, then
\begin{align}\label{eq:normalised-dual-pointwise-bound}
    |\mathcal D_\mu(u_2,\ldots,u_k)(x)|
    \leq (2C_0+1)^{k-1}
    \quad\text{for all }x\in\Z .
\end{align}
Indeed, the numerator in
\eqref{eq:normalised-dual-function} is at most
$(2C_0B+1)^{k-1}H$, while $M=B^{k-1}H$.

The following identity records the coefficient weight whenever the outside
slot is changed.

\begin{lemma}[Anchoring]\label{lem:anchoring}
Let $2\leq\ell\leq k$, let $C_0\geq1$, and let
$\mu\colon\Z^{k-1}\to\C$ be $1$-bounded and supported on
$([-C_0B,C_0B]\cap\Z)^{k-1}$.  Then there is a $1$-bounded
$\widetilde\mu\colon\Z^{k-1}\to\C$, supported on
$([-2C_0B,2C_0B]\cap\Z)^{k-1}$, such that for
all $h,g_2,\ldots,g_k\colon\Z\to\C$ supported on $[N]$,
\begin{align}\label{eq:anchoring}
    \E_{x\in[N]}h(x)\,\mathcal D_\mu(g_2,\ldots,g_k)(x)
    =
    \E_{y\in[N]}g_\ell(y)\,
    \mathcal D_{\widetilde\mu}(h,g_2,\ldots,g_{\ell-1},g_{\ell+1},\ldots,g_k)(y).
\end{align}
Explicitly, $\widetilde\mu$ is the pushforward of $\mu$ under the map
\begin{align}\label{eq:anchoring-tuple}
    T\colon(b_2,\ldots,b_k)\longmapsto
    (-b_\ell,\ b_2-b_\ell,\ \ldots,\ b_{\ell-1}-b_\ell,\ b_{\ell+1}-b_\ell,\ \ldots,\ b_k-b_\ell),
\end{align}
which is a bijection of $\Z^{k-1}$.
\end{lemma}

\begin{proof}
Writing out the normalised dual function \eqref{eq:normalised-dual-function}, and using
that $h$ is supported on $[N]$,
\begin{align}\label{eq:move-sparse-out-expanded}
    &\E_{x\in[N]}h(x)\mathcal D_\mu(g_2,\ldots,g_k)(x)\notag\\
    &\quad =
    \frac1{\lfloor N\rfloor M}
    \sum_{\mathbf b}\mu(\mathbf b)
    \sum_{m\in[H]}
    \sum_{x\in\Z}
    h(x)g_\ell(x+b_\ell m)
    \prod_{\substack{2\leq i\leq k\\ i\neq \ell}}g_i(x+b_i m)\\
    &\quad =
    \frac1{\lfloor N\rfloor M}
    \sum_{\mathbf b}\mu(\mathbf b)
    \sum_{m\in[H]}
    \sum_{y\in\Z}
    g_\ell(y)h(y-b_\ell m)
    \prod_{\substack{2\leq i\leq k\\ i\neq \ell}}g_i(y+(b_i-b_\ell)m),
    \notag
\end{align}
by the substitution $y=x+b_\ell m$, under which $x=y-b_\ell m$ and
$x+b_im=y+(b_i-b_\ell)m$.

The map $T$ of \eqref{eq:anchoring-tuple} is a bijection of $\Z^{k-1}$: its first coordinate
determines $b_\ell$, and then the remaining coordinates determine the remaining $b_i$.  Set
$\widetilde\mu=\mu\circ T^{-1}$, which is $1$-bounded, and which is supported on
$([-2C_0B,2C_0B]\cap\Z)^{k-1}$ because $|-b_\ell|\leq C_0B$ and
$|b_i-b_\ell|\leq2C_0B$
whenever $\mu(\mathbf b)\neq0$.  Reindexing the sum over $\mathbf b$ in
\eqref{eq:move-sparse-out-expanded} by $\widetilde{\mathbf b}=T(\mathbf b)$, the last line of
\eqref{eq:move-sparse-out-expanded} becomes
\begin{align*}
    \frac1{\lfloor N\rfloor}
    \sum_{y\in\Z}g_\ell(y)\,
    \mathcal D_{\widetilde\mu}(h,g_2,\ldots,g_{\ell-1},g_{\ell+1},\ldots,g_k)(y),
\end{align*}
the slots of the new dual function carrying, in order, the functions
$h,g_2,\ldots,g_{\ell-1},g_{\ell+1},\ldots,g_k$ with the coefficients listed in
\eqref{eq:anchoring-tuple}.  Since $g_\ell$ is supported on $[N]$, this is the right side
of \eqref{eq:anchoring}.
\end{proof}

The next lemma carries out the basic replacement step: a large correlation
with the dual function allows any chosen input slot to be replaced by the
indicator of an arithmetic progression.

\begin{lemma}[Bounded slot replacement]\label{lem:bounded-slot-replacement}
For every $C_0\geq1$ there are constants $A=A(k)\geq1$ and
$c=c(k,C_0)>0$ such that the following holds.  Let $0<\kappa<1$, let
$g_1,\ldots,g_k$ be $1$-bounded and supported on $[N]$, and let $\mu$ be a
$1$-bounded coefficient weight supported on
$([-C_0B,C_0B]\cap\mathbb Z)^{k-1}$.  Assume
\[
    \kappa^{-A}\leq B\leq N\kappa^A.
\]
If
\begin{align}\label{eq:claim-hyp}
 \left|\E_{x\in[N]}g_1(x)\mathcal D_\mu(g_2,\ldots,g_k)(x)\right|
 \geq\kappa,
\end{align}
then for every $2\leq i\leq k$ there exist an arithmetic progression
$P\subseteq[N]$ and a $1$-bounded weight $\widetilde\mu$, supported on
$([-2C_0B,2C_0B]\cap\mathbb Z)^{k-1}$, such that
\begin{align}\label{eq:claim-conc}
 \left|\E_{x\in[N]}g_i(x)
 \mathcal D_{\widetilde\mu}
 (1_P,g_2,\ldots,g_{i-1},g_{i+1},\ldots,g_k)(x)\right|
 \geq c\kappa^A.
\end{align}
\end{lemma}

\begin{proof}
Since $|g_1|\leq1_{[N]}$, Cauchy--Schwarz and
\eqref{eq:claim-hyp} give
\[
 \E_{x\in[N]}|\mathcal D_\mu(g_2,\ldots,g_k)(x)|^2\geq\kappa^2.
\]
 The pointwise estimate
\eqref{eq:normalised-dual-pointwise-bound} and
Corollary~\ref{cor:Uk-anchored}, with support parameter $C_0$ and applied to
$\overline{\mathcal D_\mu(g_2,\ldots,g_k)}1_{[N]}/(2C_0+1)^{k-1}$ in the
anchored slot,
therefore give
\[
 \|\mathcal D_\mu(g_2,\ldots,g_k)\|_{U^k[N]}
  \geq c\kappa^A.
\]
If $k\geq3$, Lemma~\ref{lem:degreelower}, applied with $s=k$ to the
corresponding unnormalised dual restricted to $[N]$ and with support parameter
$C_0$, lowers this estimate to $U^2[N]$.  For $k=2$ the preceding estimate is
already a $U^2[N]$ estimate.
Thus Lemma~\ref{lem:dualU1+} applies in both cases and yields
\[
 \|\mathcal D_\mu(g_2,\ldots,g_k)\|_{U^{1+}[N]}
  \geq c\kappa^A.
\]
Hence there is an arithmetic progression $P\subseteq[N]$ such that
\[
 \left|\E_{x\in[N]}1_P(x)
 \mathcal D_\mu(g_2,\ldots,g_k)(x)\right|
  \geq c\kappa^A.
\]
Lemma~\ref{lem:anchoring}, with $h=1_P$ and $\ell=i$, gives
\eqref{eq:claim-conc} and identifies $\widetilde\mu$ as the pushforward of
$\mu$ under \eqref{eq:anchoring-tuple}.  Enlarging $A$ once absorbs the
finitely many polynomial losses in these applications.
\end{proof}

\begin{proof}[Proof of Theorem~\ref{thm:1+}]
By simultaneous permutation of the functions and coefficient variables, it
suffices to control $f_k$.  Pigeonholing the
sum over $b_1$ in \eqref{eq:R-large} gives a value $b_1^\ast$ for which the
corresponding slice has modulus $\gg_{k,C_0}\delta B^{k-2}N^2$.  Define
\[
 \mu_0(c_2,\ldots,c_k)
 \coloneqq \lambda(b_1^\ast,c_2+b_1^\ast,\ldots,c_k+b_1^\ast).
\]
The substitution $x\mapsto x-b_1^\ast m$ shows that $\mu_0$ is
$1$-bounded, supported on a centred cube of side length at most $4C_0B$, and
\begin{align}\label{eq:sec5-anchored}
 \left|\E_{x\in[N]}f_1(x)
 \mathcal D_{\mu_0}(f_2,\ldots,f_k)(x)\right|
 \geq\kappa_0,
\end{align}
where $\kappa_0=c_{k,C_0}'\delta$ for a fixed $c_{k,C_0}'>0$.

For $0\leq j\leq k-1$, set
$C_j^{\mathrm{sup}}=2^{j+1}C_0$.  Take
\[
    c_{\mathrm{rep}}
    =
    \min_{0\leq j<k-1}c(k,C_j^{\mathrm{sup}}),
\]
where the constants on the right are supplied by
Lemma~\ref{lem:bounded-slot-replacement}.  Let $A=A(k)$ be the exponent in
that lemma and define
\[
    \kappa_{j+1}=c_{\mathrm{rep}}\kappa_j^A
    \qquad\textnormal{for }0\leq j<k-1.
\]
Choose the outer exponent $C=C(k)$ in the theorem larger than the finitely
many powers of $A$ arising from this recurrence.  After decreasing
$c=c(k,C_0)$, the range hypothesis then implies
$\kappa_j^{-A}\leq B\leq N\kappa_j^A$ for every $j<k-1$.
We now iterate Lemma~\ref{lem:bounded-slot-replacement}.  After $j$ steps,
where $0\leq j\leq k-1$, there are progressions
$P_1,\ldots,P_j\subseteq[N]$ and a $1$-bounded weight $\mu_j$ supported on
$([-C_j^{\mathrm{sup}}B,C_j^{\mathrm{sup}}B]\cap\mathbb Z)^{k-1}$ such
that the outside function is
$f_{j+1}$, the inner list consists of the $j$ progression indicators and
$f_{j+2},\ldots,f_k$, and the corresponding normalised correlation is at
least $\kappa_j$.  Assertion \eqref{eq:sec5-anchored} is the case
$j=0$.  Given the assertion for $j<k-1$, the slot containing $f_{j+2}$
may be moved outside by Lemma~\ref{lem:bounded-slot-replacement}, used with
support parameter $C_j^{\mathrm{sup}}$. The
pushforward formula \eqref{eq:anchoring-tuple} preserves the required
$1$-boundedness and gives support with parameter
$C_{j+1}^{\mathrm{sup}}=2C_j^{\mathrm{sup}}$.  This proves the assertion for
$j+1$.

At step $j=k-1$ we have
\[
 \left|\E_{x\in[N]}f_k(x)
 \mathcal D_{\mu_{k-1}}(1_{P_1},\ldots,1_{P_{k-1}})(x)\right|
 \geq\kappa_{k-1}.
\]
Opening the dual function and using that there are $O_{k,C_0}(M)$ admissible
pairs $(\mathbf b,m)$, we find fixed $\mathbf b$ and $m$ such that
\[
 \left|\sum_x f_k(x)
 \prod_{j=1}^{k-1}1_{P_j}(x+b_{j+1}m)\right|
 \gg_{k,C_0}\kappa_{k-1}N.
\]
The product is the indicator of an intersection of translates of arithmetic
progressions, hence of an arithmetic progression.  Since $f_k$ is supported on $[N]$,
\eqref{eq:U1-plus-def} gives
$\|f_k\|_{U^{1+}[N]}\gg_{k,C_0}\kappa_{k-1}$.  The recurrence gives
$\kappa_{k-1}\geq c'(k,C_0)\delta^C$, which proves the theorem.
\end{proof}

\section{A relative form of Theorem~\ref{thm:1+}}\label{sec:relative-u1plus}

We prove a relative version of Theorem~\ref{thm:1+}.  The argument combines
the Cauchy--Schwarz estimate of Section~\ref{sec:engine} with weighted
and clipping estimates for the normalised dual function from
Definition~\ref{def:dual}.  A replacement induction then transfers the
bounded conclusion to functions dominated by a pseudorandom majorant, in the
spirit of~\cite{teravainen-wang-sparse}.

Set $B=N/H$ and
$M=B^{k-2}N=B^{k-1}H$.  We use the pointwise bound
\eqref{eq:normalised-dual-pointwise-bound} and the anchoring identity of
Lemma~\ref{lem:anchoring}.  We retain an explicit support parameter $C_0$:
multiplicative constants may depend on $C_0$, but all exponents in the
quantitative bounds below depend only on $k$.  The following reformulation of
Theorem~\ref{thm:1+} supplies the bounded
endpoint from which the replacement induction will begin.  Unlike the slot
replacement argument of Lemma~\ref{lem:bounded-slot-replacement}, it preserves
the outside function and produces a progression correlating directly with it.

\begin{corollary}[Bounded dual endpoint]\label{cor:bounded-dual-endpoint}
Let $C_0\geq1$.  There are constants $c=c(k,C_0)>0$ and $C=C(k)\geq1$
such that the following holds.  Let $0<\kappa<c$, let
$F,u_2,\ldots,u_k$ be $1$-bounded and supported on $[N]$, and let $\mu$ be
$1$-bounded and supported on
$([-C_0B,C_0B]\cap\mathbb Z)^{k-1}$.  Assume
\[
 \kappa^{-C}\leq B\leq N\kappa^C
\]
and
\begin{align}\label{eq:bounded-endpoint-hypothesis}
 \left|\E_{x\in[N]}F(x)
 \mathcal D_\mu(u_2,\ldots,u_k)(x)\right|\geq\kappa.
\end{align}
Then there is an arithmetic progression $P\subseteq[N]$ such that
\begin{align}\label{eq:bounded-endpoint-progression}
 \left|\E_{x\in[N]}1_P(x)F(x)\right|\geq c\kappa^C.
\end{align}
In particular, $\|F\|_{U^{1+}[N]}\geq c\kappa^C$.
\end{corollary}

\begin{proof}
Let $I=[B]$.  Define a weight $\lambda_I$ on $\Z^k$ by
\begin{align*}
 \lambda_I(t,t+b_2,\ldots,t+b_k)
 \coloneqq \mu(b_2,\ldots,b_k)
 \qquad\textnormal{for all }t\in I,
\end{align*}
and set $\lambda_I=0$ elsewhere.  This parametrisation is injective, so
$\lambda_I$ is $1$-bounded and supported on
$([-(C_0+1)B,(C_0+1)B]\cap\Z)^k$.  Expanding the counting operator and
substituting $y=x+tm$, we obtain
\begin{align*}
 R_H(\lambda_I;F,u_2,\ldots,u_k)
 &=
 \sum_{t\in I}\sum_{\mathbf b}\mu(\mathbf b)
 \sum_x\sum_{m\in[H]}F(x+tm)
 \prod_{i=2}^ku_i(x+(t+b_i)m)\\
 &=
 |I|\sum_{\mathbf b}\mu(\mathbf b)
 \sum_y\sum_{m\in[H]}F(y)
 \prod_{i=2}^ku_i(y+b_im)\\
 &=
 |I|M|[N]|\,
 \E_{y\in[N]}F(y)\mathcal D_\mu(u_2,\ldots,u_k)(y).
\end{align*}
Since $B\geq1$, we have $|I|=\lfloor B\rfloor\geq B/2$.  The hypothesis
therefore implies
\[
 |R_H(\lambda_I;F,u_2,\ldots,u_k)|
 \gg_{k,C_0}\kappa B^{k-1}N^2.
\]
Theorem~\ref{thm:1+}, applied with support parameter $C_0+1$, gives
$\|F\|_{U^{1+}[N]}\geq c\kappa^C$.
Equation~\eqref{eq:U1-plus-def} yields
\eqref{eq:bounded-endpoint-progression}.
\end{proof}

\subsection{The linear forms condition}

All affine linear systems below have one variable at scale $N$ and a bounded
number of variables at scale $H$.  We require a linear forms condition
adapted precisely to these scales.

\begin{definition}[Linear forms condition]\label{def:relative-LFC}
Let $N,B\geq1$, let $H$ be a positive integer, let $K,L\geq2$, and let
$0<\eta<1$.
Extend $\nu\colon[N]\to\R_{\geq0}$ by zero outside $[N]$.  We say that $\nu$
satisfies the $(K,L,\eta)$ linear forms condition at scale $(N,B,H)$ if the
following holds.

Let $0\leq s\leq K$ and $1\leq J\leq K$.  Let $I\subseteq[N]$ and
$I_1,\ldots,I_s\subseteq[H]$ be intervals with
\[
    |I|\geq N/L,
    \qquad |I_t|\geq H/L\quad\textnormal{for}\,\, 1\leq t\leq s.
\]
For $1\leq j\leq J$, let
\[
    \Psi_j(x,\mathbf h)=x+a_j+\sum_{t=1}^sc_{j,t}h_t,
\]
where $|a_j|\leq KN$, $|c_{j,t}|\leq KB$, and suppose the vectors
$(c_{j,1},\ldots,c_{j,s})$ are pairwise distinct.  We say that
$\Psi=(\Psi_1,\ldots,\Psi_J)$ has structural complexity at most $K$ if
these conditions hold; the least such integer $K$ is called the structural
complexity of $\Psi$.  Then, for every
$W_1,\ldots,W_J\in\{\nu,1_{[N]}\}$,
\begin{align}\label{eq:relative-LFC}
&\left|\E_{x\in I}\E_{h_1\in I_1,\ldots,h_s\in I_s}
 \prod_{j=1}^JW_j(\Psi_j(x,\mathbf h))-
 \E_{x\in I}\E_{h_1\in I_1,\ldots,h_s\in I_s}
 \prod_{j=1}^J1_{[N]}(\Psi_j(x,\mathbf h))\right|\leq\eta.
\end{align}
When $s=0$, the expectations over the $h_t$ are omitted.
\end{definition}

For example, taking $J=1$, $s=0$, $I=[N]$, and $W_1=\nu$ in
Definition~\ref{def:relative-LFC} gives the requirement
\begin{align}\label{eq:single-form-mean}
    \E_{x\in[N]}\nu(x)=1+O(\eta).
\end{align}

The next lemma supplies the two inputs needed when the
Cauchy--Schwarz estimate is applied to the sums indexed by $d$ below, if the
majorants satisfy the linear forms condition.  Its first conclusion
gives the required cube moment bounds, and its second gives cancellation after
the terminal contributions have been summed over $d$.

For $|d|<H$, set
\begin{align}\label{eq:Id-relative}
    I_{d,H}\coloneqq \{m\in\Z:m,m+d\in[H]\}.
\end{align}
We use the cube domain $\cQ_J(I)$ defined in~\eqref{eq:Qj-def}. 

\begin{lemma}[Cube moments and terminal cancellation]\label{lem:terminal-LFC}
Let $J_0$ be a positive integer, let $C_2\geq1$ and $L\geq2$, and let
$0<\eta<1$.  There is
$K_0=K_0(J_0,C_2)$ such that the following holds.  Suppose that $K\geq K_0$
and that $\nu$ satisfies the $(K,L,\eta)$ linear forms condition at scale
$(N,B,H)$.  Set
\begin{align}\label{eq:eta-star}
    \eta_*\coloneqq \eta+L^{-1}+H^{-1}.
\end{align}
Then every translate $w(\,\cdot+t)$, with
$w\in\{\nu,1_{[N]}\}$ and $|t|\leq C_2N$, satisfies the
$(J_0,A,C_2)$ boundedness condition of Definition~\ref{def:cube-moment} at
scale $(N,B,H)$ for some $A=O_{J_0,C_2}(1)$.  Equivalently, uniformly for
$0\leq j\leq J_0$ and nonzero $q_1,\ldots,q_j$ with
$|q_i|\leq C_2B$,
\begin{align}\label{eq:translated-cube-moment}
 \sum_x\sum_{u_1,\ldots,u_j\in[0,H-1]}
 \Delta_{q_1u_1,\ldots,q_ju_j}w(x+t)
 \ll_{J_0,C_2}NH^j.
\end{align}

Moreover, let $F=\nu-1_{[N]}$, let $1\leq J\leq J_0$, let
$a_1,\ldots,a_J$ be pairwise distinct nonzero integers of modulus at most
$C_2B$, let $\epsilon_j\in\{0,1\}$, and let
$W_j\in\{\nu,1_{[N]}\}$.  Define
\begin{align}\label{eq:sliced-terminal-def}
 \mathcal T_d\coloneqq {}&
 \sum_{(m,h_1,\ldots,h_J)\in\cQ_J(I_{d,H})}\sum_x
 \Delta_{a_1h_1,\ldots,a_Jh_J}F(x)\times\prod_{j=1}^J
 \Delta_{((a_i-a_j)h_i)_{i\ne j}}
 W_j(x+a_jm+\epsilon_ja_jd).
\end{align}
Then
\begin{align}\label{eq:sliced-terminal-LFC}
    \left|\sum_{|d|<H}\mathcal T_d\right|
    \ll_{J_0,C_2}\eta_*NH^{J+2}.
\end{align}
\end{lemma}

\begin{proof}
For~\eqref{eq:translated-cube-moment}, substitute $y=x+t$ and expand the
multiplicative differences.  The resulting forms are
\[
    y+\sum_{i=1}^j\omega_iq_iu_i,
    \qquad \boldsymbol\omega\in\{0,1\}^j.
\]
Their slope vectors in $(u_1,\ldots,u_j)$ are pairwise distinct because every
$q_i$ is nonzero.  The factor corresponding to $\boldsymbol\omega=\mathbf0$
restricts $y$ to $[N]$.  After translating $[0,H-1]$ into $[H]$, the
linear forms condition shows that the normalised moment is
$O_{J_0,C_2}(1)$.  For a translate of $1_{[N]}$, the same estimate follows
directly from the support condition.

We then prove~\eqref{eq:sliced-terminal-LFC}.  The summation domain is the set of
lattice points satisfying
\begin{align}\label{eq:sliced-polytope}
 m-\boldsymbol\omega\cdot\mathbf h\in[H],\qquad
 m-\boldsymbol\omega\cdot\mathbf h+d\in[H]
 \qquad\textnormal{for all }\boldsymbol\omega\in\{0,1\}^J.
\end{align}
It is a convex polytope in $(d,m,h_1,\ldots,h_J)$, contained in a box of side
length $O_J(H)$ and cut out by $O_J(1)$ linear inequalities.

After expanding all differences, the forms carrying a $F$ factor are
\begin{align}\label{eq:sliced-root-forms}
 \Phi_{\boldsymbol\omega}(x,\mathbf h)
 =x+\sum_{i=1}^J\omega_ia_ih_i,
 \qquad \boldsymbol\omega\in\{0,1\}^J,
\end{align}
whereas the forms from the $j$th majorant are
\begin{align}\label{eq:sliced-slot-forms}
 \Psi_{j,\boldsymbol\omega'}(x,m,d,\mathbf h)
 =x+a_jm+\epsilon_ja_jd
  +\sum_{i\ne j}\omega_i'(a_i-a_j)h_i,
\end{align}
with $\boldsymbol\omega'\in\{0,1\}^{[J]\setminus\{j\}}$.
These forms have pairwise distinct slope vectors. 

The $F$ form corresponding to \(\boldsymbol\omega=\mathbf0\) is \(F(x)\).
Since \(F\) is supported on \([N]\), we may restrict \(x\) to \([N]\), while
the remaining variables range over the polytope
\eqref{eq:sliced-polytope}, contained in a box of side length \(O_J(H)\).

To apply the linear forms condition, partition the range of \(x\) into intervals
of length comparable to \(N/L\), and each parameter range into intervals of
length comparable to \(H/L\).  The cells meeting the boundary of the
polytope account for an
\[
    O_{J_0}(L^{-1}+H^{-1})
\]
proportion of the lattice points in the ambient box.  Every other cell is
either contained in the polytope or disjoint from it.

On a cell contained in the polytope, translate the coordinate intervals to
subintervals of \([N]\) and \([H]\).  The forms
\eqref{eq:sliced-root-forms} and \eqref{eq:sliced-slot-forms} then satisfy
the hypotheses of the linear forms condition.  Expanding each $F$ factor as
\(F=\nu-1_{[N]}\), we may replace every term by the corresponding term
involving only \(1_{[N]}\), with normalised error
\(O_{J_0,C_2}(\eta)\).  The comparison terms cancel by the alternating
expansion.

On the boundary cells, we instead use
\[
    |F|\leq \nu+1_{[N]}.
\]
After enlarging to the corresponding product cells, the linear forms
condition bounds the resulting normalised averages by
\(O_{J_0,C_2}(1)\).  Combining the interior and boundary contributions gives
\[
    \left|\sum_{|d|<H}\mathcal T_d\right|
    \ll_{J_0,C_2}
    \bigl(\eta+L^{-1}+H^{-1}\bigr)NH^{J+2},
\]
which is \eqref{eq:sliced-terminal-LFC}.
\end{proof}

\subsection{Weighted estimates and clipping for dual functions}

For $R\geq1$ and $z\in\C$ set
\[
 \operatorname{clip}_{R}(z)=
 \begin{cases}z,&|z|\leq R,\\ Rz/|z|,&|z|>R.
 \end{cases}
\]
We also write
\[
    R_\nu\coloneqq \max(1,\|\nu\|_{\ell^\infty([N])}).
\]
Thus $R_\nu$ is the pointwise size parameter of the majorant.

The next proposition is the main input for densification.  Its
weighted estimate \eqref{eq:relative-dual-weighted-L2} permits the outer majorant $\nu$ to be replaced by
$1_{[N]}$, while its clipping estimate \eqref{eq:relative-dual-clipping} replaces an unbounded dual function
by a bounded one.  Both errors are controlled by the linear forms condition,
apart from a diagonal contribution bounded in terms of
$R_\nu^{O_k(1)}/B$.

\begin{proposition}[Weighted and clipping estimates for dual functions]
\label{prop:relative-dual-estimates}
Let $C_0\geq1$.  There are constants $K_0=K_0(k,C_0)$ and $C=C(k)$ such
that the following holds.  Let $N\geq1$, let $H$ be a positive integer, set
$B=N/H$, assume $B\geq1$, let $L\geq2$ and $0<\eta<1$, and suppose that $\nu$ satisfies
the $(K,L,\eta)$ linear forms condition at scale $(N,B,H)$ with
$K\geq K_0$.  Set
\[
    \eta_*=\eta+L^{-1}+H^{-1},
    \qquad c\coloneqq 2^{-2(k-1)}.
\]
Let $\mu\colon\Z^{k-1}\to\C$ be $1$-bounded and supported on
$([-C_0B,C_0B]\cap\Z)^{k-1}$.  Let
$S\subseteq\{2,\ldots,k\}$, and let $g_2,\ldots,g_k$ be supported on $[N]$
with $|g_i|\leq\nu$ for $i\in S$ and $|g_i|\leq1_{[N]}$ otherwise.  Set
$D=\mathcal D_\mu(g_2,\ldots,g_k)$, and choose $K_{\mathrm{bd}}= (2C_{\mathrm{sup}}+1)^{k-1}$, where $C_{\mathrm{sup}}$ is the largest support of the relevant coefficient weights,
\[
    D^\sharp(x)=1_{[N]}(x)\operatorname{clip}_{K_{\mathrm{bd}}}(D(x)).
\]
Then
\begin{align}
 \left|\E_{x\in[N]}(\nu(x)-1)|D(x)|^2\right|
 &\ll_{k,C_0}\eta_*^c+R_\nu^C/B,
 \label{eq:relative-dual-weighted-L2}\\
 \E_{x\in[N]}|D(x)-D^\sharp(x)|^2
 &\ll_{k,C_0}\eta_*^c+R_\nu^C/B.
 \label{eq:relative-dual-clipping}
\end{align}
\end{proposition}

\begin{proof}
Set $r=k-1$, $J=2r$, $F=\nu-1_{[N]}$, and
$M=B^rH$.  Choose $K_0$ so that Lemma~\ref{lem:terminal-LFC} applies with
$J_0=2r$ and every coefficient bound below.  Expanding the first average gives
\begin{align}\label{eq:weighted-expand}
 \sum_xF(x)|D(x)|^2
 =\frac1{M^2}\sum_{\mathbf b,\mathbf b'}
 \mu(\mathbf b)\overline{\mu(\mathbf b')}
 \mathcal S(\mathbf b,\mathbf b'),
\end{align}
where
\begin{align}\label{eq:Sbb}
 \mathcal S(\mathbf b,\mathbf b')
 =\sum_x\sum_{m,m'\in[H]}F(x)
 \prod_{i=2}^kg_i(x+b_im)
 \prod_{i=2}^k\overline{g_i(x+b_i'm')}.
\end{align}

Call a pair $(\mathbf b,\mathbf b')$ exceptional if the $2r$ integers
\begin{align}\label{eq:combined-coefficients}
    b_2,\ldots,b_k,b_2',\ldots,b_k'
\end{align}
are not pairwise distinct and nonzero.  The exceptional set is a union of
$O_k(1)$ affine hyperplanes in a box of side length $O_{C_0}(B)$, and hence
contains $O_{k,C_0}(B^{2r-1})$ pairs.

Fix a nonexceptional pair and write $m'=m+d$.  For $2\leq i\leq k$, let $W_i$
be the majorant $\nu$ or $1_{[N]}$ associated with $g_i$.  Define the $2r$
coefficients and functions by
\begin{align}\label{eq:scalar-fibre-data}
 \begin{array}{lll}
 a_{i-1}=b_i,& f_{i-1,d}=g_i,&w_{i-1,d}=W_i,\\[1mm]
 a_{r+i-1}=b_i',&
 f_{r+i-1,d}(y)=\overline{g_i(y+b_i'd)},&
 w_{r+i-1,d}(y)=W_i(y+b_i'd).
 \end{array}
\end{align}
Then
\begin{align}\label{eq:S-fibres}
 \mathcal S(\mathbf b,\mathbf b')
 =\sum_{|d|<H}\sum_x\sum_{m\in I_{d,H}}F(x)
   \prod_{j=1}^{2r}f_{j,d}(x+a_jm).
\end{align}
The shifts in the second block have modulus
$|b_i'd|\ll_{C_0}BH\ll_{C_0}N$.  We apply the iterated
Cauchy--Schwarz lemma, Lemma~\ref{lem:fibrewise-CS}, to the sum over
$I_{d,H}$ for each $d$.  Write $\cP_j(d)$ for the auxiliary factor
$\cP_j$ in \eqref{eq:Pj} in the application indexed by $d$.  The translated
cube moment bound \eqref{eq:translated-cube-moment} of
Lemma~\ref{lem:terminal-LFC}, applied to the majorants $w_{j,d}$ in
\eqref{eq:scalar-fibre-data}, gives
\begin{align*}
 \cP_j(d)\ll_{k,C_0} NH^{j-1}
 \qquad\textnormal{for }1\leq j\leq 2r.
\end{align*}
The associated terminal form is precisely the quantity $\cT_d$ in
\eqref{eq:sliced-terminal-def}.  Indeed, for $2\leq i\leq k$,
\begin{align*}
 w_{j,d}(x+a_jm)
 =
 \begin{cases}
 W_i(x+b_im),&j=i-1,\\
 W_i(x+b_i'm+b_i'd),&j=r+i-1,
 \end{cases}
\end{align*}
which corresponds to $\epsilon_j=0$ in the first block and
$\epsilon_j=1$ in the second.  Lemma~\ref{lem:fibrewise-CS} shows that each
$\cT_d$ and each $\cP_j(d)$ is nonnegative.  Lemma~\ref{lem:fibrewise-CS},
H\"older's inequality and \eqref{eq:sliced-terminal-LFC} of
Lemma~\ref{lem:terminal-LFC} now give
\begin{align}\label{eq:nonexceptional-pair}
 |\mathcal S(\mathbf b,\mathbf b')|
 &\leq
 \left(\sum_{|d|<H}\mathcal T_d\right)^{2^{-2r}}
 \prod_{j=1}^{2r}\left(\sum_{|d|<H}\cP_j(d)\right)^{2^{-j}}\notag\\
 &\ll_{k,C_0}
 \left(\eta_*NH^{2r+2}\right)^{2^{-2r}}
 \prod_{j=1}^{2r}\left(NH^j\right)^{2^{-j}}
 =\eta_*^{2^{-2r}}NH^2.
\end{align}

For an exceptional pair, the pointwise bounds and the support conditions give
\begin{align}\label{eq:exceptional-pair}
    |\mathcal S(\mathbf b,\mathbf b')|
    \ll_{k,C_0}R_\nu^{C}NH^2.
\end{align}
Since $M=B^rH$, insertion of
\eqref{eq:nonexceptional-pair}--\eqref{eq:exceptional-pair} into
\eqref{eq:weighted-expand}, followed by division by $|[N]|$, proves
\eqref{eq:relative-dual-weighted-L2}.

We now prove the clipping estimate~\eqref{eq:relative-dual-clipping}.  If
$S=\varnothing$, the pointwise bound
\eqref{eq:normalised-dual-pointwise-bound} gives
$|D|\leq K_{\mathrm{bd}}$, so there is nothing to prove.  Otherwise set
\begin{align*}
 A_\nu(x)&=\frac1M\sum_{\mathbf b}|\mu(\mathbf b)|\sum_{m\in[H]}
  \prod_{i\in S}\nu(x+b_im)\prod_{i\notin S}1_{[N]}(x+b_im),\\
 A_0(x)&=\frac1M\sum_{\mathbf b}|\mu(\mathbf b)|\sum_{m\in[H]}
  \prod_{i=2}^k1_{[N]}(x+b_im).
\end{align*}
Then $|D|\leq A_\nu$, while
\eqref{eq:normalised-dual-pointwise-bound}, applied to $|\mu|$ and the
functions $1_{[N]}$, gives $A_0\leq K_{\mathrm{bd}}$.  Thus, on $[N]$,
\begin{align}\label{eq:clip-A-bound}
 |D-D^\sharp|=(|D|-K_{\mathrm{bd}})_+
 \leq(A_\nu-A_0)_+\leq|A_\nu-A_0|.
\end{align}
For fixed $(\mathbf b,m)$, writing $L_i=x+b_im$, the identity
\begin{align}\label{eq:A-diff-identity}
 &\prod_{i\in S}\nu(L_i)\prod_{i\notin S}1_{[N]}(L_i)
 -\prod_{i=2}^k1_{[N]}(L_i)\notag\\
 &\qquad=\sum_{\varnothing\ne T\subseteq S}
 \prod_{i\in T}(\nu(L_i)-1_{[N]}(L_i))
 \prod_{i\notin T}1_{[N]}(L_i)
\end{align}
gives
\begin{align*}
 A_\nu(x)-A_0(x)
 ={}&\frac1M\sum_{\mathbf b}|\mu(\mathbf b)|\sum_{m\in[H]}
 \sum_{\varnothing\ne T\subseteq S}
 \prod_{i\in T}
 \bigl(\nu(x+b_im)-1_{[N]}(x+b_im)\bigr)\\
 &\qquad\times
 \prod_{i\in\{2,\ldots,k\}\setminus T}1_{[N]}(x+b_im).
\end{align*}

After squaring, consider one pair of nonempty sets $T,T'\subseteq S$ and one
coefficient pair $(\mathbf b,\mathbf b')$.  Write $m'=m+d$ once more.  For a
nonexceptional coefficient pair and fixed $d$, all factors are evaluated at
the forms
\[
   x+b_im,
   \qquad x+b_i'm+b_i'd,
   \qquad 2\leq i\leq k,
\]
whose coefficients of $m$ are the pairwise distinct list
\eqref{eq:combined-coefficients}.  If $|I_{d,H}|\geq H/L$, expand all
factors $\nu-1_{[N]}$ and apply the linear forms condition on
$m\in I_{d,H}$.  Every term differs by $O_{k,C_0}(\eta)$ from the
corresponding term involving only $1_{[N]}$, and these comparison terms
cancel in the alternating expansion because $T$ and $T'$ are nonempty.

It remains to consider those $d$ for which $|I_{d,H}|<H/L$.  For each such
$d$, choose an interval $J_{d,H}\subseteq[H]$ of length
$\lceil H/L\rceil$ containing $I_{d,H}$.  Since $|I_{d,H}|=H-|d|$,
\begin{align}\label{eq:short-fibre-mass}
 \sum_{\substack{|d|<H\\|I_{d,H}|<H/L}}|I_{d,H}|
 \leq
 \sum_{\substack{|d|<H\\|I_{d,H}|<H/L}}|J_{d,H}|
 \ll H^2(L^{-2}+H^{-1}).
\end{align}
On $J_{d,H}$, bound each factor $|\nu-1_{[N]}|$ by
$\nu+1_{[N]}$ and apply the linear forms condition to obtain a uniform
$O_{k,C_0}(1)$ normalised bound.  After dividing the squared expansion by
$M^2|[N]|$, the two contributions from nonexceptional coefficient pairs
are therefore bounded by
\begin{align*}
 \frac1{H^2}
 \sum_{\substack{|d|<H\\|I_{d,H}|\geq H/L}}
 |I_{d,H}|\,O_{k,C_0}(\eta)
 &\ll_{k,C_0}\eta,\\
 \frac1{H^2}
 \sum_{\substack{|d|<H\\|I_{d,H}|<H/L}}
 |J_{d,H}|\,O_{k,C_0}(1)
 &\ll_{k,C_0}L^{-2}+H^{-1}.
\end{align*}
The exceptional coefficient pairs form an
$O_{k,C_0}(B^{-1})$ proportion and contribute
$O_{k,C_0}(R_\nu^C/B)$.  Consequently,
\begin{align*}
 \E_{x\in[N]}|A_\nu(x)-A_0(x)|^2
 \ll_{k,C_0}\eta+L^{-2}+H^{-1}+R_\nu^C/B.
\end{align*}
It follows from \eqref{eq:clip-A-bound} that
\begin{align*}
 \E_{x\in[N]}|D-D^\sharp|^2
 \ll_{k,C_0}\eta+L^{-2}+H^{-1}+R_\nu^C/B
 \ll_{k,C_0}\eta_*^c+R_\nu^C/B,
\end{align*}
which is~\eqref{eq:relative-dual-clipping}.
\end{proof}

\begin{remark}\label{rmk:weak-sup-application}
In sieve applications one can truncate the majorant, at the cost of a
negligible error, so that $0\leq\nu\leq(\log N)^{C_1}$ for some fixed $C_1$.
Thus $R_\nu^C/B$ is negligible when $B$ exceeds a sufficiently large power of
$\log N$.
\end{remark}

\subsection{The replacement step}

The purpose of this subsection is to replace the sparse inner slots in a
dual function one at a time by arithmetic progression indicators.
A translate or finite intersection of arithmetic progressions is again an
arithmetic progression, possibly empty.  Thus a finite product of shifted
progression indicators is the indicator of one progression.  If the
correlated function is supported on $[N]$, intersecting that progression with
$[N]$ does not change the correlation and produces a progression admissible
in the definition of the $U^{1+}[N]$ norm.

We call an inner slot sparse if its function is dominated by $\nu$, and dense
if it is dominated by $1_{[N]}$.  We replace sparse slots one at a time by
progression indicators.

\begin{proposition}[Relative replacement]\label{prop:relative-replacement}
Let $k\geq3$ and $C_0\geq1$.  There are nondecreasing constants
$C_r=C_r(k)$ for $0\leq r\leq k-1$, a constant $C=C(k)$, and constants
$K_0=K_0(k,C_0)$ and $\tau_0=\tau_0(k,C_0)>0$ with the following property.  Let
$0<\tau\leq\tau_0$ and $N\geq1$, let $H$ be a positive integer, set
$B=N/H$, let $K\geq K_0$ and $L\geq2$, let $0<\eta<1$, and
suppose that $\nu$ satisfies the $(K,L,\eta)$ linear forms condition at scale
$(N,B,H)$.  Set
\begin{align*}
 \eta_*\coloneqq\eta+L^{-1}+H^{-1},
 \qquad c\coloneqq2^{-2(k-1)}.
\end{align*}
Assume
\begin{align}\label{eq:replacement-parameters}
 \eta_*^c+R_\nu^C/B\leq\tau^C,
 \qquad \tau^{-C}\leq B\leq N\tau^C.
\end{align}
For every integer $r$ with $0\leq r\leq k-1$, let $\mu$ be $1$-bounded and supported on
$([-C_0B,C_0B]\cap\mathbb Z)^{k-1}$.  Let $g_2,\ldots,g_k$ be supported on
$[N]$, each dominated by
$\nu$ or $1_{[N]}$, with exactly $r$ sparse slots, and let
$h\colon[N]\to\C$.  Set
$D=\mathcal D_\mu(g_2,\ldots,g_k)$ and suppose
\begin{align}\label{eq:replacement-hyp}
    |\E_{x\in[N]}h(x)D(x)|\geq\tau.
\end{align}
Then the following assertions hold.
\begin{enumerate}
\item[\emph{(B)}$_r$] If $|h|\leq1_{[N]}$, then some arithmetic progression $P\subseteq[N]$
satisfies
\[
    \left|\sum_{x\in P}h(x)\right|\gg_{k,C_0}\tau^{C_r}N.
\]
In particular $\|h\|_{U^{1+}[N]}\gg_{k,C_0}\tau^{C_r}$.
\item[\emph{(S)}$_r$] If $|h|\leq\nu$, then some arithmetic progression $P\subseteq[N]$
satisfies
\begin{align}\label{eq:replacement-sparse-conclusion}
    |\E_{x\in[N]}1_P(x)D(x)|\gg_{k,C_0}\tau^{C_r}.
\end{align}
\end{enumerate}
\end{proposition}

\begin{proof}
Set
\begin{align*}
 C_{\mathrm{sup}}=2^kC_0,
 \qquad K_{\mathrm{bd}}=(2C_{\mathrm{sup}}+1)^{k-1}.
\end{align*}
Every coefficient weight produced by the anchoring steps below has support
parameter at most $C_{\mathrm{sup}}$.  We choose $K_0$ and $\tau_0$ uniformly for
these finitely many support parameters.  This affects multiplicative
constants, but not any of the exponents $C_r$ or $C$.

We establish $(\mathrm B)_0$, then $(\mathrm S)_r$ from $(\mathrm B)_r$,
and, for $r\geq1$,
$(\mathrm B)_r$ from $(\mathrm S)_{r-1}$ and $(\mathrm B)_{r-1}$.  This is an
induction on the number of sparse inner slots. Corollary~\ref{cor:bounded-dual-endpoint} gives $(\mathrm B)_0$.  Its range hypotheses follow
from~\eqref{eq:replacement-parameters} after increasing $C$.

Suppose $(\mathrm B)_r$ is known and $|h|\leq\nu$.  By
\eqref{eq:replacement-hyp}, weighted Cauchy--Schwarz, and
\eqref{eq:single-form-mean},
\begin{align*}
 \tau^2
 &\leq\left|\E_{x\in[N]}h(x)D(x)\right|^2\\
 &\leq\left(\E_{x\in[N]}|h(x)|\right)
       \left(\E_{x\in[N]}|h(x)||D(x)|^2\right)\\
 &\leq\left(\E_{x\in[N]}\nu(x)\right)
       \left(\E_{x\in[N]}\nu(x)|D(x)|^2\right)
 \ll_{k,C_0}\E_{x\in[N]}\nu(x)|D(x)|^2.
\end{align*}
The weighted estimate~\eqref{eq:relative-dual-weighted-L2} of
Proposition~\ref{prop:relative-dual-estimates} and
\eqref{eq:replacement-parameters} therefore give
\begin{align}\label{eq:replacement-unweighted-L2}
   \E_{x\in[N]}|D(x)|^2\gg_{k,C_0}\tau^2.
\end{align}
Set
\[
    D^\sharp(x)=1_{[N]}(x)\operatorname{clip}_{K_{\mathrm{bd}}}(D(x))
\]
and let $\Delta=D-D^\sharp$.  The clipping estimate gives
\begin{align}\label{eq:replacement-clipping-error}
    \E_{x\in[N]}|\Delta(x)|^2\ll_{k,C_0}\tau^C.
\end{align}
Pointwise, we have
\begin{align}\label{eq:clip-identity}
 D^\sharp(x)\overline{D(x)}
 =|D(x)|^2-|\Delta(x)|^2-K_{\mathrm{bd}}|\Delta(x)|.
\end{align}
Indeed, when $|D(x)|\leq K_{\mathrm{bd}}$ the identity is trivial, while for
$|D(x)|>K_{\mathrm{bd}}$ both sides equal
$K_{\mathrm{bd}}|D(x)|$.  Equations
\eqref{eq:replacement-unweighted-L2}--\eqref{eq:replacement-clipping-error},
Cauchy--Schwarz, and a sufficiently large $C$ imply
\begin{align}\label{eq:replacement-sharp-correlation}
 \left|\E_{x\in[N]}\frac{D^\sharp(x)}{K_{\mathrm{bd}}}
 \overline{D(x)}\right|
 \gg_{k,C_0}\tau^2.
\end{align}
Now $D^\sharp/K_{\mathrm{bd}}$ is $1$-bounded, while
$\overline D=\mathcal D_{\overline\mu}(\overline{g_2},\ldots,
\overline{g_k})$ has the same $r$ sparse slots.  Applying $(\mathrm B)_r$ to
\eqref{eq:replacement-sharp-correlation} yields a progression $P$ with
$|\E_{x\in[N]}1_P(x)D^\sharp(x)|
\gg_{k,C_0}\tau^{O_k(1)}$.  Equation
\eqref{eq:replacement-clipping-error} then permits $D^\sharp$ to be replaced
by $D$.  This proves $(\mathrm S)_r$.

Finally, suppose $r\geq1$, $(\mathrm S)_{r-1}$ and
$(\mathrm B)_{r-1}$ are known, and $|h|\leq1_{[N]}$.  Let $g_\ell$ be a
sparse inner slot.
Lemma~\ref{lem:anchoring} gives
\[
 \E_{x\in[N]}h(x)\mathcal D_\mu(g_2,\ldots,g_k)(x)
 =\E_{y\in[N]}g_\ell(y)\,
 \mathcal D_{\widetilde\mu}
 (h,g_2,\ldots,g_{\ell-1},g_{\ell+1},\ldots,g_k)(y).
\]
The new inner list has $r-1$ sparse slots and the outside function $g_\ell$
is sparse.  Applying $(\mathrm S)_{r-1}$ yields a progression $P_1$ and the
corresponding correlation with $1_{P_1}$.  The inverse anchoring change puts
$1_{P_1}$ in the slot formerly occupied by $g_\ell$ and restores $h$ to the
outside.
The resulting dual function has $r-1$ sparse slots, so $(\mathrm B)_{r-1}$
produces a progression correlating with $h$.  This proves $(\mathrm B)_r$.

More explicitly, the bounded number of applications produces thresholds
$\tau^{A_j}$ with $A_j=O_k(1)$.  We choose the constant $C$ in
\eqref{eq:replacement-parameters} larger than the constants required when the
preceding argument is applied at each of these thresholds.  The finitely many
losses can then be absorbed into nondecreasing exponents $C_r$.  This
completes the induction.
\end{proof}

\subsection{The relative theorem}

We are now ready to transfer Theorem~\ref{thm:1+} to the unbounded setting.

\begin{theorem}[Relative $U^{1+}$ control]\label{thm:relative-u1plus}
Let $k\geq3$ and $C_0\geq1$.  There are constants $K=K(k,C_0)$,
$C=C(k)$ and $c=c(k,C_0)>0$ such that the following holds.  Let $N\geq1$
and let $0<\delta<c$ be a threshold.  Let $H$ be a positive integer, set
$B=N/H$, and assume
\[
    \delta^{-C}\leq B\leq N\delta^C.
\]
Set $L=\delta^{-C}$ and let $0<\eta\leq\delta^C$.  Let
$\nu\colon[N]\to\R_{\geq0}$ satisfy the $(K,L,\eta)$ linear forms condition
at scale $(N,B,H)$ and suppose
\begin{align}\label{eq:relative-sup-hypothesis}
    R_\nu^C/B\leq\delta^C.
\end{align}
Let $\lambda\colon\Z^k\to\C$ be $1$-bounded and supported on
$([-C_0B,C_0B]\cap\mathbb Z)^k$, and let $f_1,\ldots,f_k$ be supported on
$[N]$ with $|f_i|\leq\nu$.  If
\begin{align}\label{eq:relative-count-hypothesis}
 |R_H(\lambda;f_1,\ldots,f_k)|
 \geq\delta B^k\frac{N^2}{B},
\end{align}
then
\[
    \min_{1\leq i\leq k}\|f_i\|_{U^{1+}[N]}\geq c\delta^C.
\]
\end{theorem}

\begin{proof}
Simultaneous permutation relabels the functions and coefficient coordinates,
preserving the support bounds, the linear forms condition, and the counting
form.  It therefore suffices to prove the bound for $f_k$.  The symmetry
under diagonal shifts from
Remark~\ref{rmk:symmetric}, followed by pigeonholing the first coefficient, gives a
$1$-bounded weight $\mu_0$ supported on
$([-2C_0B,2C_0B]\cap\mathbb Z)^{k-1}$ such
that
\begin{align}\label{eq:relative-initial-anchored}
 \left|\E_{x\in[N]}f_1(x)
 \mathcal D_{\mu_0}(f_2,\ldots,f_k)(x)\right|
 \gg_{k,C_0}\delta.
\end{align}
Let $\tau$ denote the lower bound for a correlation to which
Proposition~\ref{prop:relative-replacement} is applied.  In the iteration
below, each such $\tau$ is bounded below by
$c_{k,C_0}\delta^A$ for one of finitely many exponents $A=O_k(1)$.
The range hypothesis gives $H^{-1}\leq\delta^C$, so
$\eta_*\ll\delta^C$.  Together with
\eqref{eq:relative-sup-hypothesis}, the two range bounds on $B$, and a
sufficiently large choice of $C=C(k)$, this verifies
\eqref{eq:replacement-parameters} for each of these values of $\tau$, after
decreasing $c=c(k,C_0)$ if necessary.

The conclusion for a sparse outside function in
Proposition~\ref{prop:relative-replacement}, applied with the appropriate
support parameter to
\eqref{eq:relative-initial-anchored}, yields a correlation of the same form
with the outside factor $f_1$ replaced by a progression indicator.  Anchoring
then puts $f_2$ outside.  Iterating these two deductions, for every
$0\leq j\leq k-2$ there are progressions $P_1,\ldots,P_j$, a
$1$-bounded weight $\mu_j$ supported on
\[
    ([-C_j^{\mathrm{rel}}B,C_j^{\mathrm{rel}}B]\cap\mathbb Z)^{k-1},
    \qquad
    C_j^{\mathrm{rel}}=2^{j+1}C_0,
\]
and an exponent $A_j=O_k(1)$ such that
\begin{align}\label{eq:relative-j-stage}
 \left|\E_{x\in[N]}f_{j+1}(x)
 \mathcal D_{\mu_j}
 (1_{P_1},\ldots,1_{P_j},f_{j+2},\ldots,f_k)(x)\right|
 \gg_{k,C_0}\delta^{A_j}.
\end{align}
Here the list of progression indicators is empty when $j=0$.  The anchoring
identity records the transformed weight at each stage and preserves its
$1$-boundedness, while replacing $C_j^{\mathrm{rel}}$ by
$C_{j+1}^{\mathrm{rel}}$.  Since there are at most $k-1$ stages, $K$ may be
chosen uniformly for these finitely many support parameters.  The choice of
$C=C(k)$ made in Proposition~\ref{prop:relative-replacement} supplies its
parameter hypotheses for every threshold $\delta^{A_j}$ after decreasing
$c=c(k,C_0)$.

Applying the conclusion for a sparse outside function once more when $j=k-2$
replaces the outside factor $f_{k-1}$ and gives
\begin{align}\label{eq:relative-last-correlation}
 \left|\E_{x\in[N]}1_{P_{k-1}}(x)
 \mathcal D_\mu(1_{P_1},\ldots,1_{P_{k-2}},f_k)(x)\right|
 \gg_{k,C_0}\delta^{O_k(1)}.
\end{align}
Expanding the dual function in
\eqref{eq:relative-last-correlation} and pigeonholing first over the
coefficients and then over $m$, there are a coefficient tuple and an
$m\in[H]$ such that, on
setting
\[
    \Psi_m(x)=1_{P_{k-1}}(x)
    \prod_{j=1}^{k-2}1_{P_j}(x+b_{j+1}m),
\]
one has
\begin{align*}
 \left|\sum_x\Psi_m(x)f_k(x+b_km)\right|
 \gg_{k,C_0}\delta^{O_k(1)}N.
\end{align*}
After the change of variables $y=x+b_km$, the remaining factor is
$\Psi_m(y-b_km)$.  By the progression observation preceding
Proposition~\ref{prop:relative-replacement},
this is the indicator of a
single arithmetic progression $P$.  Since $f_k$ is supported on $[N]$, the
same correlation is obtained with $P'=P\cap[N]$, which is a nonempty
arithmetic progression contained in $[N]$.  Hence
\[
    \left|\sum_{y\in P'}f_k(y)\right|
    \gg_{k,C_0}\delta^{O_k(1)}N,
\]
and the definition of the $U^{1+}[N]$ norm gives
$\|f_k\|_{U^{1+}[N]}\gg_{k,C_0}\delta^{O_k(1)}$.  Enlarge $C=C(k)$ once
more to obtain
the stated exponent, which proves the theorem by the symmetry noted above.
\end{proof}

\section{A sieve majorant for the primes}\label{sec:prime-majorant}

We now record a sieve input which supplies majorants satisfying the linear forms
condition used in Section~\ref{sec:relative-u1plus}.  The statement is a
variant of the pseudorandom majorant in
\cite[Lemma~5.3]{teravainen-wang-sparse}, whose proof in turn follows the work
of Green--Tao~\cite[Appendix~D]{green-tao-linear-equations}.  There are two
features of the cited construction to modify and one further obstruction to
control.  First, our linear forms have variables at the two scales $N$ and
$H=\lfloor N/B\rfloor$, and coefficients of the variables at scale $H$ may have size
$O_K(B)$ for a fixed structural complexity $K$.  Secondly, the GPY
majorant is not pointwise bounded, so we truncate it and use bounds for tails
of the divisor function to show that this changes all relevant linear forms averages by a
negligible amount.  Finally, uniformity in the coefficients requires control
of the large prime divisors of their slope differences.

There is a genuine local obstruction unless one controls the large prime
divisors of the coefficients.  For example, if $q\leq B$ is a product of
primes larger than $w$ and one tests the two forms
$x$ and $x+qm$, then these two forms coincide modulo every prime $p\mid q$.
The corresponding local factor is multiplied by
$\prod_{p\mid q,\ p>w}(1+O(1/p))$, which need not be $1+o(1)$ uniformly in
$q\leq B$.  The proposition below avoids this obstruction by imposing the
upper bound $B\leq\exp(w^{1/2})$.  Indeed, this gives
\begin{align*}
    \sum_{\substack{p>w\\p\mid q}}\frac1p
    \leq\frac{\log q}{w\log w}
    \leq\frac1{w^{1/2}\log w}.
\end{align*}
The proof in fact uses only the explicit condition on bad primes stated in
Remark~\ref{rmk:prime-majorant-range}.

Let $\Psi=(\Psi_1,\ldots,\Psi_J)$, with $1\leq J\leq K$, be a system of
structural complexity at most $K$ as in
Definition~\ref{def:relative-LFC}, and write
\begin{align*}
    \mathbf c_j=(c_{j,1},\ldots,c_{j,s})
\end{align*}
for its slope vectors.  For $1\leq i<j\leq J$, set
\begin{align*}
    d_{ij}=\gcd(c_{i,1}-c_{j,1},\ldots,c_{i,s}-c_{j,s}).
\end{align*}
The pairwise distinctness of the slope vectors ensures that $d_{ij}$ is a
positive integer.  Define the slope discriminant by
\begin{align}\label{eq:linear-discriminant-def}
    \mathfrak d(\Psi)
    \coloneqq \prod_{1\leq i<j\leq J}d_{ij},
\end{align}
with the convention that an empty product is $1$.  The coefficient bounds in
Definition~\ref{def:relative-LFC} give
\begin{align}\label{eq:disc-size-bound}
    1\leq \mathfrak d(\Psi)\leq (2KB)^{\binom J2}.
\end{align}
Indeed, $d_{ij}\leq\max_t|c_{i,t}-c_{j,t}|\leq2KB$.
When $s=0$, pairwise distinctness of the slope vectors forces $J=1$, and
then $\mathfrak d(\Psi)=1$ by the convention that an empty product is $1$.

\begin{lemma}[Discriminant bound for large primes]\label{lem:bad-prime-budget}
Let $K\geq2$, let $w\geq2$, and suppose that
$2\leq B\leq \exp(w^{1/2})$.  If
$\Psi$ is a system of structural complexity at most $K$ in
Definition~\ref{def:relative-LFC}, then
\begin{align}\label{eq:bad-prime-budget}
    \sum_{\substack{p>w\\ p\mid \mathfrak d(\Psi)}}\frac1p
    \ll_K \frac1{w^{1/2}\log w}.
\end{align}
\end{lemma}

\begin{proof}
For every integer $Q\geq1$,
$$
    \sum_{\substack{p>w\\p\mid Q}}\frac1p
    \leq
    \frac1w\#\{p>w:p\mid Q\}
    \leq
    \frac{\log Q}{w\log w}.
$$
Applying this with $Q=\mathfrak d(\Psi)$ and using~\eqref{eq:disc-size-bound} gives
$$
    \sum_{\substack{p>w\\ p\mid \mathfrak d(\Psi)}}\frac1p
    \ll_K \frac{\log B}{w\log w}
    \ll_K \frac1{w^{1/2}\log w}.
$$
\end{proof}

The following estimate for divisor tails permits the GPY majorant to be
truncated without changing the required linear forms averages.  The auxiliary
modulus $q_0$ is included so that the same estimate applies to affine
pullbacks.

\begin{lemma}[Divisor tails along the relevant forms]\label{lem:divisor-tail-forms}
Let $K\geq1$.  Let
$2\leq L\leq(\log N)^K$, let $w\leq(\log N)^{1/2}$, let
$W=\prod_{p\leq w}p$, let $1\leq q_0\leq\exp(w^{1/2})$, set
$\widetilde W=Wq_0$, and suppose that $(b,\widetilde W)=1$ and $b\geq0$.  Let
$2\leq B\leq\exp(w^{1/2})$ and $H=\lfloor N/B\rfloor$.  Let
$\Psi_1,\ldots,\Psi_J$, with $1\leq J\leq K$, be a system of structural
complexity at most $K$ as in Definition~\ref{def:relative-LFC}.  Let
$I\subseteq[N]$ and $I_1,\ldots,I_s\subseteq[H]$ be intervals satisfying
\begin{align*}
    |I|\geq N/L,
    \qquad |I_t|\geq H/L\quad\textnormal{for}\,\,1\leq t\leq s,
\end{align*}
and set
$\Omega=I\times I_1\times\cdots\times I_s$.  Define
\begin{align*}
    \mathcal K_\Psi
    =\{\mathbf v\in\Omega:\Psi_j(\mathbf v)\in[N]
      \text{ for all }1\leq j\leq J\}.
\end{align*}
Then, for every integer $\ell\geq1$,
\begin{align}\label{eq:divisor-moment-forms}
    \frac1{|\Omega|}\sum_{\mathbf v\in\mathcal K_\Psi}
    \prod_{j=1}^J\tau(\widetilde W\Psi_j(\mathbf v)+b)^\ell
    \ll_{K,\ell} (\log N)^{C(K,\ell)}.
\end{align}
Consequently, for every $A,E\geq1$ there is
$D_0=D_0(K,A,E)$ such that, whenever $D\geq D_0$ and
$1\leq j_0\leq J$,
\begin{align}\label{eq:divisor-tail-forms}
    \frac1{|\Omega|}\sum_{\mathbf v\in\mathcal K_\Psi}
    &\tau(\widetilde W\Psi_{j_0}(\mathbf v)+b)^A
    1_{\tau(\widetilde W\Psi_{j_0}(\mathbf v)+b)>(\log N)^D}\notag\\
    &\hspace{35mm}\times
    \prod_{j\neq j_0}\tau(\widetilde W\Psi_j(\mathbf v)+b)^A
    \ll_{K,A,E} (\log N)^{-E}.
\end{align}
\end{lemma}

\begin{proof}
By H\"older's inequality, it is enough to prove the single form estimate
\begin{align}\label{eq:one-form-divisor-moment}
    \frac1{|\Omega|}
    \sum_{\substack{\mathbf v\in\Omega\\\Psi(\mathbf v)\in[N]}}
    \tau(\widetilde W\Psi(\mathbf v)+b)^Q
    \ll_{K,Q} (\log N)^{C_1(K,Q)}
\end{align}
for every integer $Q\geq1$, where $C_1(K,Q)$ is a suitable exponent.  Write
$$
    \Psi(x,\mathbf h)=x+a+\sum_{t=1}^s c_t h_t .
$$
For fixed $h_1,\ldots,h_s$, the admissible values of
$y=\Psi(x,\mathbf h)$ form an interval in $[N]$ contained in an interval
$I'_{\mathbf h}\subseteq[N]$ of length $|I|$.  Shiu's bound for a nonnegative
multiplicative function in a primitive arithmetic progression
\cite[Theorem~1]{shiu}, applied to $\tau^Q$, gives
\begin{align}\label{eq:ap-divisor-moment}
    \sum_{y\in I'_{\mathbf h}}\tau(\widetilde Wy+b)^Q
    \ll_{K,Q}|I|(\log N)^{C_1(K,Q)}.
\end{align}
Indeed, $|I|\geq N/L\geq N(\log N)^{-K}$,
$\widetilde W=\exp(O(w))=N^{o(1)}$, and $(b,\widetilde W)=1$; hence the
estimate is uniform in the location of $I'_{\mathbf h}$ and in all the
parameters above.  For each fixed $\mathbf h$, the sum over the admissible
$x\in I$ is at most the left-hand side of
\eqref{eq:ap-divisor-moment}.  Summing this bound over
$\mathbf h\in I_1\times\cdots\times I_s$ and dividing by
$|\Omega|=|I|\prod_{t=1}^s|I_t|$ proves
\eqref{eq:one-form-divisor-moment}, so
\eqref{eq:divisor-moment-forms} follows.

For the tail estimate,
\[
  1_{\tau(Y)>(\log N)^D}\leq \tau(Y)(\log N)^{-D}.
\]
Since $\tau(Y)\geq1$, the left side of
\eqref{eq:divisor-tail-forms} is therefore at most $(\log N)^{-D}$ times
the moment in \eqref{eq:divisor-moment-forms} with $\ell=A+1$.  Choosing
$D_0>C(K,A+1)+E$ proves \eqref{eq:divisor-tail-forms}.
\end{proof}

Fix a smooth function $\chi\colon\R\to\R_{\geq0}$ supported on $[-2,2]$
and equal to $1$ on $[-1,1]$, and set
\[
 I_\chi\coloneqq \int_0^\infty|\chi'(t)|^2\,\d t>0.
\]
For $R\geq2$, define the GPY sieve weight
\begin{align}\label{eq:selberg-weight-def}
 \Lambda_{\chi,R,2}(n)
 \coloneqq 
 \frac{\log R}{I_\chi}
 \left(\sum_{d\mid n}\mu(d)
 \chi\left(\frac{\log d}{\log R}\right)\right)^2.
\end{align}
The factor $(\log R)/I_\chi$ is chosen so that the average for a single form, apart
from its local factor, is $1+o(1)$.

\begin{lemma}[GPY linear forms estimate at two scales]
\label{lem:two-scale-selberg}
For every $K\geq2$ there are constants $\gamma=\gamma(K)>0$ and
$c_2=c_2(K)>0$ such that the following holds.  Let $N$ be sufficiently
large, let $w\to\infty$ with $W=\prod_{p\leq w}p=N^{o(1)}$, let
$(b,W)=1$, and set $R=N^\gamma$.
Let $B,L\geq1$, set $H=\lfloor N/B\rfloor$, suppose
$L\leq(\log N)^K$, and
\[
 B\leq\exp(w^{1/2}),\qquad
 |I|\geq N/L,\qquad |I_t|\geq H/L.
\]
Suppose that $\Psi_1,\ldots,\Psi_J$, with $1\leq J\leq K$, are forms with
structural complexity at most $K$, as in
Definition~\ref{def:relative-LFC}.
For
\[
 \Omega=I\times I_1\times\cdots\times I_s,\qquad
 \mathcal K=\{\mathbf v\in\Omega:\Psi_j(\mathbf v)\in[N]\ 
 \text{for every }j\},
\]
set
\[
 \nu_0(n)=1_{[N]}(n)\frac{\varphi(W)}W
 \Lambda_{\chi,R,2}(Wn+b).
\]
Then, for every $S\subseteq[J]$,
\begin{align}\label{eq:selberg-linear-forms-asymp}
 &\E_{\mathbf v\in\Omega}
 \prod_{j\in S}\nu_0(\Psi_j(\mathbf v))
 \prod_{j\notin S}1_{[N]}(\Psi_j(\mathbf v))\notag\\
 &\qquad=
 \frac{|\mathcal K|}{|\Omega|}\prod_p\beta_p(S)
 +O_K((\log N)^{-c_2}),
\end{align}
where $\beta_p(S)=1$ for $p\mid W$, while for $p\nmid W$,
\begin{align}\label{eq:local-factor-def}
 \beta_p(S)=
 \E_{\mathbf v\in(\Z/p\Z)^{s+1}}
 \prod_{j\in S}\frac{p}{p-1}
 1_{W\Psi_j(\mathbf v)+b\not\equiv0\pmod p}.
\end{align}
\end{lemma}

\begin{proof}
If $S=\varnothing$, the asserted identity is immediate, since both sides
equal $|\mathcal K|/|\Omega|$.  We may therefore assume that
$S\neq\varnothing$.
Set
\begin{align*}
    \rho_d=\mu(d)\chi\left(\frac{\log d}{\log R}\right).
\end{align*}
Opening the square in every GPY factor in
\eqref{eq:selberg-weight-def} and exchanging the order of summation gives
\begin{equation}\label{eq:two-scale-gpy-expansion}
\begin{aligned}
&\E_{\mathbf v\in\Omega}
 \prod_{j\in S}\nu_0(\Psi_j(\mathbf v))
 \prod_{j\notin S}1_{[N]}(\Psi_j(\mathbf v))\\
&\quad=
\left(\frac{\varphi(W)\log R}{WI_\chi}\right)^{|S|}
\sum_{\substack{1\leq d_j,e_j\leq R^2\\j\in S}}
\left(\prod_{j\in S}\rho_{d_j}\rho_{e_j}\right)
\frac1{|\Omega|}\\
&\qquad\qquad\times
\#\left\{\mathbf v\in\mathcal K:
 \operatorname{lcm}(d_j,e_j)\mid W\Psi_j(\mathbf v)+b
 \text{ for every }j\in S\right\}.
\end{aligned}
\end{equation}
For each tuple $(\mathbf d,\mathbf e)$, these divisibility conditions have
combined modulus
\begin{align*}
    m=\mathop{\rm lcm}_{j\in S}\operatorname{lcm}(d_j,e_j)
    \leq R^{O_K(1)}=N^{O_K(\gamma)}
\end{align*}
and define a set
$\mathcal R(\mathbf d,\mathbf e)\subseteq(\mathbb Z/m\mathbb Z)^{s+1}$
of residue classes.  The Chinese remainder theorem separates these
conditions one prime at a time.

It remains to establish, uniformly in $m$, the lattice point estimate needed
to count the solutions of the divisibility conditions in the preceding
display.  The set $\mathcal K$ is the set of integer points in a polytope
contained in a rectangular box with one side of length $O(N)$ and $s$ sides
of length $O(H)$, and this polytope has $O_K(1)$ faces.  For every modulus
$m\leq R^{O_K(1)}$ and every residue class
$\mathbf r\pmod m$, we claim that
\begin{align}\label{eq:two-scale-lattice-count}
 \#\bigl(\mathcal K\cap(\mathbf r+m\Z^{s+1})\bigr)
 =\frac{|\mathcal K|}{m^{s+1}}
 +O_K\left(\frac{NH^{s-1}+H^s}{m^s}+1\right).
\end{align}
Indeed, partition the rectangular box into half-open boxes aligned with
$m\mathbb Z^{s+1}$.  After decreasing $\gamma(K)$ if necessary, we have
$m\leq H$ for all sufficiently large $N$.  Each box lying wholly inside the
polytope contributes
one point to every residue class, so only boxes meeting its boundary can
contribute to the discrepancy in
\eqref{eq:two-scale-lattice-count}.  For each defining face, project along a
coordinate for which the corresponding coefficient of its normal vector has
maximal modulus.  The face is then a graph in that coordinate with slopes of
modulus at most $1$, so each projected $m$-box lifts to $O_K(1)$ boxes meeting
the face.  Projecting along the first coordinate $x$ gives a region contained
in an $s$-dimensional box of volume $O(H^s)$ and hence
$O_K(H^s/m^s+1)$ boundary boxes.  Projecting along the $h_t$-coordinate gives
a region contained in an $s$-dimensional box of volume $O(NH^{s-1})$ and
hence $O_K(NH^{s-1}/m^s+1)$ boundary boxes.  Summing over the
$O_K(1)$ faces proves \eqref{eq:two-scale-lattice-count}.

Since $|\Omega|\gg_KNH^s/L^{K+1}$, normalising
\eqref{eq:two-scale-lattice-count} gives
\begin{align*}
\frac1{|\Omega|}\#\bigl(\mathcal K\cap(\mathbf r+m\Z^{s+1})\bigr)
=\frac{|\mathcal K|}{|\Omega|m^{s+1}}
+O_K\left(\frac{L^{K+1}}{m^sH}
+\frac{L^{K+1}}{m^sN}
+\frac{L^{K+1}}{NH^s}\right).
\end{align*}
The last summand is the contribution of the endpoint term $1$ in
\eqref{eq:two-scale-lattice-count}.  Relative to the density
$m^{-(s+1)}$, the three error terms are
\begin{align*}
    O_K\left(L^{K+1}\left(
    \frac mH+\frac mN+\frac{m^{s+1}}{NH^s}\right)\right).
\end{align*}
Summing the lattice main term over
$\mathcal R(\mathbf d,\mathbf e)$ and then over the divisor tuples in
\eqref{eq:two-scale-gpy-expansion} therefore gives
\begin{align*}
\frac{|\mathcal K|}{|\Omega|}
\left(\frac{\varphi(W)\log R}{WI_\chi}\right)^{|S|}
\sum_{\substack{1\leq d_j,e_j\leq R^2\\j\in S}}
\left(\prod_{j\in S}\rho_{d_j}\rho_{e_j}\right)
\frac{|\mathcal R(\mathbf d,\mathbf e)|}{m^{s+1}}
+O_K\left(\frac{N^{O_K(\gamma)}L^{O_K(1)}}H\right).
\end{align*}
Here the two other normalised errors above, including the explicit endpoint
term, are absorbed by the displayed error after choosing $\gamma(K)$
sufficiently small.

The divisor sum algebra at each prime and the normalisation by $I_\chi$ are
those of \cite[Appendix~D]{green-tao-linear-equations}; they turn the main
term in the preceding display into
$|\mathcal K||\Omega|^{-1}\prod_p\beta_p(S)$.  The error in that
calculation, which saves a power of $\log N$, contains a factor $e^{O(X)}$,
where
\begin{align*}
   X= \sum_{\substack{p>w\\p\mid\mathfrak d(\Psi)}}p^{-1/2}
    \leq
    \frac{\log\mathfrak d(\Psi)}{\sqrt w\log w}
    \ll_K
    \frac{\log B}{\sqrt w\log w}
    \ll_K\frac1{\log w}.
\end{align*}
By~\eqref{eq:disc-size-bound} and $B\leq\exp(\sqrt w)$, this factor is
uniformly bounded.  Thus, after decreasing $c_2(K)$ if necessary, the
analytic error remains $O_K((\log N)^{-c_2})$.  Finally, the hypothesis on
$B$ gives $H=N^{1-o(1)}$, so choosing $\gamma(K)$ sufficiently small makes
the lattice error $O_K((\log N)^{-c_2})$.  This proves
\eqref{eq:selberg-linear-forms-asymp}.
\end{proof}

Combining the preceding two results, we can now construct a majorant bounded by a power of $\log N$ that we use in the proof of Theorem \ref{thm:prime-main}.
\begin{proposition}[Truncated GPY majorant]\label{prop:prime-majorant}
For every $K\geq2$ there are constants $c_0=c_0(K)>0$, $c_1=c_1(K)>0$,
$C=C(K)>0$ and $D=D(K)>0$ such that the following holds for all sufficiently
large $N$.
Let $L$ satisfy
$$
    2\leq L\leq(\log N)^{c_0},
$$
let
$$
    \tfrac12(\log N)^{c_0}\leq w\leq2(\log N)^{c_0},
$$
and set
\[
    W=\prod_{p\leq w}p.
\]
Let $1\leq b\leq W$ with $(b,W)=1$.  Let $B$ satisfy
\begin{align}\label{eq:prime-majorant-B-range}
    (\log N)^C\leq B\leq \min\left(\frac{N}{(\log N)^C},\exp(w^{1/2})\right).
\end{align}
Set $H=\lfloor N/B\rfloor$.
Then there is a nonnegative function
$\nu=\nu_{K,W,b}\colon\Z\to\R_{\geq0}$ with the following properties.
\begin{enumerate}
\item[\emph{(i)}] \emph{Ambient scale.}
The function $\nu$ is supported on $[N]$ and, for every $n\in[N]$,
\begin{align}\label{eq:prime-majorant-majorises}
    \frac{\varphi(W)}{W}\Lambda(Wn+b)
    \ll_K \nu(n).
\end{align}
For every $n\in\Z$,
\begin{align}\label{eq:prime-majorant-sup}
    0\leq \nu(n)\leq (\log N)^D.
\end{align}
At scale $(N,B,H)$, the function $\nu$ satisfies the
$(K,L,(\log N)^{-c_1})$ linear forms condition of
Definition~\ref{def:relative-LFC}, simultaneously for every $L$ in the
displayed range.

\item[\emph{(ii)}] \emph{Primitive affine pullbacks.}
Let $Q\geq1$ satisfy
\begin{align}\label{eq:prime-majorant-pullback-range}
    QB\leq\exp(w^{1/2}),
\end{align}
and let
\begin{align*}
    P=\{qn+a:1\leq n\leq N_P\}\subseteq[N],
    \qquad 1\leq q\leq Q,
    \qquad N_P\geq N/Q.
\end{align*}
Set $b_P=Wa+b$, and suppose that $(q,b_P)=1$.  Then the primitive affine
pullback
\begin{align}\label{eq:prime-majorant-pullback}
    \nu_P(n)=1_{[N_P]}(n)\nu(qn+a)
\end{align}
satisfies $0\leq\nu_P(n)\leq(\log N)^D$ for all $n\in\Z$.  At scale
$(N_P,B,\lfloor N_P/B\rfloor)$, it satisfies the
$(K,L,(\log N)^{-c_1})$ linear forms condition simultaneously for every $L$
in the displayed range.  Moreover, for every $n\in[N_P]$,
\begin{align}\label{eq:prime-majorant-pullback-majorises}
    \frac{\varphi(W)}W\Lambda(Wqn+b_P)\ll_K\nu_P(n).
\end{align}
\end{enumerate}
\end{proposition}

\begin{proof}
We choose $c_0=c_0(K)<1/2$ sufficiently small and $C=C(K)$ and $D=D(K)$
sufficiently large at the end of the proof.  All estimates below are uniform
for $w$ in the displayed interval, since changing $w$ by a bounded factor
preserves both $W=\exp(O(w))=N^{o(1)}$ and every stated power saving.  Let
$0<\gamma=\gamma(K)$ be a parameter, to be fixed subject to the constraints
below, and set $R=N^\gamma$.
Define
\begin{align}\label{eq:nu-prime-untruncated}
    \nu_0(n)
    \coloneqq 1_{[N]}(n)\frac{\varphi(W)}{W}\Lambda_{\chi,R,2}(Wn+b).
\end{align}
Let $\mathcal E$ be the set of $n\in[N]$ for which either $Wn+b<R^2$ or
$Wn+b$ is a proper prime power.  Set
\begin{align}\label{eq:nu-prime-before-trunc}
    \nu'(n)\coloneqq \nu_0(n)+(\log N)1_{\mathcal E}(n),
    \qquad
    \nu(n)\coloneqq \nu'(n)1_{\nu'(n)\leq(\log N)^D}.
\end{align}
The term $(\log N)1_{\mathcal E}$ is included only to cover the small primes and
proper prime powers in the von Mangoldt function.

We first verify the majorisation.  If $Wn+b$ is prime and $Wn+b\geq R^2$, then
only the divisor $1$ contributes to the inner sum in~\eqref{eq:selberg-weight-def},
so $\Lambda_{\chi,R,2}(Wn+b)=(\log R)/I_\chi\asymp_K\log N$.  For the
primes under consideration,
\[
    R^2\leq Wn+b\leq WN+b.
\]
Since $R=N^\gamma$ and $W=N^{o(1)}$, this gives
$\log(Wn+b)\asymp\log N$.
Thus~\eqref{eq:prime-majorant-majorises} holds for such primes.  Suppose now
that $\Lambda(Wn+b)\ne0$.  Then $Wn+b$ is a prime power.  In the divisor sum
in~\eqref{eq:selberg-weight-def}, only the squarefree divisors $1$ and the
underlying prime can contribute, and therefore $\nu'(n)\ll_K\log N$.
If this prime power is not a prime at least $R^2$, then $n\in\mathcal E$, and
hence
$$
    \nu'(n)\geq\log N
    \gg \frac{\varphi(W)}W\Lambda(Wn+b).
$$
Consequently, if $D>1$, the truncation does not affect any value needed for
the majorisation when $N$ is large.  The pointwise bound
\eqref{eq:prime-majorant-sup} follows immediately from the definition of $\nu$.

It remains to prove the linear forms condition.  Fix a system
$\Psi_1,\ldots,\Psi_J$ of structural complexity at most $K$ as in
Definition~\ref{def:relative-LFC}.  Let
$I\subseteq[N]$ and $I_1,\ldots,I_s\subseteq[H]$ be intervals satisfying
\begin{align*}
    |I|\geq N/L,
    \qquad |I_t|\geq H/L\quad\textnormal{for}\,\,1\leq t\leq s,
\end{align*}
and fix functions
$W_1,\ldots,W_J\in\{\nu,1_{[N]}\}$.  Set
\begin{align*}
    \Omega=I\times I_1\times\cdots\times I_s,
    \qquad
    S=\{j:W_j=\nu\}.
\end{align*}
Let
$$
    \mathcal K\coloneqq \{\mathbf v\in\Omega:\Psi_j(\mathbf v)\in[N]
      \text{ for all }1\leq j\leq J\}.
$$
The set $\mathcal K$ is obtained from the product box $\Omega$ by imposing $O_K(1)$
additional affine inequalities of the form $1\leq\Psi_j\leq N$.  Thus it is a
convex polytope with $O_K(1)$ faces, and this is the only boundary regularity
used below.  Lemma~\ref{lem:two-scale-selberg}, with $\gamma(K)$ chosen
sufficiently small, gives \eqref{eq:selberg-linear-forms-asymp}.

We now estimate the Euler product in
\eqref{eq:selberg-linear-forms-asymp}.  If $p>w$ and
$p\nmid \mathfrak d(\Psi)$, then the
slope vectors of the forms in $S$ remain pairwise distinct modulo $p$.
Consequently the events $W\Psi_j(\mathbf v)+b\equiv0\pmod p$ have the expected
codimension and pairwise intersections have codimension two, giving
\begin{align}\label{eq:good-local-factor}
    \beta_p(S)=1+O_K(p^{-2}).
\end{align}
For the remaining primes $p>w$ we use only the trivial bound
\begin{align}\label{eq:bad-local-factor}
    \beta_p(S)=1+O_K(p^{-1}).
\end{align}
Combining~\eqref{eq:good-local-factor},~\eqref{eq:bad-local-factor} and
Lemma~\ref{lem:bad-prime-budget}, we get
\begin{align}\label{eq:euler-product-one}
    \prod_p\beta_p(S)=1+O_K(w^{-1/3})
    =1+O_K((\log N)^{-c_3(K)})
\end{align}
for some $c_3(K)>0$.  Hence
\begin{align}\label{eq:nu0-lfs}
    \E_{\mathbf v\in\Omega}
    \prod_{j\in S}\nu_0(\Psi_j(\mathbf v))
    \prod_{j\notin S}1_{[N]}(\Psi_j(\mathbf v))
    =
    \frac{|\mathcal K|}{|\Omega|}+O_K((\log N)^{-c_3(K)}).
\end{align}
Here $|\mathcal K|/|\Omega|$ equals
$$
    \E_{\mathbf v\in\Omega}\prod_{j=1}^J1_{[N]}(\Psi_j(\mathbf v)).
$$

We next replace $\nu_0$ by $\nu'$.  The exceptional set $\mathcal E$ has size
$O(R^2/W+(WN)^{1/2})=O(N^{1/2+o(1)})$, provided $\gamma<1/4$.  For each
$\Psi_j$, fix the variables at scale $H$.  Since the coefficient of
$x$ is $1$, the map $x\mapsto\Psi_j(x,\mathbf h)$ is injective, so at most
$|\mathcal E|$ values of $x\in I$ enter the exceptional set.  As
$|I|\geq N/L$, the proportion of $\mathbf v\in\Omega$ for which
$\Psi_j(\mathbf v)\in\mathcal E$ is at most
$L|\mathcal E|/N=N^{-1/2+o_K(1)}$.  Using
Lemma~\ref{lem:divisor-tail-forms} with a large fixed moment to control the
other GPY factors, the total contribution of the terms containing at
least one factor $(\log N)1_{\mathcal E}$ associated with the exceptional set is
$O_K((\log N)^{-c_4(K)})$ for some $c_4(K)>0$.  Thus
\eqref{eq:nu0-lfs} remains true with $\nu'$ in place of $\nu_0$.

Finally, we truncate.  From~\eqref{eq:selberg-weight-def} we have the pointwise
bound
\begin{align}\label{eq:selberg-divisor-bound}
    \nu'(n)
    \ll_K (\log N)\tau(Wn+b)^2
    +(\log N)1_{\mathcal E}(n).
\end{align}
Expanding the difference between the average with $\nu'$ and the average with
$\nu$ and using the union bound, it suffices to estimate terms in which one of
the factors satisfies $\nu'(\Psi_j(\mathbf v))>(\log N)^D$.  By
\eqref{eq:selberg-divisor-bound} this forces either $\Psi_j(\mathbf v)\in\mathcal E$
or $\tau(W\Psi_j(\mathbf v)+b)>(\log N)^{(D-1)/2}$.  The exceptional set
contribution was just shown to be negligible.  If $D=D(K)$ is sufficiently
large, Lemma~\ref{lem:divisor-tail-forms} bounds the contribution from divisor tails
by $O_K((\log N)^{-c_4(K)})$.  Therefore the averages with $\nu$ and $\nu'$
differ by
$O_K((\log N)^{-c_4(K)})$ for every system and every choice of the functions
$W_j$ in the definition.  Combining this with~\eqref{eq:nu0-lfs} proves the
$(K,L,(\log N)^{-c_1})$ linear forms condition, where we may take
\[
    0<c_1=c_1(K)
    <\min(c_0/3,c_2(K),c_3(K),c_4(K)).
\]

It remains to verify the affine pullback assertion.  Fix $Q$ and $P$ as in the
statement.  Since $(b_P,W)=1$ and $(q,b_P)=1$, we have $(Wq,b_P)=1$.  Also
$q\leq Q\leq\exp(w^{1/2})$ and
\begin{align*}
    \frac{N_P}{B}\geq\frac{N}{QB}
    \geq N\exp(-w^{1/2})=N^{1-o(1)}.
\end{align*}
In particular, $N_P=N^{1-o(1)}$. Set $H_P=\lfloor N_P/B\rfloor$.  For a
system $(\Psi_1,\ldots,\Psi_J)$ at scale $(N_P,B,H_P)$, substituting
\eqref{eq:prime-majorant-pullback} into the required average replaces the
prime forms $W\Psi_j+b$ in the preceding argument by
\begin{align}\label{eq:prime-majorant-pulled-forms}
    Wq\Psi_j+b_P.
\end{align}
Every divisor tuple with a nonempty residue set consists of integers coprime
to $Wq$, since $(Wq,b_P)=1$; tuples failing this condition contribute zero.
Multiplication by $Wq$ is therefore invertible modulo every divisor modulus.
Repeating the proof of Lemma~\ref{lem:two-scale-selberg} for the forms
\eqref{eq:prime-majorant-pulled-forms}, the same Chinese remainder and lattice
point calculations apply.  The resulting local factors agree with
\eqref{eq:local-factor-def} away from primes dividing $q$; the primes
dividing $q$ are treated below.  The coefficient vectors of the underlying forms
$\Psi_j$ are unchanged, the divisor moduli are still
$N^{O_K(\gamma)}$, and the shortest parameter interval has length at least
$H_P/L\gg N_P/(BL)=N^{1-o(1)}$.  Consequently, the normalised boundary
contribution
is
\[
    \ll_K N^{O_K(\gamma)}L^{O_K(1)}\frac{B}{N_P}
    =N^{-c(K)}
\]
after decreasing $\gamma(K)$ if necessary; here
\[
    \frac{B}{N_P}\leq\frac{QB}{N}
    \leq\frac{\exp(w^{1/2})}{N}=N^{-1+o(1)}
\]
by \eqref{eq:prime-majorant-pullback-range}.

For the Euler product, a prime $p>w$ is good provided
$p\nmid q\mathfrak d(\Psi)$.  At every such prime the local factor is
$1+O_K(p^{-2})$.  At a prime dividing $\mathfrak d(\Psi)$ we use the bound
$1+O_K(p^{-1})$.  If $p>w$ divides $q$, the assumption $(q,b_P)=1$ ensures that none
of the forms in~\eqref{eq:prime-majorant-pulled-forms} vanishes modulo $p$;
the local factor is then $(p/(p-1))^{|S|}=1+O_K(p^{-1})$.  By
\eqref{eq:disc-size-bound} and
\eqref{eq:prime-majorant-pullback-range},
\begin{align}\label{eq:prime-majorant-pullback-bad-primes}
    \sum_{\substack{p>w\\p\mid q\mathfrak d(\Psi)}}\frac1p
    \leq\frac{\log(q\mathfrak d(\Psi))}{w\log w}
    \ll_K\frac{\log(qB)}{w\log w}
    \ll_K\frac1{w^{1/2}\log w}.
\end{align}
The analytic error in the GPY divisor sum estimate is controlled
after pullback by
\begin{align}\label{eq:prime-majorant-pullback-analytic-primes}
    \sum_{\substack{p>w\\p\mid q\mathfrak d(\Psi)}}p^{-1/2}
    &\leq
    \frac{\log(q\mathfrak d(\Psi))}{\sqrt w\log w}\notag\\
    &\ll_K
    \frac{\log(qB)}{\sqrt w\log w}
    \ll_K\frac1{\log w}.
\end{align}
Thus both the Euler product and this analytic error are uniform after
pullback, and the Euler product is
$1+O_K((\log N)^{-c_3(K)})$, after decreasing $c_3(K)$ if necessary.

It remains to control the exceptional set and the truncation.  Since
$n\mapsto qn+a$ is injective, the preimage in $[N_P]$ of $\mathcal E$ has
cardinality at most
$N^{1/2+o(1)}=N_P^{1/2+o(1)}$.  Its proportion on every admissible
interval at scale $N_P$ is therefore $N_P^{-1/2+o(1)}$.  The divisor moment and
tail estimates of Lemma~\ref{lem:divisor-tail-forms}, applied with $q_0=q$,
then bound both the exceptional set terms and the truncation error by
$O_K((\log N)^{-c_4(K)})$.  Together with the boundary and Euler product
estimates above, this proves the asserted linear forms condition.
The pointwise bound follows from~\eqref{eq:prime-majorant-sup} and
\eqref{eq:prime-majorant-pullback}, and
\eqref{eq:prime-majorant-pullback-majorises} follows from
\eqref{eq:prime-majorant-majorises} evaluated at $qn+a$.
\end{proof}

\begin{remark}\label{rmk:prime-majorant-range}
The upper bound $B\leq\exp(w^{1/2})$ has two roles: it gives the
budget for bad primes in Lemma~\ref{lem:bad-prime-budget}, and it ensures that the
parameter intervals at scale $H$ are $N^{1-o(1)}$, as needed for the boundary
estimates in the GPY expansion.  The first role can be replaced by the following explicit
condition on the systems of forms to be tested:
$$
    \sum_{\substack{p>w\\p\mid\mathfrak d(\Psi)}}\frac1p\leq(\log N)^{-c_K}.
$$
Without such a condition, or without making $w$ much larger, the fully uniform
claim for all coefficients $|c_{j,t}|\leq KB$ is false for the local reason
explained at the start of this section. It may however be possible to state a weaker pseudorandomness assumption that applies to almost all forms rather than all forms, and this may allow proving Theorem~\ref{thm:prime-main} in a larger range of $B$, matching the range of Theorem~\ref{thm:main}. However, due to the additional technicalities involved, we do not pursue this here. 
\end{remark}

\section{Density increment}\label{sec:density-increment}

In this section we isolate the density increment mechanism used to pass from a
uniform counting statement to an existence theorem.  The point of the formulation
below is that the iteration is independent of the particular source of the
uniform counting statement: in the bounded setting the input is
Theorem~\ref{thm:1+}, while in the prime setting the inputs are
Theorem~\ref{thm:relative-u1plus} and Proposition~\ref{prop:prime-majorant}.

If
$F\colon[N]\to\R_{\geq0}$ and $P=\{qn+a:1\leq n\leq N'\}\subseteq[N]$ is an
arithmetic progression, we write
$$
    F_P(n)\coloneqq F(qn+a)\quad\text{for }1\leq n\leq N' 
$$
for the affine pullback of $F$ to $P$.

\subsection{Uniform counting inputs}

We first record the bounded uniform counting input.  This is the part of the
argument where Theorem~\ref{thm:1+} is used.

\begin{lemma}[Count under uniformity]\label{lem:count}
Let $k\geq 3$, let $c_k>0$ be small enough in terms of $k$, let
$C_k=C_k(k)$ be sufficiently large, and suppose
$(\log N)^{C_k}\le B\le N(\log N)^{-C_k}$ and
$H=\lfloor N/(2B)\rfloor$. Let
$A\subseteq[N]$ have density $\delta_A=|A|/|[N]|\ge\tfrac12(\log N)^{-c_k}$ and satisfy
\begin{align}\label{eq:main-count-uniform-hyp}
    \|1_A-\delta_A1_{[N]}\|_{U^{1+}[N]} \le \xi
\end{align}
for some
\begin{align}\label{eq:cnt-xi}
    0<\xi \le \tfrac12\bigl(c_1\,\delta_A^{\,k}(\log N)^{-c_k}\bigr)^{C_k'},
\end{align}
where $c_1=c_1(k)>0$ is a sufficiently small constant and
$C_k'=C_k'(k)>0$ is sufficiently large.  Then, for all but an
$O_k((\log N)^{-c_k})$ proportion of tuples
$\mathbf{b}\in(\{B/2<b_i\le B\}\cap\mathbb{Z})^{k}$,
\begin{align}\label{eq:main-count-normalised}
    r_H(\mathbf{b};1_A,1_A,\ldots,1_A)
     \gg_k \left(\delta_A \right)^k \frac{N^2}{B}.
\end{align}
\end{lemma}

\begin{proof}
Write $g=1_A-\delta_A1_{[N]}$, so $1_A=\delta_A1_{[N]}+g$ with
$\|g\|_{U^{1+}[N]}\le\xi$ by~\eqref{eq:main-count-uniform-hyp}. Expanding every
slot of $r_H$,
\begin{align}\label{eq:cnt-expand}
    r_H(\mathbf b;1_A,\ldots,1_A)
    =\sum_{\eps\in\{0,1\}^k}\delta_A^{\,k-|\eps|}\,
       r_H(\mathbf b;h_1^{\eps_1},\ldots,h_k^{\eps_{k}}),
    \qquad h_j^0=1_{[N]},\ h_j^1=g,
\end{align}
where $|\eps|=\eps_1+\cdots+\eps_{k}$. The term $\eps=\mathbf0$ is the main term:
since $|b_i|\le B$, for each $m\le H/4$ the $k$ points $x+b_im$ lie in $[N]$ whenever
$Bm<x\le N-Bm$, a range of length $\ge N/2$; this leaves
$\gg_k NH\gg N^2/B$ pairs
$(x,m)$, so there is $c_0=c_0(k)>0$ with
\begin{align}\label{eq:cnt-main}
    \delta_A^{\,k}\,r_H(\mathbf b;1_{[N]},\ldots,1_{[N]})
     \ge c_0\,\delta_A^{\,k}\,\frac{N^2}{B}.
\end{align}
Call $\mathbf b$ \emph{deficient} if
$r_H(\mathbf b;1_A,\ldots,1_A)<\tfrac12 c_0\delta_A^k N^2/B$. By
\eqref{eq:cnt-expand}--\eqref{eq:cnt-main}, a deficient $\mathbf b$ has a mixed
pattern $\eps\ne\mathbf0$ with
\begin{align}\label{eq:cnt-mixed}
    \bigl|r_H(\mathbf b;h_1^{\eps_1},\ldots,h_k^{\eps_{k}})\bigr|
    >\frac{c_0}{2^{k+1}}\,\delta_A^{\,k}\,\frac{N^2}{B}.
\end{align}
Fix a mixed pattern $\eps^*$ realised by a $\ge2^{-k}$ fraction of the deficient
tuples, set $F_j=h_j^{\eps^*_j}\in\{1_{[N]},g\}$, let $\mathcal G$ be that set
of tuples, and take
$\lambda(\mathbf b)=\sgn\overline{r_H(\mathbf b;F_1,\ldots,F_k)}$ on
$\mathcal G$ and $0$ elsewhere.  Then $\lambda$ is $1$-bounded and supported on
$([-B,B]\cap\Z)^k$, and by~\eqref{eq:R-def} and~\eqref{eq:cnt-mixed},
$$
    |R_H(\lambda;F_1,\ldots,F_k)|
    > |\mathcal G|\cdot\frac{c_0}{2^{k+1}}\,\delta_A^{\,k}\,\frac{N^2}{B}.
$$
Set $\eta_0=c_5\delta_A^k(\log N)^{-c_k}$ with $c_5=c_5(k)>0$ sufficiently
small. The range hypothesis of Theorem~\ref{thm:1+} holds with $\eta_0$ in
place of $\delta$ and with coefficient scale $N/H\asymp B$, by the assumed
range of $B$ and $C_k$ large.
If $|\mathcal G|>C(\log N)^{-c_k}B^k$ for a sufficiently large constant
$C=C(k)$, then the
preceding display gives
$$
    |R_H(\lambda;F_1,\ldots,F_k)|
    \ge \eta_0 \left(\frac NH\right)^k\frac{N^2}{N/H}.
$$
Theorem~\ref{thm:1+}
would then force $\min_j\|F_j\|_{U^{1+}[N]}\gg_k\eta_0^{C_k'}$.  Since some
$F_j$ is equal to $g$, this contradicts~\eqref{eq:cnt-xi} after making $c_1$
sufficiently small. Thus $|\mathcal G|\ll_k(\log N)^{-c_k}B^k$.  Summing over
the $<2^k$ mixed patterns gives the claimed exceptional proportion, and every
tuple that is not deficient satisfies~\eqref{eq:main-count-normalised}.
\end{proof}

We shall also need the corresponding relative uniform counting input.  The proof
is the same expansion as above, with Theorem~\ref{thm:relative-u1plus} replacing
Theorem~\ref{thm:1+}.  We include the details because this is the form used in
the proof of Theorem~\ref{thm:prime-main}.

\begin{lemma}[Relative count under uniformity]\label{lem:relative-count-uniformity}
Let $k\geq3$.  There is a constant $C=C(k)$ with the following property.  Let
$0<\varepsilon,\alpha<1/2$, let $H=\lfloor N/(2B)\rfloor$, and let
$F\colon[N]\to\R_{\geq0}$ have mean $\E_{n\in[N]}F(n)=\alpha$.  Let
$\nu\colon[N]\to\R_{\geq0}$ satisfy
$F\leq \nu$.  Suppose that, with
$$
    \delta_0=c\,\alpha^k\varepsilon
$$
for a sufficiently small constant $c=c(k)>0$, the majorant
$\nu_+=(\nu+1_{[N]})/2$ satisfies, at threshold $\delta_0$, the hypotheses on
$\nu$ in Theorem~\ref{thm:relative-u1plus} for the choice $C_0=1$, and that
\begin{align}\label{eq:relative-uniformity-small}
    \|F-\alpha1_{[N]}\|_{U^{1+}[N]}
    \leq c\,\delta_0^C .
\end{align}
Then, for all but an $O_k(\varepsilon)$ proportion of tuples
$\mathbf b\in((B/2,B]\cap\Z)^k$, one has
\begin{align}\label{eq:relative-uniform-count}
    r_H(\mathbf b;F,F,\ldots,F)
     \gg_k \alpha^k\frac{N^2}{B}.
\end{align}
\end{lemma}

\begin{proof}
Set $g=F-\alpha1_{[N]}$.  Expanding
$F=\alpha1_{[N]}+g$ in each slot, the term in which every slot is constant is again
$\gg_k \alpha^kN^2/B$ for every $\mathbf b$ in the coefficient box, by the same
argument as in~\eqref{eq:cnt-main}.  Let $c_0=c_0(k)>0$ be the constant in
this lower bound for the constant term.  If $\mathbf b$ is deficient, meaning that
$r_H(\mathbf b;F,\ldots,F)<c_0\alpha^kN^2/(2B)$, then some mixed pattern
$\eps\neq\mathbf0$ satisfies
$$
    |r_H(\mathbf b;G_1,\ldots,G_k)|
    \gg_k \alpha^k\frac{N^2}{B},
    \qquad G_j\in\{1_{[N]},g\},
$$
with at least one $G_j=g$.  If this happens for more than
$C\varepsilon B^k$ tuples for a suitable $C=C(k)$, then after pigeonholing the
mixed pattern we obtain a set $\mathcal G$ of deficient tuples with the same
mixed pattern.  Define
\[
    \lambda(\mathbf b)
    =\sgn\overline{r_H(\mathbf b;G_1,\ldots,G_k)}
    \quad\text{for }\mathbf b\in\mathcal G,
\]
and set $\lambda(\mathbf b)=0$ otherwise.  Then
$$
    |R_H(\lambda;G_1,\ldots,G_k)|
    \geq \delta_0\left(\frac NH\right)^k\frac{N^2}{N/H}.
$$
Since $F,\alpha1_{[N]}\geq0$ we have
$|g|=|F-\alpha1_{[N]}|\leq F+\alpha1_{[N]}$, so that $|1_{[N]}|/2\leq\nu_+$ and
$|g|/2\leq (F+\alpha1_{[N]})/2\leq \nu_+$ (using $F\leq\nu$ and $\alpha\leq1$).
Theorem~\ref{thm:relative-u1plus}, with $C_0=1$ and at coefficient scale
$N/H\asymp B$, applied to the functions $G_j/2$ and the majorant $\nu_+$,
gives
$$
    \min_j\|G_j\|_{U^{1+}[N]}
    \gg_k \delta_0^{C}.
$$
This contradicts~\eqref{eq:relative-uniformity-small}, since one of the $G_j$ is
$g$, provided the constant $c=c(k)$ in~\eqref{eq:relative-uniformity-small} is
chosen small enough.  Hence the deficient tuples have density $O_k(\varepsilon)$,
and the remaining tuples satisfy~\eqref{eq:relative-uniform-count}.
\end{proof}

\subsection{The density increment reduction}

The following elementary lemma supplies the weighted increment used at every
stage of the density increment iteration in Lemma~\ref{lem:di-reduction}.  The
bounded case is obtained by taking $R=1$.

\begin{lemma}[Weighted density increment]\label{lem:densityincrement}
Let $F\colon[N]\to[0,R]$ be nonnegative, let
$n=|[N]|=\lfloor N\rfloor$ and $\alpha=\E_{m\in[N]}F(m)$, and let
$\rho\in(0,1]$ with $\rho n\geq2R$.  If
$$
    \|F-\alpha1_{[N]}\|_{U^{1+}[N]} > \rho,
$$
then there is an arithmetic progression $P'=\{qm+a:1\leq m\leq N'\}\subseteq[N]$
of common difference
$$
    q\leq 2R\rho^{-1}
$$
and length
$$
    N'\geq \frac{1}{12}\frac{\rho^2}{R^2}n
$$
such that the pullback $F'=F_{P'}$ satisfies
\begin{align}\label{eq:di-increment}
    \E_{m\in[N']}F'(m) \geq \alpha+\tfrac1{12}\rho .
\end{align}
\end{lemma}

\begin{proof}
Let $g=F-\alpha1_{[N]}$, so $\sum_{m\in[N]}g(m)=0$.  By the definition of the $U^{1+}$ norm (cf. 
\eqref{eq:U1-plus-def}) there is an arithmetic progression $P\subseteq[N]$ of
common difference $q$ and length $L$ with
$|\sum_{m\in P}g(m)|>\rho n$.  Since $|g|\leq R$ and $\rho n\geq2R$, we have
$L>\rho n/R$ and hence
$$
    q\leq \frac{n}{L-1}\leq 2R\rho^{-1}.
$$
If $\sum_{m\in P}g(m)>0$, then
$$
    \E_{m\in P}F(m)\geq \alpha+\frac{\rho n}{L}\geq \alpha+\rho,
$$
and the claim follows with $P'=P$; in this case $N'=L>\rho n/R$, which is
stronger than the stated lower bound.

It remains to consider the case $\sum_{m\in P}g(m)<0$.  Partition $[N]$ into the
residue classes modulo $q$, with the class containing $P$ split into $P$ and the
possibly empty parts before and after $P$.  This gives $J\leq q+2\leq3q$
arithmetic progressions $T_j$ of common difference $q$ which partition $[N]$.
Since the total sum of $g$ is zero,
$$
    \sum_j\left(\sum_{m\in T_j}g(m)\right)_+
    =\frac12\sum_j\left|\sum_{m\in T_j}g(m)\right|
    \geq \frac12\left|\sum_{m\in P}g(m)\right|
    >\frac12\rho n.
$$
Thus for some $T_{j_0}$,
$$
    \sum_{m\in T_{j_0}}g(m)>\frac{\rho n}{2(q+2)}\geq \frac{\rho n}{6q}.
$$
Let $P'=T_{j_0}$ and $N'=|P'|$.  Since $g\leq F\leq R$ on $P'$, the last display
implies
$$
    N'\geq \frac{1}{R}\sum_{m\in P'}g(m)>\frac{\rho n}{6qR}
    \geq \frac{1}{12}\frac{\rho^2}{R^2}n.
$$
Also $N'\leq n/q+1\leq2n/q$, and hence
$$
    \E_{m\in P'}F(m)-\alpha
    =\frac{1}{N'}\sum_{m\in P'}g(m)
    \geq \frac{\rho n/(6q)}{2n/q}
    =\frac{\rho}{12}.
$$
This proves the lemma.
\end{proof}

We now iterate the density increment to construct for any sequence a not too sparse affine pullback that behaves pseudorandomly in the $U^{1+}$ norm. 

\begin{lemma}[Density increment reduction]\label{lem:di-reduction}
Let $k\geq1$, $B\geq1$, and $R\geq1$.  Let $F\colon[N]\to[0,R]$ be
nonnegative, and let
$n=|[N]|=\lfloor N\rfloor$ and $\alpha_0=\E_{m\in[N]}F(m)$.  Let
$0<\rho\leq1$ and let $A_*\geq\alpha_0$.
Set
\begin{align}\label{eq:di-Qstar}
    M_*\coloneqq \left\lceil 12A_*\rho^{-1}\right\rceil,
    \qquad
    Q_*\coloneqq \left(12R^2\rho^{-2}\right)^{M_*}.
\end{align}
Assume that $\rho n/Q_*\geq2R$, and that every affine pullback $F_P$ with
$P\subseteq[N]$ an arithmetic progression of length at least $n/Q_*$ has mean at
most $A_*$.  Then there is an arithmetic progression
$P_{\mathrm{fin}}=\{qt+a:1\leq t\leq N_{\mathrm{fin}}\}\subseteq[N]$ with
\begin{align}\label{eq:di-reduction-compression}
    N_{\mathrm{fin}}\geq n/Q_*,\qquad q\leq 2n/N_{\mathrm{fin}}\leq 2Q_*,
\end{align}
such that, writing $F_{\mathrm{fin}}=(F)_{P_{\mathrm{fin}}}$ and
$\alpha_{\mathrm{fin}}=\E_{m\in[N_{\mathrm{fin}}]}F_{\mathrm{fin}}(m)$, one has
\begin{align}\label{eq:di-terminal-uniform}
    \alpha_{\mathrm{fin}}\geq\alpha_0,
    \qquad
    \|F_{\mathrm{fin}}-\alpha_{\mathrm{fin}}1_{[N_{\mathrm{fin}}]}\|_{U^{1+}[N_{\mathrm{fin}}]}\leq\rho .
\end{align}
Moreover, if for some coefficient tuple $\mathbf b$ there are
$x'\in\Z$ and $m'\in[N_{\mathrm{fin}}/(2B)]$ such that
$$
    \prod_{i=1}^k F_{\mathrm{fin}}(x'+b_im')>0,
$$
then there are $x\in\Z$ and $m\in[N/B]$ such that
$$
    \prod_{i=1}^k F(x+b_im)>0.
$$
\end{lemma}

\begin{proof}
Start with $F_0=F$, $N_0=n$, $q^{(0)}=1$, and $a^{(0)}=0$.  Suppose that after
$i$ steps we have a pullback $F_i$ to an arithmetic progression
$P_i=\{q^{(i)}t+a^{(i)}:1\leq t\leq N_i\}\subseteq[N]$, with mean
$\alpha_i$, length $N_i\geq n/Q_*$ and compression $n/N_i\leq Q_*$.  If
$\|F_i-\alpha_i1_{[N_i]}\|_{U^{1+}[N_i]}\leq\rho$, we stop.  Otherwise,
Lemma~\ref{lem:densityincrement}, applied to $F_i$, gives a further progression
inside $[N_i]$ of common difference $q_i\leq2R\rho^{-1}$ and length
$N_{i+1}\geq(\rho^2/(12R^2))N_i$ on which the mean increases by at least
$\rho/12$.  Composing the affine maps gives the next pullback $F_{i+1}$.

Each nonterminal step raises the mean by at least $\rho/12$, while every
admissible pullback has mean at most $A_*$.  Since
$M_*=\lceil12A_*/\rho\rceil$, a further nonterminal step after step $M_*$
would force the mean above $A_*$.  Hence the process stops after at most
$M_*$ steps.  Along the way,
$$
    N_i\geq (12R^2\rho^{-2})^{-i}n\geq n/Q_*,
    \qquad
    q^{(i)}\leq 2n/N_i\leq 2Q_*,
$$
where the second inequality follows from
$q^{(i)}(N_i-1)\leq n-1$ and $N_i\geq2$.  This justifies each application of Lemma~\ref{lem:densityincrement}; the
condition $\rho n/Q_*\geq2R$ guarantees $\rho N_i\geq2R$.  Denote the
terminal pullback by $F_{\mathrm{fin}}$ on $P_{\mathrm{fin}}$.  It therefore
satisfies~\eqref{eq:di-reduction-compression} and~\eqref{eq:di-terminal-uniform}.

Finally, if $x'+b_im'$ lies in the support of $F_{\mathrm{fin}}$ for every $i$, then
under the affine map defining $P_{\mathrm{fin}}$ these points lift to
$$
    x+b_i m\coloneqq q(x'+b_i m')+a
    =(qx'+a)+b_i(qm') .
$$
Since $q\leq2n/N_{\mathrm{fin}}$, one has
$m=qm'\leq qN_{\mathrm{fin}}/(2B)\leq n/B\leq N/B$, proving the lifting claim.
\end{proof}

\subsection{Proof of Theorem~\ref{thm:main}}

\begin{proof}[Proof of Theorem~\ref{thm:main}]
Put $n=|[N]|=\lfloor N\rfloor$.  Let $\delta_A=|A|/n$, and let $C_k'$ be as in
Lemma~\ref{lem:count}.  We choose
$$
    \rho=c_2(\log N)^{-\theta},
    \qquad
    \theta\coloneqq c_kC_k'(k+1),
$$
where $c_2=c_2(k)>0$ is sufficiently small and $c_k$ is chosen small enough in
terms of $k$ that $\theta<c_k'$, where $c_k'$ is the exponent governing the
upper range in Theorem~\ref{thm:main}.  Apply Lemma~\ref{lem:di-reduction} to $F=1_A$, with
$R=A_*=1$.  Then
$$
    \log Q_*\ll_k (\log N)^\theta\log\log N,
    \qquad \rho n/Q_*=N^{1-o(1)}\geq2.
$$
Since $\theta<c_k'$, for large $N$ we have
$$
    C_k\log\log N+\log Q_*
    \ll_k(\log N)^\theta\log\log N
    \leq\tfrac12(\log N)^{c_k'}.
$$
Together with the hypothesis $B\leq N\exp(-(\log N)^{c_k'})$, this gives
\begin{align}\label{eq:di-terminal-range-main}
    2B(\log N)^{C_k}Q_*\leq n.
\end{align}
Write the terminal pullback supplied by Lemma~\ref{lem:di-reduction} as
$F_M\colon[N_M]\to\R_{\geq0}$.  Then $N_M$ satisfies
$(\log N_M)^{C_k}\leq B\leq N_M(\log N_M)^{-C_k}$ and
$\log N_M\asymp\log N$.

Let $A_M$ be the support of the terminal pullback $F_M$, and let
$\delta_M=|A_M|/|[N_M]|$.  By construction,
$\|1_{A_M}-\delta_M1_{[N_M]}\|_{U^{1+}[N_M]}\leq\rho$ and
$\delta_M\geq\delta_A$.  The choice of $\rho$ ensures
$$
    \rho\leq\tfrac12\bigl(c_1\delta_M^k(\log N_M)^{-c_k}\bigr)^{C_k'}.
$$
Since $N_M=N^{1-o(1)}$, the original density lower bound gives
$\delta_M\geq\delta_A\geq\tfrac12(\log N_M)^{-c_k}$ for all sufficiently
large $N$.  Lemma~\ref{lem:count}, applied at scale $N_M$, gives
$r_{\lfloor N_M/(2B)\rfloor}(\mathbf b;1_{A_M},\ldots,1_{A_M})>0$ for all but an
$O_k((\log N)^{-c_k})$ proportion of the coefficient tuples.  The lifting part of
Lemma~\ref{lem:di-reduction} then gives a configuration in $A$ with
$m\in[N/B]$ for the same set of good tuples.  The tuples with a repeated entry
form an $O(1/B)$ proportion of the box, which is negligible compared with
$(\log N)^{-c_k}$, and for all remaining tuples the lifted configuration is
nontrivial.  This proves Theorem~\ref{thm:main}.
\end{proof}

\subsection{Proof of Theorem~\ref{thm:prime-main}}\label{sec:proof-prime-main}

\begin{proof}
We pass to a majorised model obtained using the $W$-trick, apply the density increment
reduction, and finish with the relative counting theorem.  Set
$\alpha=(\log N)^{-c_k}$.  We choose $c_k$ small enough, in terms of $k$, so
that all polynomial losses below are absorbed by powers of $\log N$.  Let
$K$ be as in Theorem~\ref{thm:relative-u1plus} with $C_0=1$, and let
$\gamma=c_0(K)$ and
$\gamma_1=c_1(K)$ be the exponents supplied by
Proposition~\ref{prop:prime-majorant}.  Set $w=(\log N)^\gamma$ and
$W=\prod_{p\leq w}p$.  Since $W=N^{o(1)}$ and
$\sum_{p\mid W}\log p=\log W=o(\alpha N)$, the primes dividing $W$
contribute $o(\alpha N)$ to $\sum_{p\in A}\log p$.  The prime number theorem gives
$\pi(N^{1/2})=o(|A|)$, so all but $o(|A|)$ elements of $A$ exceed $N^{1/2}$;
thus $\sum_{p\in A}\log p\gg |A|\log N\gg\alpha N$.  The pigeonhole
principle therefore gives an integer $b_0$ with $1\leq b_0\leq W$ and
$(b_0,W)=1$ such that, for
$N_0=\lfloor(N-b_0)/W\rfloor$,
\begin{align}\label{eq:prime-main-pigeonhole-proof}
    \E_{n\in[N_0]}
    \frac{\varphi(W)}{W}\Lambda(Wn+b_0)1_A(Wn+b_0)
    \gg \alpha .
\end{align}
Choose a sufficiently small constant $c_*=c_*(k)>0$ and set
$$
    F_0(n)=c_*\frac{\varphi(W)}{W}\Lambda(Wn+b_0)1_A(Wn+b_0)
    \quad\text{for }n\in[N_0].
$$
Then $\E_{n\in[N_0]}F_0(n)\gg\alpha$.

Let $\varepsilon=(\log N)^{-c_k}$ and choose
$$
    \rho=(\log N)^{-\theta_{\mathbb P}},
    \qquad \theta_{\mathbb P}=C_1c_k(k+1),
$$
where $C_1=C_1(k)$ is large enough for the relative uniform counting lemma and
$c_k$ is small enough that $\theta_{\mathbb P}<\gamma/10$.
After decreasing $c_k$, the range in Theorem~\ref{thm:prime-main} gives
$(\log N_0)^{C(K)}\leq2B\leq4B\leq
N_0(\log N_0)^{-C(K)}$ and
$2B\leq\exp(w^{1/2})$.  Moreover,
$2^{-1/\gamma}\log N\leq\log N_0\leq\log N$ for all sufficiently large $N$, so the fixed value
$w=(\log N)^\gamma$ lies between
$\tfrac12(\log N_0)^\gamma$ and $2(\log N_0)^\gamma$ for large $N$.
We may therefore apply Proposition~\ref{prop:prime-majorant} with ambient
length $N_0$, coefficient scale $2B$, and the same $w,W$ and residue
$b_0$.  It provides a
GPY majorant $\nu_0$ for $F_0$ with
$$
    0\leq F_0\leq \nu_0,
    \qquad
    \|\nu_0\|_\infty\leq R\coloneqq (\log N)^D
$$
for some $D=D(k)$.  Choose $c_*$ initially small enough in terms of the
implicit constant in~\eqref{eq:prime-majorant-majorises} that
$F_0\leq\nu_0/4$.  We apply Lemma~\ref{lem:di-reduction} to $F_0$ with
$A_*=1/2$.  For these parameters
$$
    \log Q_*\ll_k (\log N)^{\theta_{\mathbb P}}\log\log N,
    \qquad \rho N_0/Q_*=N^{1-o(1)}\geq2(\log N)^D=2R.
$$
After decreasing $c_k$ if necessary, all affine pullbacks of compression at most
$Q_*$ have length at least $N_0/Q_*=N^{1-o(1)}$, and
\[
 \log(8Q_*B)\ll_k(\log N)^{\theta_{\mathbb P}}\log\log N+(\log N)^{c_k}
 \leq(\log N)^{\gamma/2}=w^{1/2}.
\]
Thus
\begin{align}\label{eq:prime-QB-budget}
    8Q_*B\leq \exp(w^{1/2}).
\end{align}
Consider an arithmetic progression
\begin{align*}
    P=\{qn+a:1\leq n\leq N_P\}\subseteq[N_0],
    \qquad N_P\geq N_0/Q_*,
\end{align*}
and set $b_P=Wa+b_0$.  Then $q\leq2Q_*$ and $(b_P,W)=1$.  

If
$(q,b_P)>1$, choose a prime $p\mid(q,b_P)$.  Every integer
\begin{align*}
    W(qn+a)+b_0=Wqn+b_P
\end{align*}
is divisible by $p$, and hence it can be prime only when it equals $p$.
Thus $(F_0)_P$ is supported at most one point, and
\begin{align}\label{eq:nonprimitive-pullback-mean}
    \E_{n\in[N_P]}(F_0)_P(n)
    \ll\frac{\log N}{N_P}
    \ll\frac{Q_*\log N}{N_0}
    =N^{-1+o(1)}\ll_k N^{-1/2}.
\end{align}

Suppose now that $(q,b_P)=1$, so that $(Wq,b_P)=1$, and define
\begin{align*}
    \nu_P(n)=1_{[N_P]}(n)\nu_0(qn+a).
\end{align*}
Then $(F_0)_P\leq\nu_P$ and $\|\nu_P\|_\infty\leq(\log N)^D$.
Put
\[
    H_P=\left\lfloor\frac{N_P}{2B}\right\rfloor.
\]
For sufficiently large $N$, one has $H_P\geq1$.  Set $B_P=N_P/H_P$; then
$2B\leq B_P\leq4B$.  The affine pullback assertion of
Proposition~\ref{prop:prime-majorant} applies at scale $(N_P,B_P,H_P)$ with
$Q=2Q_*$, since
$2Q_*B_P\leq8Q_*B\leq\exp(w^{1/2})$.  Thus, for all mesh parameters
$2\leq L\leq(\log N)^{\gamma}$, the function $\nu_P$ satisfies the
$(K,L,(\log N)^{-\gamma_1})$ linear forms condition at scale
$(N_P,B_P,H_P)$.  Its single form instance gives
\begin{align*}
    \E_{n\in[N_P]}\nu_P(n)
    =1+O((\log N)^{-\gamma_1}).
\end{align*}

It follows that every pullback of length at least $N_0/Q_*$ has mean at most
$1/2$: this follows from~\eqref{eq:nonprimitive-pullback-mean} in the
nonprimitive case and, in the primitive case, from $F_0\leq\nu_0/4$ and the
single form estimate.  In particular $\E_{n\in[N_0]}F_0(n)<1/2$.
Hence the hypotheses of Lemma~\ref{lem:di-reduction} hold.  Moreover, every
pullback produced by the iteration is primitive, since its mean is at least
$\E_{n\in[N_0]}F_0(n)\gg\alpha$, whereas a nonprimitive pullback has mean
$O_k(N^{-1/2})<\tfrac12\E_{n\in[N_0]}F_0(n)$ for all sufficiently large $N$.
We obtain a terminal pullback
$F_M\colon[N_M]\to\R_{\geq0}$, with mean $\alpha_M\geq c\alpha$, such that
$$
    \|F_M-\alpha_M1_{[N_M]}\|_{U^{1+}[N_M]}\leq\rho,
    \qquad
    N_M\geq N_0/Q_*.
$$
The terminal progression is primitive by the preceding paragraph.  Set
$\delta_0=c\alpha_M^k\varepsilon$, and let $C$ be the constant in
Theorem~\ref{thm:relative-u1plus}.  Since
$\alpha_M\gg\alpha=(\log N)^{-c_k}$, decreasing $c_k$ in terms of $k$ gives
\begin{align*}
    L_M\coloneqq \delta_0^{-C}\leq(\log N)^\gamma,
    \qquad
    (\log N)^{-\gamma_1}\leq\delta_0^C.
\end{align*}
Set
\[
    H_M=\left\lfloor\frac{N_M}{2B}\right\rfloor.
\]
For sufficiently large $N$, one has $H_M\geq1$.  Set $B_M=N_M/H_M$; then
$2B\leq B_M\leq4B$.  The corresponding GPY majorant $\nu_M$ after
pullback therefore satisfies the
linear forms hypothesis of Theorem~\ref{thm:relative-u1plus} with mesh
parameter $L_M$, by the same application of
Proposition~\ref{prop:prime-majorant} at scale $(N_M,B_M,H_M)$.  Define
\begin{align*}
    \nu_{M,+}=\frac{\nu_M+1_{[N_M]}}2.
\end{align*}
Expanding every occurrence of $\nu_{M,+}$ shows that it satisfies the same
linear forms condition with the same error.  It satisfies the remaining
hypotheses at scale $(N_M,B_M,H_M)$: the range follows
from \eqref{eq:prime-main-B-range}, \eqref{eq:prime-QB-budget}, and
$N_M=N^{1-o(1)}$, while the pointwise diagonal term
$R_{\nu_{M,+}}^C/B_M\leq R_{\nu_{M,+}}^C/(2B)
\leq(\log N)^{O_k(1)}/B$
is at most $\delta_0^{C}$.  Here $\delta_0$ is the threshold at which
Lemma~\ref{lem:relative-count-uniformity} invokes
Theorem~\ref{thm:relative-u1plus}; this holds once the lower exponent $1/c_k$ in
\eqref{eq:prime-main-B-range} is taken large in terms of $C$, $D$ and $k$.

Since $\rho$ was chosen smaller than a suitable
power of $\alpha_M^k\varepsilon$, and $\alpha_M<1/2$ by the preceding
pullback bound, Lemma~\ref{lem:relative-count-uniformity} applied to $F_M$ gives
$$
    r_{H_M}(\mathbf b;F_M,\ldots,F_M)>0
$$
for all but an $O_k((\log N)^{-c_k})$ proportion of tuples
$\mathbf b\in((B/2,B]\cap\Z)^k$.

For each good tuple, the positivity of this weighted count gives a configuration
in the support of $F_M$.  By the lifting part of Lemma~\ref{lem:di-reduction},
the support of $F_0$ contains a configuration
$$
    n_0+b_1m_0,\ n_0+b_2m_0,\ \ldots,\ n_0+b_km_0
$$
with $m_0\leq N_0/B$.  Returning to the original primes, all the numbers
$$
    W(n_0+b_im_0)+b_0\quad\text{for }1\leq i\leq k
$$
lie in $A$, and the multiplier in the original variable is
$Wm_0\leq WN_0/B\leq N/B$.  As in the proof of Theorem~\ref{thm:main},
the tuples with a repeated entry form an $O(1/B)$ proportion, negligible compared
with $(\log N)^{-c_k}$, and for all remaining tuples the configuration is
nontrivial.  This proves Theorem~\ref{thm:prime-main}.
\end{proof}

\begingroup
\sloppy
\renewcommand{\texttt}[1]{\url{#1}}
\bibliography{refs}
\bibliographystyle{plain}
\endgroup

\end{document}